\documentclass[12pt,centertags]{amsart}  

\usepackage{amsmath, amssymb, amsthm}

\usepackage[hidelinks]{hyperref}

\swapnumbers

\usepackage{geometry}  
\usepackage{graphicx}
\newcommand{\pindep}{\mathop{\,\rotatebox[origin=c]{90}{$\models$}\,}}

\makeatletter
\def\indsym#1#2{%
  \setbox0=\hbox{$\m@th#1x$}%
  \kern\wd0%
  \hbox to 0pt{\hss$\m@th#1\mid$\hbox to 0pt{$\m@th#1^{#2}$\hss}\hss}%
  \lower.9\ht0\hbox to 0pt{\hss$\m@th#1\smile$\hss}%
  \kern\wd0} \newcommand{\ind}[1][]{\mathop{\mathpalette\indsym{#1}}}
\def\nindsym#1#2{%
  \setbox0=\hbox{$\m@th#1x$}%
  \kern\wd0%
  \hbox to 0pt{\mathchardef\nn="3236\hss$\m@th#1\nn$\kern1.4\wd0\hss}
  \hbox to 0pt{\hss$\m@th#1\mid$\hbox to 0pt{$\m@th#1^{#2}$\hss}\hss}%
  \lower.9\ht0\hbox to 0pt{\hss$\m@th#1\smile$\hss}%
  \kern\wd0}
\newcommand{\nind}[1][]{\mathop{\mathpalette\nindsym{#1}}}
\makeatother

\def\indep{\ind}
\def\nindep{\nind}

\def\E{\mathbb{E}}        \def\P{\mathbb{P}}       \def\U{\mathcal{U}}
\def\B{\mathcal{B}}              \def\M{\mathcal{M}}
\def\NN{\mathcal{N}}    \def\SS{\mathcal{S}}    \def\K{\mathcal K}
    \def\A{\mathcal{A}}       \def\B{\mathcal{B}}
\def\D{\mathcal{D}}      \def\C{\mathcal{C}}       \def\F{\mathcal{F}}
\def\V{\mathcal{V}}       \def\EE{\mathcal{E}}     \def\ZZ{\mathcal{Z}}
\def\tensor{\otimes}                     \def\s{\sigma}

\renewcommand{\S}{\mathcal{S}}

\def\LLpr{L^{pr}}

\def\Sum{\sum}
\def\bra{\langle} 
\def\ket{\rangle}
\def\vphi{\varphi}
 
 \def\Q{\mathbb{Q}}
\def\R{\mathbb{R}} 

\def\N{\mathbb{N}}
\def\emp{\emptyset}

\def\F{\mathcal{F}}

\def\1{\mathbf{1}}
\def\0{\mathbf{0}}

\theoremstyle{plain}
\newtheorem{theo}{Theorem}[section]
\newtheorem{prop}[theo]{Proposition}
\newtheorem{lema}[theo]{Lemma}
\newtheorem*{lema*}{Lemma}
\newtheorem{coro}[theo]{Corollary}

\theoremstyle{definition}
\newtheorem{defi}[theo]{Definition}
\newtheorem{rema}[theo]{Remark}
\newtheorem{fact}[theo]{Fact}
\newtheorem{note}[theo]{Note}
\newtheorem{obse}[theo]{Observation}

\newtheorem{nota}[theo]{Notation}

\newtheorem{exam}[theo]{Example}

\newtheorem{exer}[theo]{Exercise}

\makeatletter

\def\dotminussym#1#2{%
  \setbox0=\hbox{$\m@th#1-$}%
  \kern.5\wd0%
  \hbox to 0pt{\hss\hbox{$\m@th#1-$}\hss}%
  \raise.3\ht0\hbox to 0pt{\hss$\m@th#1\cdot$\hss}%
  \kern.5\wd0}
\newcommand{\dotminus}{\mathbin{\mathpalette\dotminussym{}}}

\makeatother

\def\dotmin{\dotminus}

\DeclareMathOperator{\card}{card}
\DeclareMathOperator{\dcl}{dcl}
\DeclareMathOperator{\acl}{acl}

\DeclareMathOperator{\meq}{meq}

\DeclareMathOperator{\tp}{tp}

\DeclareMathOperator{\dist}{dist}

\DeclareMathOperator{\rharp}{{\upharpoonright}}
\DeclareMathOperator{\lharp}{{\upharpoonleft}}

\title[Model Theory of Probability Spaces]{Model Theory of Probability Spaces}

\author[Alexander Berenstein]{\noindent Alexander Berenstein}
\address{Alexander Berenstein; Universidad de los Andes,  \newline
\phantom{-------------------------------------}Cra 1 N${}^o$ 18A-12,  Edificio H, Bogot\'{a}, Colombia}
\urladdr{http://www.matematicas.uniandes.edu.co/\textasciitilde aberenst}

\author[C.\ Ward Henson]{C. Ward Henson}
\address{C. Ward Henson; University of Illinois, Urbana-Champaign; \newline 
\phantom{-------------------------------------}Urbana, Illinois 61801, USA}
\urladdr{https://faculty.math.illinois.edu/\textasciitilde henson}

\thanks{\vspace{-0.25cm}\newline The authors are grateful to Ita\"{\i} Ben Yaacov for helpful conversations.  
Research for this paper was partially supported by NSF grants and by grants from the 
Simons Foundation (202251 and 422088, to the second author).}

\begin{document}

\begin{abstract}
This expository paper treats the model theory of probability
spaces using the framework of continuous $[0,1]$-valued first order logic.
The metric structures discussed, which we call probability algebras, 
are obtained from probability spaces by identifying two
measurable sets if they differ by a set of measure zero. The class
of probability algebras is axiomatizable in continuous first order logic;
we denote its theory by $Pr$.  We show that the
existentially closed structures in this class are exactly the ones
in which the underlying probability space is atomless.  This subclass is also 
axiomatizable; its theory $APA$ is the model companion of $Pr$.
We show that $APA$ is separably categorical (hence complete), 
has quantifier elimination, is $\omega$-stable, and has built-in canonical bases,
and we give a natural characterization of its independence relation.  For 
general probability algebras, we prove that the set of atoms 
(enlarged by adding $0$) is a definable
set, uniformly in models of $Pr$.  We use this fact as
a basis for giving a complete treatment of the model theory of 
arbitrary probability spaces.  The core of this paper is an extensive
presentation of the main model theoretic properties of $APA$.
We discuss Maharam's
structure theorem for probability algebras, and indicate the close connections
between the ideas behind it and model theory.  
We show how probabilistic entropy provides
a rank connected to model theoretic forking in probability algebras.
In the final section we mention some open problems.
\end{abstract}


\maketitle

\section{Introduction}
\label{intro}

In this paper we use the continuous version of first order logic to
investigate probability spaces $(X,\B,\mu)$.\index{$(X,\B,\mu)$}  
Here $\B$ is a $\sigma$-algebra of subsets of $X$ (requiring
$\emptyset,X \in \B$) and $\mu$ 
is a $\sigma$-additive probability measure on $\B$.
There is a canonical pseudometric $d$ on $\B$, obtained
by taking the distance between sets to be given by  $d(A,B) := \mu(A \triangle B)$.
(Here $\triangle$ denotes the symmetric difference operation on sets.) 
This gives rise to a \emph{prestructure}
\[
(\B,0,1,\cdot^{c},\cap,\cup,\mu,d)
\]
(which we  often write as $(\B,\mu,d)$, regarding $\B$ 
as a boolean algebra but suppressing the constants and the
operations from our notation).
We obtain a  \emph{structure} in the usual way by
turning $d$ into a metric.  (Usually we would also need to take the metric completion, but the 
metric quotient is automatically complete here, as we indicate below.) 
This yields the structure $(\widehat \B, \widehat \mu, \widehat d)$,
where $\widehat \B$ is the quotient
of $\B$ by the equivalence relation $\mu(A \triangle B) = 0$ and $\widehat \mu, \widehat d$
are the canonical measure and metric induced on $\widehat \B$.  
That is, for each $A \in \B$ and 
$[A]_\mu := \{ B \in \B \mid \mu(A \triangle B) = 0 \} \in \widehat \B$,
we have $\widehat \mu([A]_\mu) = \mu(A)$; similarly
$\widehat d([A]_\mu,[B]_\mu) = d(A,B)$
($= \mu(A \triangle B) = \widehat \mu([A \triangle B]_\mu) = \widehat \mu([A]_\mu \triangle [B]_\mu)$).

One sees that $\widehat \B$ is a complete (in the sense of order) boolean algebra, $\widehat \mu$ is a strictly positive
$\sigma$-additive measure on $\widehat \B$, and $\widehat d$ is the metric
defined canonically from $\widehat \mu$.  The metric structures $(\widehat \B,\widehat \mu,\widehat d)$%
\index{$(\widehat \B,\widehat \mu,\widehat d)$}
are the principal objects of study in this paper.

We use standard background from measure
theory and analysis (which we summarize in Section \ref{probability spaces})
and from continuous first order logic.
The model theoretic background for this paper comes from \cite{BBHU}
and \cite{BU}, which present the $[0,1]$-valued continuous version
of first order logic.  In Section \ref{continuous model theory} 
we give references for some additional concepts
and tools from continuous logic that we need here.

The main content of this paper is in Sections \ref{model theory of prob spaces},
\ref{random variables}, 
 \ref{atomless prob  spaces}, and \ref{stability of APA}.
Sections \ref{model theory of prob spaces} and \ref{random variables}
present the model theory of arbitrary probability
spaces in the framework of continuous first order logic.  Our results show that everything model
theoretic about arbitrary probability spaces can be systematically
reduced to the atomless case, which is given a full treatment in Sections \ref{atomless prob  spaces}
and \ref{stability of APA}.

Atomless probability spaces were studied by Ben Yaacov
\cite{BY1} using the framework of compact abstract theories, with an
emphasis on issues around model theoretic stability. In
Sections \ref{atomless prob  spaces} and \ref{stability of APA}
we study atomless probability
spaces in the context of continuous logic and present
analogues of results from \cite{BY1}, as well as additional results
that are specific to the continuous first order setting.  
In Section \ref{atomless prob  spaces} we give axioms for the class of 
atomless probability algebras, and show that the theory
of these structures (denoted by $APA$) is well behaved
from the model theoretic point of view: in particular, it is
complete, has quantifier elimination, is separably categorical, and
is $\omega$-stable. We characterize (up to equivalence) the
induced metric on type spaces of
$APA$.  In Section \ref{stability of APA} we focus on features of
$APA$ that are connected to its stability.
Following the work of Ben Yaacov \cite{BY1},
we give an intrinsic characterization of the independence relation
of $APA$, and show that it has built-in canonical bases.  
We give a direct, elementary proof that $APA$ is strongly finitely based,
a fact originally proved in \cite{BBH} using lovely pairs of $APA$ models.
We also look at $APA$ from the point of view of Shelah's classification program,
and show that $APA$ is non-multidimensional but not unidimensional,  using
natural translations of those concepts into continuous model theory.

Section \ref{maharam} is 
devoted to Maharam's structure theorem for probability algebras
and its connections with model theory. 
In Section \ref{entropy} we show how model theoretic forking in 
probability algebras is related to probabilistic entropy.  
In the last section we identify some open problems that
seem worth investigating further.

\section{Probability spaces}
\label{probability spaces}
In this section we present basic information about probability spaces %
\index{probability space}
$(X,\B,\mu)$ and their measure algebra quotients.

We recall that $A\in \B$ is
\emph{atomless}\index{atomless!element}
if for every $B \in \B$ with $B \subseteq A$ and
$\mu(B)>0$, there are $B_1,B_2\in \B$ such that $B=B_1 \cup B_2$,
$B_1$ and $B_2$ are disjoint and $\mu(B_1)>0$, $\mu(B_2)>0$.
We say that $(X,\B,\mu)$ is \emph{atomless}\index{atomless!probability space} 
if $X$ is atomless in $\B$.

We recall that $A \in B$ is an \emph{atom} \index{atom}
if $\mu(A)>0$, and for every $B
\in \B$ with $B \subseteq A$ one has $\mu(B)=0$ or $\mu(A \setminus
B)=0$. Evidently, if $A_1,A_2 \in \B$ are atoms then either $\mu(A_1
\cap A_2)=0$ or $\mu(A_1 \triangle A_2) = 0$. Furthermore,
there exists a finite or countable family $\mathcal{A} \subseteq \B$
such that each $A \in \mathcal{A}$ is an atom and such that whenever
$A \in \B$ is an atom, there exists $A' \in \mathcal{A}$ such that
$\mu(A \triangle A')=0$.  The \emph{atomic part} of $X$ is the join (union)
of the sets in $\mathcal{A}$, and its complement is the \emph{atomless part} of $X$;
this partition of $X$ is well defined up to a set of measure $0$.
The atomic part is \emph{atomic}, \index{atomic element}
in the sense that whenever $A \in \B$ is
contained in the atomic part of $X$, then $A$ is (up to a set of measure $0$) the
union of the atoms it contains; equivalently, all atomless subsets
of such an $A$ have measure $0$.  Likewise, the atomless part of $X$ is
an atomless member of $\B$.  We regard $0$ as atomic, since it is contained
in the atomic part, and $0$ is atomless by definition.  
Further, if $A$ is atomless and $B$ is atomic, then $\mu(A \cap B) = 0$.

We say $(A_1,\dots,A_n)$ is a \emph{partition in}\index{partition}
 $\B$ if $A_i \cap A_j = 0$
whenever $i \neq j$.  If, in addition, $A_1 \cup \dots \cup A_n = B$, then we say
$(A_1,\dots,A_n)$ is a \emph{partition of $B$ in} $\B$.  Note that we allow
$A_i$ to be $0$ in such a situation.

We say that $A_1,A_2\in \B$ determine the same \emph{event},\index{event} 
and write $A_1 \sim_\mu A_2$ if the symmetric difference of the sets has
$\mu$-measure zero. Clearly $\sim_\mu$ is an equivalence relation. 
We denote the equivalence class of $A \in \B$ by $[A]_\mu$.  The
collection of equivalence classes of $\B$ modulo $\sim_\mu$ is
denoted by $\widehat \B$.
The operations of complement, union and intersection are well
defined for events and they make $\widehat \B$ a boolean algebra.
Moreover, $\mu$ induces on $\widehat \B$ a $\sigma$-additive,
strictly positive probability measure $\widehat \mu$.  
As noted above, we denote the canonical metric on $\widehat\B$ by $\widehat d$
and recall that it is defined by $\widehat d([A]_\mu,[B]_\mu) = d(A,B)$.
It is important in this paper that $(\widehat \B,\widehat d)$ is a complete metric space  
(see the calculation in \cite[Lemma 323F]{Fr-treatise}).

We refer to $(\widehat \B,\widehat \mu,\widehat d)$ as the \emph{probability algebra of}%
\index{probability algebra}  $(X,\B,\mu)$.  

\begin{nota}
\label{widehat C notation}\index{$C^\#$, $\bra S \ket$}
(a) If $C$ is a subset of a boolean algebra, we denote the boolean subalgebra generated by $C$ by $C^\#$.

Let $(X,\B,\mu)$ be a probability space and let $(\widehat \B,\widehat \mu,\widehat d)$ be
its probability algebra.  

(b) If $S$ is a subset of $\B$, we let $\bra S \ket$ denote the
$\sigma$-subalgebra of $\B$ generated by $S$.

(c) If $S$ is a subset of $\widehat \B$, we let $\bra S \ket$ denote the
$\widehat d$-closure of $S^\#$.  Note that $\bra S \ket$ is equal to the $\sigma$-subalgebra
of $\widehat \B$ generated by $S$, by \cite[Lemma 323F]{Fr-treatise}, and it is also equal to
the $\widehat d$-closed boolean subalgebra generated by $S$.  In other words, $\bra S \ket$
is $\widehat d$-closed and has $S^\#$ as a $\widehat d$-dense subset.

(d) Throughout this paper we use upper case letters such as $A$,
$B$ for elements of the $\s$-algebra $\B$ and lower case
letters such as $a$, $b$ for elements of $\widehat \B$. If $S$ is a subset of
 $\B$, we denote by $\widehat S$ the set of events
determined by the elements of $S$; \textit{i.e.},
$\widehat S = \{ [A]_\mu \mid A \in S \} \subseteq \widehat \B$.
\end{nota}

Whenever $\C \subseteq \B$ is a $\s$-subalgebra, $(X,\C,\mu{\rharp}\C)$
is a probability space in its own right, and we have defined 
$\widehat \C$ to be the probability algebra of $\C$, and also (in \ref{widehat C notation}(d))
to be a certain subset of $\widehat \B$.  There is no real ambiguity here; indeed, the inclusion
map $j$ of $\C$ into $\B$ induces a measure-preserving boolean isomorphism $\widehat j$ (which thus also preserves 
the metric) between the two versions
of $\widehat \C$.  (The function $\widehat j$ maps $[A]_{\mu {\rharp} \C}$ in the sense of 
$(X,\C,\mu{\rharp}\C)$ to $[A]_\mu$ in the sense of $(X,\B,\mu)$, for each $A \in \C$.)
With this identification, $(\widehat \C,\widehat \mu,\widehat d)$ is
the probability algebra of the probability space $(X,\C,\mu{\rharp}\C)$, 
and it is (canonically isomorphic to) a substructure of $(\widehat \B,\widehat \mu, \widehat d)$.

In the next result, we record for later use that the converse of the preceding comment is also true.

\begin{lema}
\label{closed subalgebras are probability algebras}
Let $(X,\B,\mu)$ be a probability space and let $(\widehat \B,\widehat \mu,\widehat d)$ be
its probability algebra.  Let $S$ be a subset of $\widehat \B$ and consider $\bra S \ket \subseteq \widehat \B$ as in
\ref{widehat C notation}(c).  Let 
\[
\S := \{ A \in \B \mid [A]_\mu \in \bra S \ket \}  \text{.}
\]
Then $\S$ is a $\s$-subalgebra of $\B$ and $\bra S \ket = \widehat \S$.

In particular, every substructure $(\bra S \ket,\mu{\rharp}\bra S \ket,d{\rharp}\bra S \ket)$ 
of the probability algebra 
$(\widehat \B, \widehat \mu, \widehat d)$  of a probability space $(X,\B,\mu)$ is (isomorphic to) the
probability algebra of a probability space $(X,\S,\mu{\rharp}\S)$, with $\S$ a $\sigma$-subalgebra 
of $\B$.
\end{lema}
\begin{proof}
For $A,B \in \S \subseteq \B$ we have $[A \cup B]_\mu =  [A]_\mu \cup [B]_\mu \in \bra S \ket $, so $\S$ is closed
under the union operation of $\B$.  Similar calculations show that $\S$ is closed under $\cap$ and $\cdot^c$.
Also, note that $\S$ contains every element of $\B$ that has $\mu$-measure $0$.

To finish the proof, we need to consider an increasing sequence $(A_n)$ in $\S$ and show that the union of $(A_n)$
in $\B$ is an element of $\S$.  Given such an $(A_n)$, the sequence $([A_n]_\mu) \subseteq \bra S \ket$ must be increasing in $\widehat \B$,
so it converges in the sense of the metric $\widehat d$ to an element $[B]_\mu \in \bra S \ket$, with $B \in \S$.  This means that
$\mu(A_n \triangle B) \rightarrow 0$ in $\B$.  Since $(A_n)$ is increasing, this implies $\mu(A_n \setminus B) = 0$ for all $n$.
It also implies $\mu(B \setminus (\cup A_n)) = 0$, and therefore $B$ differs from $\cup A_n$ by a set of $\mu$-measure $0$ in $\B$.
Hence $\cup A_n \in \S$.
\end{proof}

\bigskip
We need the following familiar special case of the Radon-Nikodym theorem:\index{Radon-Nikodym theorem}

\begin{theo}\cite[Theorem 3.8]{Folland}\label{Radon-Nikodym}
Let $(X,\B,\mu)$ be a probability space, let $\C\subseteq \B$ be a
$\s$-subalgebra, and consider  $A\in \B$. Then there exists $g\in L_1(X,\C,\mu)$ 
such that for every $B\in \C$, one has
$\int_{B}g d\mu=\int_{B} \chi_{A} d\mu$. The function $g$ is determined by 
$a = [A]_\mu$ up to equality $\mu$-almost everywhere; it is called the 
\index{conditional probability} 
\emph{conditional probability of $a$ with respect to $\C$} 
and we  denote it by $\P(a|\C)$ or, equivalently, by $\P(A|\C)$.  
We also refer to $g$ as \emph{representing} $\P(a|\C)$ and $\P(A|\C)$.%
\index{$\P(A  \vert  \C), \E(f  \vert \C)$}

More generally, for $f\in L_1(X,\mathcal{B},\mu)$ there exists 
$\E(f|\C)\in L_1(X,\C,\mu)$
such that for every $B\in \C$, one has
$\int_{B} \E(f|\C) d\mu=\int_{B} f d\mu$. 
The function $\E(f|\C)$ is unique in the sense that
the operation $f \mapsto \E(f|\C)$ preserves the equivalence
relation of equality $\mu$-almost everywhere.  The element
$\E(f|\C)$ is called the \emph{conditional expectation\index{conditional expectation} of
$f$ with respect to $\C$} and we  also denote it by $\E_{\C}(f)$.
\end{theo}

Note that for any $A \in \B$, the function $\P(A | \C)$ must have
its values in $[0,1]$ $\mu$-ae.  

\index{$\P(a \vert D)$}
\begin{nota}
\label{cond prob notation over a sigma-subalgebra of prob alg}
Let $(X,\B,\mu)$ be a probability space, and consider $a \in \widehat \B$.
Suppose $D$ is a boolean subalgebra of $\widehat \B$ that is closed
(with respect to the metric $\widehat d$), and let 
$\D =  \{ A \in \B \mid [A]_\mu \in D \} $ be the $\s$-subalgebra discussed
in the proof of Lemma \ref{closed subalgebras are probability algebras},
so $\widehat \D = D$.  We  write $\P(a|D)$ to denote 
$\P(\chi_A|\D)$, where $a = [A]_\mu$.  Similarly, for $a \in \widehat \B$ we write
$\chi_a$ to denote one of the characteristic functions $\chi_A$
where $A \in \B$ and $a = [A]_\mu$.  Note that if $A_1,A_2$ are two such sets for $a$,
then $\chi_{A_1}=\chi_{A_2}$ holds $\mu$-almost everywhere, and hence
the same is true of $\P(\chi_{A_1}|\D)$ and $\P(\chi_{A_2}|\D)$.

When we are given a probability algebra $(\B,\mu,d)$ without specifying the
underlying probability space, and $\A$ is a closed subalgebra of $\B$, 
we  refer to $\P(b|\A)$ as being an \emph{$\A$-measurable function}
in order to avoid explicitly introducing the probability space representing $(\B,\mu,d)$
and the $\sigma$-subalgebra representing $\A$ (as described 
in Lemma \ref{closed subalgebras are probability algebras}).  
Similarly we  refer to $\chi_a$ as being $\A$-measurable, when $a \in \A$.

As is customary, when $a,b \in \widehat \B$, we write
$\P(a|D) = \P(b|D)$ to mean that the functions 
$\P(a|D)$ and $\P(b|D)$ are equal $\mu$-almost everywhere, 
and we give a similar interpretation to $\P(a|D) \leq \P(b|D)$.  
Therefore, the associated strict partial ordering $\P(a|D) < \P(b|D)$ 
(meaning that $\P(a|D) \leq \P(b|D)$ is true while $\P(a|D) = \P(b|D)$ is false)
is true if and only if $\P(a|D) \leq \P(b|D)$ holds $\mu$-almost everywhere
and $\P(a|D) < \P(b|D)$ holds on a set of positive $\mu$-measure.
Similar remarks apply to these relations between other measurable real-valued functions 
(such as $\chi_a$).
\end{nota}

When $\EE = \{\emptyset,X\}$ is the trivial subalgebra, and $A \in \B$, then $\P(A|\EE)$ is the 
constant function $ f(x) := \mu(A)$ for all $x \in X$.
Indeed, this $f$ is $\EE$-measurable and $\int_E f d\mu = \mu(A \cap E) = \int_E \chi_A d\mu$ for
all $E \in \EE$, namely for $E = \emptyset$ and $E = X$.  More generally, for finite subalgebras
$\EE \subseteq \B$ we have the following formula for $\P(A|\EE)$, which is useful in many places
below.

\begin{lema}
\label{P(A|C) for finite C}
Let $\EE \subseteq \B$ be a finite subalgebra and $A \in \B$.  Suppose $E_1,\dots,E_n$ are the
atoms in $\EE$.  Then
$$ \P(A|\EE) = \sum_j \frac{\mu(A \cap E_j)}{\mu(E_j)} \chi_{E_j} \text{.}$$
\end{lema}
\begin{proof}
By additivity of the integral, it suffices to prove that the integral of the displayed function over each atom
$E_j$ is equal to $\int_{E_j} \chi_A \ d\mu$, which equals $\mu(A \cap E_j)$.  Since the sets $E_j$ are
pairwise disjoint, this is clear.
\end{proof}

\begin{fact}\label{droponnorm}
Let $(X,\B,\mu)$ be a probability space and let $\C\subseteq \B$ be a
$\s$-subalgebra. The conditional expectation operator $\E_{\C}$ restricted
to $L_2(X,\B,\mu)$ is the Hilbert space orthogonal projection of $L_2(X,\B,\mu)$ onto 
the subspace $L_2(X,\C,\mu)$. (See \cite[Proposition 4.2]{Br}.)

Let $\D\subseteq \C\subseteq \B$ be $\s$-subalgebras and $A\in \B$. 
Then $\P(A|\D)=\E_{\D}(\chi_A)=\E_{\D}(\E_{\C}(\chi_A))$ is 
the orthogonal projection of $\P(A|\C) = \E_{\C}(\chi_A)$ into $L_2(X,\D,\mu)$, so
\[
\| \P(A|\C) \|_2^2 - \| \P(A|\D) \|_2^2  = \| \P(A|\C) - \P(A|\D) \|_2^2 \text{.}
\]
Further, since we are working over a probability space, we have $\|f\|_2 = \|f\|_2 \|1\|_2 \geq |\bra f,1 \ket| =  \|f\|_1$ for
all $L_2$ functions $f$, by the Cauchy-Schwartz inequality, and therefore
\[
\| \P(A|\C) - \P(A|\D) \|_2  \geq \| \P(A|\C) - \P(A|\D) \|_1 \text{.}
\]
These give useful quantitative conditions for $\P(A|\C) \neq \P(A|\D)$.  They are used
in proving Remark \ref{superstable}, Fact \ref{properties0}(5) and Corollary \ref{quantitative entropy-forking}.

If $\EE \subseteq \D \subseteq \C \subseteq \B$ and $A \in \B$, the preceding discussion yields 
$$\| \P(A|\C) - \P(A|\D) \|_1 \leq
\| \P(A|\C) - \P(A|\D) \|_2 \leq  \| \P(A|\C) - \P(A|\EE) \|_2 \text{,}$$ 
which can be useful in working with approximations to $\P(A|\C)$, as we illustrate next.
\end{fact}

\begin{lema}
\label{approximating P(a|C)}
Let $(X,\B,\mu)$ be a probability space, $\C \subseteq \B$ a $\sigma$-subalgebra, and
$A \in \B$.  For each $k \geq 1$ there exists $(E_1,\dots,E_k) \in \C^k$, a partition of $X$,
such that for any $\sigma$-subalgebra $\D$ with $\{ E_1,\dots,E_k \} \subseteq \D \subseteq \C$ one has
$\| \P(A|\C) - \P(A|\D) \|_1 \leq 1/k$.
\end{lema}
\begin{proof}
Let $f = \P(\chi_A|\C)$, so $f$ is a $\C$-measurable $[0,1]$-valued function.  Let $I_1,\dots,I_k$ be the
intervals $I_j = [\frac{j-1}{k},\frac{j}{k})$ for $j = 1,\dots,k-1$ and $I_k = [\frac{k-1}{k},1]$.  So the
intervals are pairwise disjoint and their union is $[0,1]$.  For each $j$ let $E_j = \{ x \in X \mid f(x) \in I_j \}$.  
Then $(E_1,\dots,E_k)$ is a partition of $X$ in $\C$.  For any sequence $(r_1,\dots,r_k)$ such that
$r_j \in I_j$ for all $j$, we have
$ | f - \sum_j r_j \chi_{E_j} | \leq 1/k$
pointwise on $X$, and therefore
$\|  f - \sum_j r_j \chi_{E_j} \|_2 \leq 1/k \text{.}$

Now set $r_j = \frac{\mu(A \cap E_j)}{\mu(E_j)}$.  Since $f(x)$ is in $I_j$ for $x \in E_j$, we have 
$ \frac{j-1}{k}\mu(E_j) \leq \int_{E_j} f d\mu \leq \frac{j}{k}\mu(E_j)$   for all $j$.  Noting that 
$\int_{E_j} f d\mu = \int_{E_j} \chi_A d\mu = \mu(A \cap E_j)$
we see that  $r_j \in I_j$ for all $j=1,\dots,k$.  It follows using Lemma \ref{P(A|C) for finite C} that
$$ \| f - \P(\chi_A | \{E_1,\dots,E_k \}^\#) \|_2 = \| f - \sum_j \frac{\mu(A \cap E_j)}{\mu(E_j)} \chi_{E_j} \|_2 \leq 1/k \text{.} $$
Further, if $\D$ is any $\sigma$-subalgebra of $\B$ that contains $\{E_1,\dots,E_k \}$, then
$ \| f - \P(\chi_A | \D) \|_1 \leq 1/k \text{,}$ by the last statement in Fact \ref{droponnorm}, so 
$(E_1,\dots,E_k)$ satisfies the stated conditions.
\end{proof}

Probabilistic independence is very important in this paper.  For $A,B \in \B$, we say $A$ and $B$ are
\emph{(probabilistically) independent}\index{independence!probabilistic}
if $\mu(A \cap B) = \mu(A)\mu(B)$, and write $A \pindep B$.
Further, if $S,T$ are subsets of $\B$, we say $S$ and $T$ are \emph{(probabilistically) independent}
 and write $S \pindep T$ if $A \pindep B$ holds for every $A \in \bra S \ket$ and $B \in \bra T \ket$.%
 \index{$\pindep$}
 Not surprisingly to model theorists, we need a more general version of independence that is relative
 to a set of parameters:

\begin{defi}
\label{conditional independence}
If $\EE \subseteq \B$ is a $\sigma$-subalgebra of $\B$, we say $A$ and $B$ are
\emph{(conditionally) independent over $\EE$},\index{independence!conditional}
 and write $A \pindep_\EE B$ if
$$ \P(A \cap B | \EE) =  \P(A | \EE)  \cdot \P( B | \EE) \text{.} $$

More generally, if $S,T,W$ are subsets of $\B$, we say $S$ and $T$ are \emph{(conditionally) independent over $W$}
 and write $S \pindep_W T$ if $A \pindep_{\bra W \ket} B$ holds for every $A \in \bra S \ket$ and $B \in \bra T \ket$.
\end{defi}

Note that $\pindep_{\bra W \ket}$ reduces to $\pindep$ when $W = \emptyset$, since $\bra \emptyset \ket = \{\emptyset,X \}$.
Also, $S \pindep_W T$ if and only if $A \pindep_{\bra W \ket} B$ holds for every $A \in S^\#$ and $B \in T^\#$.

For us the following characterization of conditional independence is fundamental.  As usual in model theory,
if $Y,Z$ are sets of parameters, we denote $Y \cup Z$ by $YZ$.

\begin{lema}
\label{characterization of independence}
If $S,T,W$ are subsets of $\B$, then the following statements are equivalent:
\begin{enumerate}
\item[(i)] $S \pindep_W T$.
\item[(ii)] $ \P(A | \bra WT \ket) = \P(A | \bra W \ket) $ for all $A \in S^\#$.
\item[(iii)] $\P(A|\bra WT \ket)$ is $\bra W \ket$-measurable, for all $A \in S^\#$.
\item[(iv)] $ \| \P(A | \bra WT \ket)\|_2 = \| \P(A | \bra W \ket) \|_2 $ for all $A \in S^\#$.
\end{enumerate}
\end{lema}
\begin{proof}
(i) $\Leftrightarrow$ (ii): Apply \cite[Theorem 8.9]{K}, noting that (ii) is equivalent to the same statement
with $S^\#$ replaced by $\bra S \ket$, since $S^\#$ is dense in $\bra S \ket$. 

(ii) $\Leftrightarrow$ (iii): This is immediate.

(iv) $\Leftrightarrow$ (ii): Let $A \in \B$.  Since $\bra W \ket \subseteq \bra WT \ket$, Fact \ref{droponnorm}
gives us
\[
\| \P(A|\bra WT \ket) \|_2^2 - \| \P(A|\bra W \ket) \|_2^2  = \| \P(A|\bra WT \ket) - \P(A|\bra W \ket) \|_2^2
\]
from which follows
\[
\| \P(A|\bra WT \ket) \|_2 = \| \P(A|\bra W \ket) \|_2
\mbox{  if and only if } \P(A|\bra WT \ket) =  \P(A|\bra W \ket) \text{.}
 \]
 Applying the quantifier ``for all $A \in S^\#$'' yields the desired equivalence.
\end{proof}

The next result shows that several different definitions of $\pindep$ that one finds in the literature
are equivalent.

\begin{coro}
\label{characterization of independence-extra}
If $S,T,W$ are subsets of $\B$, then the following statements are equivalent:
\begin{enumerate}
\item[(i)] $S \pindep_W T$.
\item[(ii)] $S \pindep_W \bra WT \ket$
\item[(iii)] $\bra WS \ket \pindep_W \bra WT \ket$
\end{enumerate}
\end{coro}

\begin{proof}
(ii) $\Rightarrow$ (i) is clear. Assume now that $S \pindep_W T$ holds and prove (ii). By Lemma \ref{characterization of independence}(ii) we have that $ \P(A | \bra WT \ket) = \P(A | \bra W \ket) $ for all $A \in S^\#$ and thus $S \pindep_W \bra WT \ket$ holds, again using Lemma \ref{characterization of independence} part (ii). The proof of
(ii) $\Leftrightarrow$ (iii) is similar, since the definition of independence is a symmetric condition on the left and right families.
\end{proof}

\begin{rema}
\label{charact of indep over finite W}
When $W \subseteq \B$ is finite, we have the following simple characterization of $\pindep$. Namely,
$S \pindep_W T$ if and only if 
$\mu(A \cap B \cap C)\mu(C) = \mu(A \cap C)\mu(B \cap C)$ for every $A \in S^\#, \mbox{ and } B \in T^\#$ and 
every atom $C \in W^\#$.

This is easily proved by comparing coefficients in the expressions for $\P(A\cap B|W^\#)$ and
$\P(A|W^\#) \cdot \P(B|W^\#)$ given by Lemma \ref{P(A|C) for finite C}.
\end{rema}

\begin{nota}
Suppose $C,D,E$ are subsets of $\widehat \B$, and $S,T,W$ are subsets of $\B$ such that $\bra C \ket = \widehat{ \bra S \ket}$, 
$\bra D \ket = \widehat{ \bra T \ket}$, and $\bra E \ket = \widehat {\bra W \ket}$.  We write $C \pindep_E D$ to mean the same
as $S \pindep_W T$.
\end{nota}

Next we prove a Lemma that will be used in the proof of Theorem \ref{pr:non-div}.

\begin{lema}[Extension]
\label{extension}
Let $(X,\B,\mu)$ be a probability space. Let $\A$ be a
finite subalgebra of $\B$, with atoms $A_1,\dots,A_m$, and let $\C \subseteq \D$ be closed subalgebras
of $\B$.  Then there exists a probability space $(X',\B',\mu')$ and a boolean, measure-preserving 
embedding $B \mapsto B'$ of $\B$ into $\B'$, together with a finite subalgebra $\EE$ of $\B'$ whose 
atoms $E_1,\dots,E_m$ satisfy $\P(E_j|\C')=\P(A'_j|\C')$ for all $i = 1,\dots,m$ and
$\EE \perp \! \!\! \perp_{\C'}\D'$.  (Here for $\ZZ = \C \mbox{ or } \D$ we write $\ZZ'$ for $\{ B' \mid B \in \ZZ \}$.)
\end{lema}
\begin{proof}
Let $([0,1],\F,\lambda)$ be the Lebesgue measure space  on $[0,1]$ and take $(X \times [0,1], \B', \mu')$ to be
the product measure space, with $\B' = \B \otimes \F$ and $\mu' = \mu \otimes \lambda$. 

For each $B \in \B$, let $B' := B \times [0,1]$. The correspondence $B \mapsto B'$ is obviously a boolean, 
measure-preserving embedding of $\B$ into $\B'$. Also, for any $\B$-measurable function $f \colon X \to [0,1]$, we let $f'$ denote the $\B'$-measurable function 
$(x,y) \mapsto f(x)$.  We see easily that the embedding preserves integration; namely for any $B \in \B$ and $\B$-measurable
$f \colon X \to [0,1]$ we have by Fubini's Theorem
$\int_{B'} f' \,d(\mu\otimes \lambda) = \int_B \big( \int_{[0,1]} f'\,d\lambda\big) \,d\mu =  \int_B f \, d\mu$.  We also
note that $\P(B' | \C') = \P(B | \C)'$ for any $B \in \B$ and $\sigma$-subalgebra $\C$ of $\B$.  (Using Lemma \ref{approximating P(a|C)} it suffices to prove this when $\C$ is finite; this is done by applying Lemma \ref{P(A|C) for finite C}.  Note that if 
$C_1,\dots,C_k$ are the atoms of $\C$, then $C'_1,\dots,C'_k$ are the atoms of $\C'$, and for each $j=1,\dots,k$
we have $(\mu\otimes\lambda)(B' \cap C'_j) = \mu(B \cap C_j), (\mu\otimes\lambda)(C'_j) = \mu(C_j)$, 
and $\chi_{C'_j} = (\chi_{C_j})'$.)

For each $i=1,\dots,m$, let $f_i :={\P}(A_i|{\C})$ and note that since $A_1,\dots, A_m$ is a partition of $1$, 
we have $f_1(x)+\dots+f_m(x)=1$ on $X$ $\mu$-a.e.. Now let 
$E_1=\{(x,y)\in X\times [0,1]: 0\leq y\leq  f_1(x)\}$ and for $1<i\leq m$ 
let $E_i=\{(x,y)\in X\times [0,1]: f_1(x)+\dots+ f_{i-1}(x)< y \leq f_1(x)+\dots+ f_{i}(x))\}$.
The sets $\{E_i\}_{i\leq m}$ are $\B'$-measurable and pairwise disjoint, and 
$X' = X \times [0,1] =\bigcup_{i\leq m} E_i$ (except possibly for a set of $\mu'$-measure zero).

Note that for any $B \in \B$ and $i=1,\dots,m$, we have
\begin{align*}
(\mu \otimes \lambda)(B' \cap E_i)  = \int_{B'} \chi_{E_i}\, d(\mu \otimes \lambda) 
 = \int_{B} \Big( \int_{[0,1]} \chi_{E_i}\, d\lambda \Big)\, d\mu = \int_B f_i \, d\mu
\end{align*}
using the definitions and Fubini's Theorem.  If $C\in {\C}$, this gives
$$\int_{C'} \chi_{E_i}\, d(\mu\otimes \lambda)=\int_{C} f_i \,d \mu=\int_{C'} {\P}(A_i|{\C})'\, d(\mu\otimes \lambda)=\int_{C'} {\P}(A_i'|{\C}')\, d(\mu\otimes \lambda)$$
and therefore  ${\P}(E_i|{\C}')={\P}(A_i'|{\C}')$ for all $i=1,\dots,m$. 
For $D\in {\D}$ we get 
$$(\mu\otimes \lambda )(D'\cap E_i)=\int_{D'} f'_i \, d (\mu\otimes \lambda)=\int_{D'} {\P}(A_i|{\C})'\, d(\mu\otimes \lambda)=\int_{D'} {\P}(A_i'|{\C}')\, d(\mu\otimes \lambda) \text{,}$$ 
so ${\P}(E_i|{\D}')= {\P}(E_i|{\C}')$ for all $i=1,\dots,m$.  Therefore $\EE \pindep_{{\C}'}{\D}'$.
\end{proof}

\section{Some continuous model theory}
\label{continuous model theory}

In this paper we use the setting of continuous logic to discuss the
model theory of probability algebras. The fundamental ideas of continuous
logic are presented in \cite{BBHU,BU}. 
We assume familiarity with the material in these sources,
and often use it without specific reference.

In addition, we need some background concerning \emph{metric imaginaries}\index{metric imaginary}
in continuous logic, and their role in some topics within
stability theory, especially when dealing with canonical parameters for definable predicates
and definability of types.  Here we give pointers to published sources for this background,
and a very brief summary of the topics we use.

There is some treatment of imaginary sorts (i.e., of \emph{interpretations})\index{interpretation}
in our key references \cite{BBHU,BU}.  In \cite[Section 11]{BBHU} 
only finitary imaginaries are presented; these are quotients 
of finite products of sorts modulo a definable pseudometric.  

However, in our Section 5 and later in the paper, in connection with certain concepts in
stability theory, we need more general, \emph{infinitary} imaginaries.  These are quotients
of the product of a countably infinite family of sorts modulo a definable pseudometric; they are connected
to the existing structure by their projection maps onto the sorts from which they come.  A central
example of these imaginaries is given by \emph{canonical parameters}\index{canonical parameters}
 for a definable predicate relative
to the (possibly infinite) sequence of parameters used in defining it.  
These are treated in detail in \cite[Section 5] {BU}.

Given a continuous theory $T$, the many sorted theory obtained by adding to $T$ all possible 
metric imaginary sorts is called the \emph{meq expansion of $T$},\index{meq expansion}
 and it is denoted $T^{\meq}$.
Likewise, given $\M \models T$, the corresponding expansion of $\M$ to a model of $T^{\meq}$
is denoted $\M^{\meq}$. Presentations of the full construction of $T^{\meq}$ and some
of its properties are in \cite[Section 3.3]{MT of C*-algebras}, and in \cite[Section 1]{BY6}.
\index{$T^{\meq}$, $\M^{\meq}$}

In Section 8 we also use concepts and tools from stability theory in the setting of 
continuous model theory.  Many of these, including canonical parameters for 
formulas, definability of types, and canonical bases for stationary types, are 
developed in \cite[Sections 7 and 8]{BU}.  Beyond these, 
we use concepts such as \emph{types being parallel}, \emph{the parallelism 
class of a stationary type}, \emph{Morley sequences}, \emph{orthogonal types}, 
and \emph{non-multidimensional theories}. While these concepts 
lack a thorough exposition in the continuous model theory literature, it is not 
difficult to formulate and understand them based on how they are treated in 
the main references for stability theory in classical model theory, especially 
given the tools provided in \cite{BU}.  An example needed here of such a 
fact is that the canonical base of a stationary type is contained in the meq 
definable closure of a Morley sequence of that type.  For this material in the 
classical discrete setting, we follow closely the presentation in \cite{Bue}. 

\section{The model theory of probability spaces}
\label{model theory of prob spaces}

We deal here with structures of the form
$$\M=(\widehat \B,0,1,\cdot^{c},\cap,\cup,\widehat \mu,\widehat d)$$
where $(\widehat \B,\widehat \mu)$ is the probability algebra
of a probability space $(X,\B,\mu)$, $0$ is the event
corresponding to $\emp$ and $1$ is the event corresponding to $X$;
$\cdot^{c}$ is the complement operation and $\cap$,$\cup$ are the
intersection and union operations on $\widehat \B$; and $\widehat d$ is the
canonical metric on $\widehat \B$ (defined for $a,b\in \widehat \B$
by $\widehat d(a,b)= \widehat \mu(a \triangle b)$).  The predicates,
namely $\widehat\mu$ and $\widehat d$, take their values in the interval $[0,1]$.
The modulus of uniform continuity for 
the unary operation $\cdot^c$ and the unary predicate $\widehat\mu$ is
given by $\Delta(\epsilon)=\epsilon$; for the binary operations
$\cap$ and $\cup$ the modulus is given by $\Delta(\epsilon)=\epsilon/2$.

For the rest of this paper we take $\LLpr$ to be the continuous signature
indicated in the previous paragraph.\index{$\LLpr$}

Note that every probability algebra of a probability space is indeed an $\LLpr$-structure. 
(This requires, in particular, that it is complete as a metric space, which we noted in 
Section \ref{probability spaces}.)

\begin{nota}
For any $\LLpr$-prestructure $\M$ and $a\in M$, we write $a^{-1}$ for
$a^{c}$ and $a^{+1}$ for $a$.
\end{nota}

The following $\LLpr$-conditions are easily seen to be true in every 
probability algebra of a probability space. 

\begin{enumerate}
\item Boolean algebra axioms: \\
Each of the usual axioms for a boolean algebra is the $\forall$-closure of an
equation between terms (see  \cite[p.38]{Ho}) and thus it can be
expressed in continuous logic as a condition. For example, the axiom
$\forall x \forall y(x\cup y=y\cup x)$ is equivalent to $\sup_x
\sup_y \big( d(x \cup y, y \cup x)\big)=0$.

\item Measure axioms:\\
$\mu(0)=0$ and $\mu(1)=1$ \\ $\sup_x \sup _y \big( \mu(x\cap
y)\dotmin \mu(x)\big)=0$\\ $\sup_x \sup _y \big(\mu(x)\dotmin
\mu(x\cup y)\big)=0$\\ $\sup_{x}\sup_{y} |(\mu(x)\dotmin \mu(x\cap
y))- (\mu(x\cup y)\dotmin \mu(y))|=0$\\ The last three axioms
express that $\mu(x\cup y)+\mu(x\cap y)=\mu(y)+\mu(x)$ for all $x$,
$y$.

\item Connections between $d$ and $\mu$:\\
$\sup_x \sup_y |d(x,y)-\mu(x\Delta y)|=0$ where $x \Delta y$ denotes
the boolean term giving the symmetric difference: $x \Delta y = (x
\cap y^c) \cup (x^c \cap y)$.
\end{enumerate}

We denote the set of $\LLpr$-conditions above by $Pr$.\index{$Pr$}

\begin{nota}
\label{notation for a general model of Pr}
For the rest of this paper we  use the notation $\M = (\B,\mu,d)$
for a general model of $Pr$.  We take care that this notation is not confused
with our usual notation $(X,\B,\mu)$ for a probability space
and $(\widehat\B,\widehat\mu, \widehat d)$ for its associated probability algebra.\index{probability algebra!abstract}
\end{nota}

\begin{theo}
\label{first description of models of Pr}
The models $\M = (\B,\mu,d)$ of $Pr$ are exactly the \emph{(abstract) probability algebras}.
That is, $\B$ is a $\sigma$-order complete Boolean algebra, and $\mu$ is a strictly positive,
$\sigma$-additive probability measure on $\B$; further, $d$ is defined on $\B$ by
$d(a,b) := \mu(a \triangle b)$.
\end{theo}
\begin{proof}
If $\M = (\B,\mu,d)$ is indeed a probability algebra as described in the statement, then
it is clear that it satisfies all conditions in $Pr$.  Moreover, the metric space $(\B,d)$ is
complete, as shown by the calculation in \cite[Lemma 323F]{Fr-treatise}.  

Conversely, suppose $\mathcal{M}$ is a model of $Pr$.  It is clear
from the axioms that $\mathcal{M}$ consists of a boolean algebra
$\B$ with a finitely additive probability measure $\mu$ such that
$\B$ is a complete metric space under the metric 
$d(a,b) = \mu(a \triangle b)$.  Moreover, $\mu$ must be continuous on $\B$
with respect to $d$; indeed, $\mu$ is $1$-Lipschitz with respect to $d$,
as is dictated by the signature $\LLpr$. 

Any increasing sequence in $\B$ is necessarily a Cauchy sequence
with respect to $d$, so it converges.  This and the continuity of
$\mu$ ensure that $\B$ is $\sigma$-order complete as a boolean
algebra and $\mu$ is $\sigma$-additive on $\B$. 
\end{proof}

It follows from Theorem \ref{first description of models of Pr}
that the models of $Pr$ are (up to isomorphism) 
exactly the probability algebras of probability spaces.
This is proved in \cite[Theorem 321J]{Fr-treatise}; a key ingredient in the proof of that
result is the Loomis-Sikorski representation theorem for 
$\sigma$-order complete boolean algebras; see \cite[Theorem 314M]{Fr-treatise}.  
In Theorem \ref{second description of models of Pr} we 
give a proof of this fact about the models of $Pr$ using tools from model theory.  

\begin{exam}
\label{models of Pr from finitely additive probability spaces}
Suppose $\A$ is a boolean algebra and $\mu$ is a finitely additive probability measure on $\A$.
We may define a distance $d$ on $\A$ in the familiar way, by setting $d(a,b)$ equal to
$\mu(a \triangle b)$, where $\triangle$ denotes the symmetric difference in $\A$.  Then
$(\A,\mu,d)$ is an $\LLpr$-prestructure, and it satisfies all of the axioms of $Pr$.
Therefore we may obtain a model $(\widehat \A, \widehat \mu, \widehat d)$ of $Pr$ 
by first taking the quotient of $(\A,\mu,d)$
by the ideal of elements of $\mu$-measure $0$, and then taking the metric completion of
the resulting quotient (as discussed in the middle of \cite[pages 329-331]{BBHU}).

For those readers who are familiar with Abraham Robinson's nonstandard analysis (NSA), we
note how this construction relates to the Loeb measure construction \cite{Loeb}, which has been one of the
most important tools for applications of NSA.  For that construction, we begin with $\A$
being an internal boolean algebra of subsets of an internal set $X$, and 
$\mu$ being  obtained from an internal finitely additive
${}^*[0,1]$-valued measure $\nu$ on $\A$, by taking $\mu(a)$ to be the standard part of $\nu(a)$
for each $a \in \A$.  Let  $(\widehat \A, \widehat \mu, \widehat d)$ be constructed as above
from $(\A,\mu,d)$ as in the preceding paragraph.  
In that setting, the quotient algebra of $\A$ by the ideal of $\mu$-null sets
is already complete with respect to the quotient metric obtained from $d$, owing to the 
assumption of $\omega_1$-saturation that is part of the basic NSA framework.  Moreover, 
the saturation assumption also implies that $\mu$ has a natural and unique extension to
a $\sigma$-additive probability measure on the $\sigma$-algebra of subsets of $X$ that is
generated by $\A$.  The resulting probability space has $(\widehat \A, \widehat \mu, \widehat d)$ 
as its probability algebra.  See  \cite[Section II.2]{Lindstrom} and \cite[Section 2.1]{Ross} for elementary discussions
of the Loeb construction and its basic properties.

The metric ultraproduct of a family of probability algebras of probability spaces is 
an example of the Loeb construction.  In that case, the internal measure space is the
discrete ultraproduct of the family of probability spaces.

This approach gives an alternative way of proving that every
model of $Pr$ is the probability algebra of some probability space, as we show next.
\end{exam}

\begin{theo}
\label{second description of models of Pr}
Let $\M$ be a $\LLpr$-structure.  The following are equivalent:\\
(1) $\M$ is a model of $Pr$. \\
(2) $\M$ is isomorphic to the probability algebra of a probability space.
\end{theo}
\begin{proof}
(2) $\Rightarrow$ (1): See the first paragraph of the proof of Theorem \ref{first description of models of Pr}.

(1) $\Rightarrow$ (2): Let $\M$ be a model of $Pr$.  Let $I$ be the set of all finite subsets of $M$.  For each $\tau \in I$, 
let $\M_\tau$ be the subalgebra $\tau^\#$ of $\M$, which is finite.  Each $\M_\tau$ is the probability algebra of a finite 
probability space $(X_\tau,\A_\tau,\mu_\tau)$.  Here $X_\tau$ is the set of atoms in $\M_\tau$, $\A_\tau$ is the boolean
algebra of all subsets of $X_\tau$, and $\mu_\tau(\{ a \}) = \mu(a)$ for each element $a$ of $X_\tau$.

There exists an ultrafilter $U$ on $I$ such that for each $a \in M$ the set $\{ \tau \in I \mid a \in \tau \}$
is an element of $U$.  As discussed in the preceding example, the $U$-ultraproduct of the family 
$(\M_\tau \mid \tau \in I)$ is the probability algebra of a probability space, by the Loeb measure construction.
Moreover, $\M$ is isomorphic to a substructure of this ultraproduct; 
the embedding maps $a \in M$ to the equivalence class of the family
$(a_\tau \mid \tau \in I)$ where we define $a_\tau$ as follows: (i) if $a \not\in \tau^\#$ we take $a_\tau = 0$; 
(ii) if $a \in \tau^\#$, we take $a_\tau$ to be the subset of $X_\tau$ consisting of all atoms of $\M_\tau$ that are
contained in $a$ (so $a$ is the join of $a_\tau$ in $\M$).

Therefore we have embedded $\M$ into the probability algebra of a probability space.  The proof is
completed by applying Lemma \ref{closed subalgebras are probability algebras}. 
\end{proof}
In the rest of this section we aim to discuss elementary equivalence
of probability algebras and to characterize (axiomatize) the
complete extensions of $Pr$.  This depends on studying the
definability in continuous logic of the set of atoms (and some related sets) in models of
$Pr$.  (See \cite[Section 9]{BBHU} for a discussion of definable predicates and definable sets.)

In the rest of this section $\M$ denotes a model of $Pr$, with
underlying boolean algebra $ \B$, measure $\mu$ and metric
$d$.  We let $A^\M_1$ denote the set of atoms of $ \B$
together with $0$. Note that for each $r>0$ there are only finitely many $a \in A^\M_1$ 
such that $\mu(a) \geq r$. Therefore $A^\M_1$ is finite or
countable; the join (union) of $A^\M_1$ is therefore in $\B$ and provides a measurable splitting of $1$ in $ \B$
between its atomic and atomless parts, either of which may be $0$.
Also, $A^\M_1$ is a closed set with respect to the metric $d$.

We consider the following formulas in the signature of $Pr$:
\begin{align*}
\chi(x) &:= \inf_y | \mu(x \cap y) - \mu(x \cap y^c)|  \\
\psi(x) &:= \mu(x) \dotminus \chi(x) \\
\varphi_1(x) &:= \inf_z (d(x,z) \dotplus \psi(z)) \\
\theta(x) &:= \sup_y \inf_z  | \mu(x \cap y \cap z) - \mu(x \cap y \cap z^c)|
\end{align*}

To understand the meanings of these formulas in models of $Pr$,
the next result is needed.  The elementary argument needed for the proof
is given in \cite[Section 41, Theorem A]{HaMT}.
\begin{lema}
\label{approximate splitting}
Suppose $\M = (\B,\mu,d) \models Pr$.  If $b \in \B$ is atomless, then for every $\delta > 0$ there
is a partition of $1$ in $\B$, say $u = (u_1,\dots,u_n)$, such that $\mu(b \cap u_i) \leq \delta$
for all $i = 1,\dots,n$.
\end{lema}
\begin{proof}
It is sufficient to prove the result assuming $b=1$ in $\B$.  
By the downward L\"{o}wenheim-Skolem Theorem, $\M$ has a separable elementary substructure $\M'$, which is
necessarily also atomless, and it obviously suffices to prove the Lemma for $\M'$.  Suppose $\M'$ is based on the algebra 
$\B'$, which is a closed subalgebra of $\B$, and the predicates of $\M'$ are the restrictions of $\mu$ and $d$ to $\B'$. 
By separability of $\M'$, we may take $(\A_n \mid n \geq 1)$ to be an increasing family of finite boolean 
subalgebras of $\B'$ such that $\bigcup (\A_n \mid n \geq 1)$ is a dense subset of $\B'$.  For each $n \geq 1$, let
$\pi_n$ be the partition of $1$ in $\A_n$ that consists of the atoms of $\A_n$.  The argument for Theorem A in
\cite[Section 41]{HaMT} shows that these partitions satisfy the conclusion of the Lemma.
\end{proof}

\begin{prop}\label{distance to atoms}
Let $\M = (\B,\mu,d) \models Pr$ and $b \in \B$.
\begin{enumerate}
\item[(a)] If $b$ is atomless, then $\chi^\M(b) = 0$.
\item[(b)] $b$ is an atom or $0$ if and only if $\chi^\M(b) = \mu(b)$.
\item[(c)] If $b$ is not atomless and $a$ is an atom of largest
measure contained in $b$, then $\chi^\M(b) \leq \mu(a)$.
\item[(d)] $\dist(b,A^\M_1) = \varphi_1^\M(b)$.
\item[(e)] $b$ is atomless in $\B$ if and only if $\theta^\M(b) = 0$.
\end{enumerate}
\end{prop}
\begin{proof}
(a)  Fix $\delta > 0$ and use Lemma \ref{approximate splitting} to obtain a partition
of $1$ in $\B$, say $u = (u_1,\dots,u_n)$ such that $\mu(b \cap u_i) \leq \delta$ for
all $i = 1,\dots,n$.  Let $a_i = b \cap u_i$ for all $i$, so $b = a_1 \cup \dots \cup a_n$.  There exists $i$ such that 
$\mu(a_1 \cup \dots \cup a_i) \leq \frac12 \mu(b) \leq \mu(a_1 \cup \dots \cup a_{i+1})$.
Then $y = a_1 \cup \dots \cup a_i$ witnesses $\chi^\M(b) \leq \delta$.

(b)  If $b$ is an atom and $a$ is arbitrary, then
one of the events $b \cap a, b \cap a^c$ equals $b$ and the other is
$0$. In that case $ | \mu(b \cap a) - \mu(b \cap a^c)| = \mu(b)$ for all $a$,
so indeed $\chi^\M(b) = \mu(b)$.

 If $b$ is not an atom, there exists $a \in \B$ such that $\mu(b) > \mu(b \cap a) > 0$
and $\mu(b) > \mu(b \cap a^c) > 0$, from which it follows that 
$| \mu(b \cap a) - \mu(b \cap a^c)| < \mu(b)$.

(c)  Suppose $b$ is not atomless and let $a_1,a_2,\dots$ be
a listing of all the (finitely or countably many) distinct atoms of
$ \B$ contained in $b$, arranged so that $\mu(a_1)
\geq \mu(a_2) \geq \dots$. Take $u \subseteq b$ to be the union of
all $a_j$ such that $j$ is odd and $v \subseteq b$ to be the union
of all $a_j$ such that $j$ is even.  Then $u,v$ are disjoint and   
$b \setminus (u \cup v)$ is atomless.  One checks easily that
$\chi^\M(b) \leq \mu(u) - \mu(v) \leq \mu(a_1)$.

(d)  The key idea is this:  if $b$ is atomless, then
$\dist(b,A^\M_1) = \mu(b)$; if $b$ is not atomless and $a$ is an 
atom of largest measure contained in $b$, then
$\dist(b,A^\M_1) = \mu(b) - \mu(a)$.  

Therefore, from (a) and (c) we
conclude that $\dist(b,A^\M_1) \leq \psi^\M(b)$ for all $b$. From
(b) we see that $\psi^\M(b) = 0$ when $b$ is an atom, and therefore
$A^\M_1$ is the zeroset of $\psi^\M$.  This makes it clear that
$\dist(b,A^\M_1) \geq \varphi_1^\M(b)$.  Conversely,
for every $b$ we have
\begin{align*}
\varphi_1^\M(b) \geq \inf_z (d(b,z) \dotplus \dist(z,A_1^\M)) \geq  \dist(b,A_1^\M)
\end{align*}
which completes the proof.

(e)  This follows from (a) and (b).  Note that $\theta(b) = 0$ is equivalent to saying
$\chi(u) = 0$ holds for every $u \leq b$.
\end{proof}

Proposition \ref{distance to atoms}(d) shows that $A_1^\M$ is a
definable set, uniformly in all models $\M$ of $Pr$.  (See \cite[Definition
9.16]{BBHU}.)  It is useful to introduce for each $n>1$ the further set
$$A_n^\M \ = \  \{x_1 \cup \dots \cup x_n \mid x_1,\dots,x_n \in A_1^\M \} \text{.} $$
Note that $A_1^\M \subseteq A_2^\M \subseteq \dots \subseteq A_n^\M
\subseteq \dots$. Using \cite[Theorem 9.17]{BBHU} and the
definability of the set $A_1^\M$, we may conclude that $A_n^\M$ is a
definable set in all models $\M$ of $Pr$, for all $n \geq 1$.
Indeed, as we show next, the distance to $A_n^\M$ is given
explicitly by the following formula in the signature of $Pr$ (where
we define the formulas for $n>1$ by induction on $n$):
$$ \varphi_n(x) \ = \ \inf_w (\varphi_{n-1}(x \cap w) \dotplus \varphi_1(x \cap w^c)) $$

\begin{prop}\label{distance to unions of n atoms}
Let $\M = (\B,\mu,d) \models Pr$ and $a \in \B$.  Then for each $n \geq 1$
$$ \dist(a,A_n^\M) = \varphi_n^\M(a) \text{.} $$
\end{prop}
\begin{proof}
 Let $a \in \B$ and let
$a_1,a_2,\dots$ be a listing of all distinct atoms
contained in $a$, arranged so that $\mu(a_1) \geq \mu(a_2) \geq
\dots$, and extended to an infinite sequence by taking $a_k = 0$
for larger $k$, if necessary.  We note that $\dist(a,A_n^\M) = d(a,w)$ where $w = a_1 \cup \dots \cup a_n$, and for
this $w$ we have $\dist(a,w) = \mu(a) - \mu(a_1 \cup \dots \cup a_n) = \mu(a) - (\mu(a_1) + \dots + \mu(a_n))$.

To prove the Lemma, we argue by induction on $n \geq 1$.  The $n=1$ case is Proposition
\ref{distance to atoms}(d). Taking $n>1$, it remains to prove the
induction step from $n-1$ to $n$. 

First note that if we take $w = a_1 \cup \dots \cup
a_{n-1}$ we have $\varphi_{n-1}^\M(a \cap w) = 0$ (by
the induction hypothesis) and $\varphi_1^\M(a \cap w^c) = \mu(a \cap w^c) - \mu(a_n)$ 
(by Proposition \ref{distance to atoms}(d)).
Therefore, for this $w$ we have
\begin{align*}
\varphi_{n-1}^\M(a \cap w) + &\varphi_1^\M(a \cap w^c) = \mu(a \cap w^c) - \mu(a_n) \\
& = (\mu(a) - \mu(a_1 \cup \dots \cup a_{n-1})) - \mu(a_n)  \\
&= \mu(a) - (\mu(a_1) + \dots + \mu(a_{n-1}) + \mu(a_n)) \\
&= \dist(a,A_n^\M) \text{.}
\end{align*}

To finish the argument, it suffices to prove that for any other $w$ we have 
$$\big{(}\varphi_{n-1}^\M(a \cap w) \dotplus \varphi_1^\M(a \cap w^c)\big{)}\geq \mu(a) - (\mu(a_1) + \dots + \mu(a_n))   \text{.}$$ 
So fix $w \in M$ and let $a^1_{1}, a^1_{2},\dots$ be a listing of all distinct atoms
contained in $a\cap w$, arranged so that $\mu(a^1_1) \geq \mu(a^1_2) \geq
\dots$, and extended to an infinite sequence by taking $a^1_k = 0$ if necessary. 
Also, let $a^2_1$ be one of the the largest atoms contained in $a\cap w^c$ (which can be $0$).
Note that the nonzero elements among $a^2_1,a^1_{1},\dots,a^1_{n-1}$ are distinct atoms, and all are $\leq a$.
By the induction hypothesis, $\varphi_{n-1}^\M(a \cap w)=\mu(a\cap w)-\sum_{i=1}^{n-1}\mu(a^1_i)$
and by Proposition \ref{distance to atoms}(d)  $\varphi_{1}^\M(a \cap w^c)=\mu(a\cap w^c)-\mu(a^2_1)$.
Thus 
\begin{align*}
 \varphi_{n-1}^\M(a \cap w) \dotplus \varphi_1^\M(a \cap w^c) &= \big(\mu(a\cap w)-\sum_{i=1}^{n-1}\mu(a^1_i)\big)
 + \big(\mu(a\cap w^c)-\mu(a^2_1)\big) \\&=\mu(a)-\Big(\sum_{i=1}^{n-1}\mu(a^1_i)+\mu(a^2_1)\Big)
 \end{align*}
The smallest possible value of this last expression occurs when $a_2^1,a^1_{1},\dots,a^1_{n-1}$ 
have the largest possible measures, which happens when the sequence
$\mu(a_2^1),\mu(a^1_{1}),\dots,\mu(a^1_{n-1})$ is a permutation of $\mu(a_1),\dots,\mu(a_n)$.
\end{proof}

For a similar treatment of atoms in the setting of random variable structures
see \cite[Lemma 2.16]{BY4}.

\begin{rema}
\label{distance to unions of n atoms decreases}
Let $\M = (\B,\mu,d)  \models Pr$ and $a \in \B$.  Propositions \ref{distance to atoms}(d) 
and \ref{distance to unions of n atoms} make it clear that
$$\mu(a) \geq \varphi_1^\M(a) \geq \varphi_2^\M(a) \geq \dots \geq \varphi_n^\M(a) \geq \dots $$
\end{rema}

\begin{nota}Let $\M = (\B,\mu,d)  \models Pr$ and $a \in \B$; let $a_1,a_2,\dots$ be a listing
of all distinct atoms contained in $a$, arranged so that
$\mu(a_1) \geq \mu(a_2) \geq \dots$, and extended to an infinite
sequence by taking $a_k = 0$ for larger $k$, if necessary. For each
$n \geq 1$, we refer to $\mu(a_n)$ as \emph{the $n^{th}$ largest measure of 
an atom contained in $a$}, and we denote this
number as $at_n^\M(a)$.  
\end{nota}
Note that the nonzero elements of $(a_n \mid n \in \N)$
are distinct, whereas the measure values $(\mu(a_n) \mid n \in \N)$ may contain
repetitions.

\begin{coro}
For each $n \geq 1$, the predicate $at_n$ is definable in all models
of $Pr$. Indeed, if $\M = (\B,\mu,d)  \models Pr$ and $a \in \B$, then
$$ at_1^\M(a) = \mu(a) \dotminus \varphi_1^\M(a)$$
and for each $n > 1$
$$ at_n^\M(a) = \varphi_{n-1}^\M(a) \dotminus \varphi_n^\M(a) \text{.} $$
\end{coro}

\begin{proof}
This is immediate from Proposition \ref{distance to unions of n atoms}.
\end{proof}

We now consider an extension by definitions of $Pr$ obtained by adding
unary predicate symbols $(P_n \mid n \geq 1)$ to the signature and by
adding as axioms the conditions
$$ \sup_x \big| P_1(x) - (\mu(x) \dotminus \varphi_1(x))\big| = 0$$
and for $n > 1$
$$  \sup_x \big| P_n(x) - (\varphi_{n-1}(x) \dotminus \varphi_n(x)) \big| = 0\text{.} $$
This extension of $Pr$ is denoted by $Pr^*$.  
\index{$Pr^*$}
Note that each model
$\M$ of $Pr$ has a unique expansion, which we denote by
$\M^*$, that is a model of $Pr^*$. This expansion is given by
interpreting each $P_n$ so that, for each $a \in M$, one takes
$P_n^{\M^*}(a)$ to be the $n^{th}$ largest measure of an atom
contained in $a$.

\begin{nota}
\label{tuples generate partitions}
Suppose $\M = (\B,\mu,d) \models Pr$.
Fix $n \geq 1$ and consider any tuple $a = (a_1,\dots,a_n) \in M^n = \B^n$.  Let
$e = (e_1,\dots,e_{2^n})$ be the partition of $1$ in the boolean algebra $\B$ generated
by $a_1,\dots,a_n$.  By this we mean that the elements of $e$ are all possible intersections
of the form $a_1^{k_1}\cap\dots\cap a_n^{k_n}$, where each $k_i$ comes
from $\{ -1,+1 \}$, and we list these intersections in order according to
lexicographic order on the tuples of superscripts $k_1,\dots,k_n$.  We refer to 
$e$ as \emph{the partition of $1$ in $\B$ associated to $a$}.\index{partition!associated to a tuple}  
Further, note that when $(a_1,\dots,a_n)$ and
$(e_1,\dots,e_{2^n})$ are as above, then each $a_i$ is the union of the
coordinates $e_j$ of $e$ that are intersections
$a_1^{k_1}\cap\dots\cap a_n^{k_n}$ in which $k_i = +1$.  That is,
the correspondence between coordinates of $(a_1,\dots,a_n)$
and coordinates of $(e_1,\dots,e_{2^n})$ is given, in both directions,
by simple boolean terms that depend only on $n$.  (In particular, these tuples are uniformly
interdefinable in models of $Pr$.)

For simplicity of notation, we write $a^s$ for $a_1^{k_1}\cap\dots\cap a_n^{k_n}$
when $s = (k_1,\dots,k_n)$ is an arbitrary element of $\{-1,+1\}^n$ and $a = (a_1,\dots,a_n) \in M^n$.
Likewise we write $x^s$ for the boolean term $x_1^{k_1}\cap\dots\cap x_n^{k_n}$
when $s = (k_1,\dots,k_n)$ and $x$ stands for the tuple $(x_1,\dots,x_n)$ of variables.
\index{$a^s$, $x^s$}
As indicated above, the identity
\[\
x_i = \bigcup_s (x^s \mid s = (k_1,\dots,k_n) \mbox { and } k_i = +1)
\]
is true in all models of $Pr$.  Frequently when we use this notation, as here, we omit the
standard specifications that $x = (x_1,\dots,x_n)$ and $s = (k_1,\dots,k_n) \in \{ -1,+1 \}^n$.
In particular, we view $k_i$ as a function of $s$ when $s \in \{ -1,+1 \}^n$.
When it is needed, we list the elements of $\{ -1,+1 \}^n$ in lexicographical order.  
\end{nota}

\begin{rema}
\label{special formulas over Pr}
For future use we note that for every $\LLpr$-formula $\vphi(x)$ there exists an $\LLpr$-formula
$\psi(y_s \mid s \in \{ -1,+1 \}^n)$, such that $\vphi(x)$ is equivalent to $\psi(x^s \mid s \in \{ -1,+1 \}^n)$ in all models of $Pr$.  
Indeed, it suffices to take $\psi(y_s \mid s \in \{ -1,+1 \}^n)$ to be the result of substituting the 
boolean term $\bigcup ( y_s \mid s = (k_1,\dots,k_n) \in \{ -1,+1 \}^n \mbox{ and } k_i = +1 )$ for
the variable $x_i$ in $\vphi(x)$, for $i=1,\dots,n$.

Further, for every $\LLpr$-formula $\vphi(x)$ there exists an $\LLpr$-formula $\psi(x)$ such that 
$\psi(x)$ is $Pr$-equivalent to $\vphi(x)$ and every atomic formula occurring in $\psi(x)$ is of the form 
$\mu(x^s)$ for some $s \in \{ -1,+1 \}^n$.  Moreover, $\psi(x)$ can be chosen so that 
it is obtained from such atomic formulas using the restricted 
connectives $0,1,t \mapsto t/2$, and $(t,u) \mapsto t \dotminus u$.

(Proof: A general atomic formula $\alpha(x)$ in $\LLpr$ can be taken to be one of the form $\mu(t(x))$ 
where $t(x)$ is a boolean term in $x$.  If $\alpha(x) = d(t_1(x),t_2(x))$, then $\alpha(x)$ can be replaced  
by $\mu(t_1(x) \triangle t_2(x))$.  For each such $t(x)$ there is a subset $S \subseteq \{ -1,+1 \}$ such 
that the equation $t(x) = \cup (x^s \mid s \in S)$ is true in all models of $Pr$.  (If $S$ is empty, then 
$t(x)=0$ is true in all models of $Pr$.) Moreover, $\mu(t(x)) =\widehat\sum(\mu(x^s) \mid s \in S)$ is true in all models of 
$Pr$, where by $\widehat\sum$ we mean the connective $(u_1,\dots,u_{2^n}) \mapsto \min(\sum(u_s \mid s \in S),1)$.  
For the ``Moreover'' statement, including 
treatment of the connectives $\widehat\sum$, see \cite[Chapter 6]{BBHU}.)

A consequence of the preceding observation is that when $C$ is a subalgebra of a model $\M$ of $Pr$, then for
any $a \in M^n$ the type $\tp_{\M}(a/C)$ is determined by the values $\psi^\M(a)$ of $\LLpr$-formulas $\psi(x)$ over $C$ 
in which all atomic formulas are of the form $\mu(x^s \cap c)$ for some $c \in C$.  
(As above, $x = (x_1,\dots,x_n)$ and $s \in \{-1,+1\}^n$.)
\end{rema}

\begin{theo}\label{QE for Pr*}
The theory $Pr^*$ admits quantifier elimination.
\end{theo}

\begin{proof} 
We use \cite[Theorem 4.16]{BU}, so we need to show
that $Pr^*$ has the \emph{back-and-forth property} given in
\cite[Definition 4.15]{BU}.  Therefore, consider two $\omega$-saturated models
$\M^*,\NN^*$ of $Pr^*$ and tuples $(a_1,\dots,a_n)$ in $\M^*;
(b_1,\dots,b_n)$ in $\NN^*$ such that the quantifier-free type of
$(a_1,\dots,a_n)$ in $\M^*$ is the same as the quantifier-free type of
$(b_1,\dots,b_n)$ in $\NN^*$. Given any $u$ in $\M^*$ we need to find $v$ in $\NN^*$
such that $(a_1,\dots,a_n,u)$ and $(b_1,\dots,b_n,v)$ have the same
quantifier-free type in the language of $Pr^*$.  It suffices to do
this for the case in which $(a_1,\dots,a_n)$ and $(b_1,\dots,b_n)$ are
partitions of $1$, by the discussion in \ref{tuples generate partitions}.

Using Lemma \ref{approximate splitting} and the fact that $\M$ is $\omega$-saturated,
for each atomless $c \in M$ and $0<r<1$ there exists $a \leq c$ in $M$ such that
$\mu(a) = r  \mu(c)$, and hence $\mu(c \cap a^c) = (\mu(c)-r) \mu(c)$.  Indeed, $x=a$
can be taken to satisfy all of the conditions $|\mu(x) - r \mu(c)| \leq \delta$ for $\delta > 0$,
which we just showed were finitely satisfiable in $\M$.  This is used
in the next paragraph.

In the assumed situation we know that for each $j = 1,\dots,n$ we have
$\mu(a_j) = \mu(b_j)$ and, for all $k \geq 1$, we also have
$P_k^{\M^*}(a_j) = P_k^{\NN^*}(b_j)$.  For each $j$, let $a_j^0$ be the atomic
part of $a_j$ (i.e., the union of the atoms of $\M$ that are $\leq a_j$), so
$a_j^1 := a_j \cap (a_j^0)^c$ is the atomless part of $a_j$.  Define $b_j^0,b_j^1$ from
$b_j$ similarly.  Our assumptions yield that $\mu(a_j^0) = \mu(b_j^0)$ (and indeed,
that the atoms below $a_j$ and $b_j$ are in a bijective, measure-preserving
correspondence).  Hence also $\mu(a_j^1) = \mu(b_j^1)$.

Take any $u$ in $\M^*$ and fix $j = 1,\dots,n$.  Define $v_j^0 \leq b_j^0$ to be
the union of the atoms below $b_j$ that correspond to atoms below $a_j \cap u$.
Further, choose $v_j^1 \leq b_j^1$ so that $\mu(v_j^1) = \mu(a_j \cap u) - \mu(v_j^0)$,
and let $v_j = v_j^0 \cup v_j^1$.  We obtain
\begin{align*}
\mu(v_j) &= \mu(a_j \cap u), \mbox{  and}\\
P_k^{\NN^*}(v_j) &= P_k^{\M^*}(a_j \cap u) \mbox{ for all } k \geq 1 \text{.}
\end{align*}
Then let $v = v_1 \cup \dots \cup v_n$, and note that $v_j = b_j \cap v$ for all $j$. 
It follows that the quantifier-free type of $(a_1,\dots,a_n,u)$ in $\M^*$ is the same as
the quantifier-free type of $(b_1,\dots,b_n,v)$ in $\NN^*$, as desired.
\end{proof}

Theorem \ref{QE for Pr*} allows us to characterize (and axiomatize)
the complete extensions of $Pr$.  
\begin{defi}
\label{invariants for atoms}
For any $\M = (\B,\mu,d) \models Pr$, let
$\Phi^\M$ denote the sequence $(at_n^\M(1) \mid n \geq 1)$, which
lists the sizes of the atoms of $\B$ in decreasing order (and then has
a tail of $0$s if there are only finitely many atoms in $\B$).  
\end{defi}
\index{$\Phi^\M$}

Note that the range of the operator $\Phi$ consists of all the sequences $(t_n
\mid n \geq 1)$ such that $1 \geq t_1 \geq t_2 \geq \dots \geq 0$
and $\sum_{n=1}^\infty t_n \leq 1$.

\begin{coro}\label{axioms for complete extensions of Pr}
For models $\M,\NN$ of $Pr$, we have that $\M \equiv \NN$ if and
only if $\Phi^\M = \Phi^\NN$.  Therefore, any complete extension
$T$ of $Pr$ (in the same signature) can be axiomatized by adding to
$Pr$ the conditions $\varphi_1(1) = 1-t_1$ and $\varphi_{n-1}(1)
\dotminus \varphi_n(1) = t_n$ (for $n>1$), where $(t_n \mid n \geq
1)$ is the common value of $\Phi^\M$ for $\M \models T$.
\end{coro}
\begin{proof}
Let $\M,\NN$ be models of $Pr$ such that $\Phi^\M = \Phi^\NN$.
From the definition of the operator $\Phi$ we see that $1$ has the
same quantifier-free type in $\M^*$ as in $\NN^*$. Theorem \ref{QE
for Pr*} yields that $\M^* \equiv \NN^*$, from which it follows that
$\M \equiv \NN$.  The converse and the rest of the Corollary follow
using the definition of $\Phi$.
\end{proof}

\begin{coro}
\label{omega categoricity and homogeneity of all models of Pr}
Every  completion $T$ of $Pr$ (in $\LLpr$) is separably categorical, and the
unique separable model of $T$ is strongly $\omega$-homogeneous.
\end{coro}
\begin{proof}
Let $T$ be a completion of $Pr$ and let $\M$ be a separable model of $T$.
A strengthening of Lemma \ref{approximate splitting} that is proved in
\cite[Section 41]{HaMT} says that the key property used in the proof of 
Theorem \ref{QE for Pr*} is actually true in all models, without assuming
they are $\omega$-saturated.  That is, for each atomless $c \in M$ and $0<r<1$ 
there exists $a \leq c$ in $M$ such that
$\mu(a) = r  \mu(c)$, and hence $\mu(c \cap a^c) = (\mu(c)-r) \mu(c)$.

So the proof of Theorem \ref{QE for Pr*} 
not only shows that $Pr^*$ admits quantifier elimination, but shows further that
for each $\M^* \models Pr^*$ and every $(a_1,\dots,a_n)$ in $\M^*$, every 
$1$-type over $(a_1,\dots,a_n)$ for the theory of $\M^*$ is realized in $\M^*$.
In other words, every model of $Pr^*$ is $\omega$-saturated.  It follows trivially
from the definition that the same is true of every structure $(\M,a_1,\dots,a_n)$
where $\M$ is a model of $Pr$ and $a_1,\dots,a_n \in M$.

It is routine to show that if $\M_1,\M_2$ are $\omega$-saturated separable 
metric structures for the same language, and $\M_1,\M_2$ are elementarily
equivalent, then $\M_1$ and $\M_2$ are isomorphic.  One uses the usual inductive
back-and-forth argument to produce an elementary bijection $f \colon S_1 \to S_2$,
where $S_i$ is a dense subset of $M_i$ for both values of $i$.  Then $f$ extends
to a map on $M_1$ that is an isomorphism from $\M_1$ onto $\M_2$.  The
proof of the corollary is completed by applying this construction to models of
the form $(\M,a_1,\dots,a_n)$ discussed above.
\end{proof}

\begin{rema}
\label{separable categoricity of APA}
Corollary \ref{omega categoricity and homogeneity of all models of Pr} yields
the following well known fact due to Carath\'{e}odory: if $(X,\B,\mu)$ and
$(Y,\C,\nu)$ are atomless, countably generated measure spaces with
$\mu(X) = \nu(Y) < \infty$, then the measured algebras of
$(X,\B,\mu)$ and $(Y,\C,\nu)$ are isomorphic.  Proof: Without loss of generality,
we may take $\mu(X) = \nu(Y) = 1$.  In that case, the measured algebras of
these two probability spaces are atomless, separable models of $Pr$.  Using
Theorem \ref{QE for Pr*} we see that these probability algebras are elementarily
equivalent (since in an atomless probability algebra the predicates interpreting
$P_n$ are identically $0$) and hence by
Corollary \ref{omega categoricity and homogeneity of all models of Pr} we 
get the desired result.  For proofs in analysis see \cite[Section 41]{HaMT} and
\cite[Theorem 4, p.\ 399]{Roy}.  
\end{rema}

\begin{coro}\label{categoricity for models of Pr}
Let $T$ be any complete $\LLpr$-theory that extends $Pr$ and let
$(t_n \mid n \geq 1)$ be the common value of $\Phi^\M$ for $\M
\models T$.
\newline
(1) If $\sum_{n=1}^\infty t_n = 1$, then $T$ has a unique model,
which consists of an atomic probability algebra having atoms $(a_n
\mid n \geq 1 \mbox{ and } t_n>0)$ with $\mu(a_n) = t_n$ for all
$n$.
\newline
(2) If $\sum_{n=1}^\infty t_n < 1$, then the models of $T$ are
exactly the probability algebras with atoms as described in (1)
together with an atomless part of measure $1 - \sum_{n=1}^\infty
t_n$.  
\end{coro}
\begin{proof}
These statements are immediate 
from Corollaries \ref{axioms for complete extensions of Pr} and 
\ref{omega categoricity and homogeneity of all models of Pr}.  
\end{proof}

\begin{rema}
Let $\M$ be any model of $Pr$.  From the previous results it follows 
that $\acl_\M(\emptyset)$ is the $\sigma$-subalgebra
generated by the atoms in $\M$.
\end{rema}

\begin{rema}
\label{definability of atomics and atomless}
If $\M \models Pr$, let $a_0 \in M$ be the join of the atoms of $\M$,
and let  $a_1 = a_0^c$.  
Further, let $
\A_i = \{ a \in M \mid a \leq a_i \} $
for each $i=0,1$.  Thus $\A_0$ is the set of atomic elements of $\M$ and $\A_1$
is the set of atomless elements, and $a_1$ is the largest atomless element.
The partition $\{a_0,a_1\}$ of $1$ is important because it splits every element $a$
of $M$ into its atomic part $a \cap a_0$ and its atomless part $a \cap a_1$.
We note the following facts concerning the definability of these elements and sets: 
\newline
(i) Relative to all models $\M$ of $Pr$: $\A_0$ is not a zeroset; $\A_1$ is
a zeroset but is not a definable set; $a_0,a_1$ are not definable elements.
\newline
(ii)  Fix a complete extension $T$ of $Pr$.  Relative to all models of $T$:
$a_0,a_1$ are definable elements and $\A_0,\A_1$ are
definable sets.  
\begin{proof}  
(i) First we show $\A_0$ is not a zeroset over $Pr$.  Suppose otherwise, so
there exists a definable predicate $R(x)$ 
over $\Pr$ such that for all $\M \models Pr$ and all $a \in M$, we have
$R^\M(a)=0$ if and only if $a$ is an atomic element in $\M$.  
For each $n \geq 1$, let $\M_n$ be the
probability algebra of the probability space having $n$ points, each of which has
measure $1/n$, and let $\M$ be the metric ultraproduct
of $(\M_n \mid n \geq 1)$ with respect to a nonprincipal
ultrafilter.  Then $1$ is atomic in every $\M_n$ while it is not atomic in $\M$; indeed,
$\M$ is atomless, so $0$ is its only atomic element.  This means $R^\M(1) \neq 0$
whereas $R^{\M_n}(1) = 0$ for all $n$.  This violates the Fundamental Theorem of
ultraproducts for the definable predicate $R$.  (See Theorem 5.4 and
the discussion of extensions by definition in Section 9 in
\cite{BBHU}.)  

Proposition \ref{distance to atoms}(e)
shows that $\A_1$ is the zeroset of the formula $\theta(x)$  in all models of $Pr$.

Further, $\A_1$ cannot be a definable set uniformly in all models of $Pr$, since otherwise there
would be a definable predicate $R(x)$ over $Pr$ such that for all $\M \models Pr$
and all $a \in M$
\[
R^\M(a) = \sup \{\mu(a \cap y) \mid y \in \A_1 \} \text{.}
\]
But then the zeroset of $R(x)$ would be $\A_0$ in all models of $Pr$, which we
just proved is not possible.

Since $a_0^c = a_1$, they are either both definable or both undefinable.
We work with $a_0$.  If $a_0$ were definable over $Pr$, then the operation
$x \mapsto a_0 \cap x$ would be definable, so its image, which is $\A_0$,
would be a definable set, but it isn't.

(ii)  Consider a complete extension $T$ of $Pr$ and let
$\U$ be an $\omega_1$-universal domain for $T$.
By Corollary \ref{categoricity for models of
Pr}, we see that $a_0$ and $a_1$ are fixed by every automorphism of $\U$.  Using
\cite[Exercise 10.7 or Theorem 9.32]{BBHU} it follows that $a_0$ and $a_1$ are each 
definable uniformly in all models of $T$.  Finally, note that for each
$i=0,1$, the set $\A_i$ is the image of the definable function
$f_i$ defined by $f_i(x) = x \cap a_i$.  Using \cite[Theorem 9.17]{BBHU}, 
we infer that $\A_i$ is a definable set uniformly in all models of $T$.
\end{proof}
\end{rema}

\begin{exer}
By Remark \ref{distance to unions of n atoms decreases}, 
for any $\M \models Pr$ and any $a \in M$, we have that $(\vphi_n^\M(a)) \mid n\geq 1)$
is a decreasing sequence from $[0,1]$, so it converges, and its limit must be the distance from $a$ to the
set of atomic elements of $\M$, which is the closure of $\cup_n A_n^\M$.  
Show that if we restrict attention to models of a
completion $T$ of $Pr$, this convergence is uniform (in $\M$ as well as $a$), so its limit is a $T$-definable
predicate.  This gives an alternative proof of Remark \ref{definability of atomics and atomless}(ii).
\end{exer}

\section{Random variables}
\label{random variables}

Here we discuss how to represent the space $RV$ of $[0,1]$-valued random variables 
(modulo pointwise equality a.e.), equipped with the $L_1$-distance, as a metric imaginary sort for the theory $Pr$. 
We focus on the measure theoretic aspects of the matter, and avoid many
technical details of the $\meq$ construction, for which we refer to various 
articles for the details (see Section \ref{continuous model theory}).
An interpretation of random variables in atomless probability algebras 
was originally given by Ben Yaacov in \cite{BY1} in the $CAT$ setting, 
and extended to an interpretation of $RV$ in $Pr$ by him in \cite{BY4}.  In \cite{BY1} certain other classes
of random variables are also treated, and the approach used here can easily be extended to apply to them.

When $\M$ is an arbitrary model of $Pr$, we let $(X,\B,\mu)$ denote a probability
space whose probability algebra $(\widehat \B, \widehat \mu, \widehat d)$ is (isomorphic to) $\M$.  
In much of this section we argue in $(X,\B,\mu)$ using measure theory.

We denote by $RV = RV(\B,\mu) := L_1(\B,\mu;[0,1])$ the space of all
$\B$-measurable functions $f \colon X \to [0,1]$.  We equip this space with the $L_1$ (pseudo)metric,
for which the distance between $f$ and $g$ is $\|f-g\|_1 = \int_X |f-g| \, d\mu$.  As done here for probability 
spaces, we consider the quotient metric space obtained by identifying two random variables if they are equal
pointwise $\mu$-a.e., equipped with the distance induced by $\| f-g \|_1$ 
(which we call the \emph{$L_1$-distance}).%
\index{$L_1$-distance}
\index{$RV$, $RV_n$}

The goal of this section is to explain how this quotient can be seen as a metric imaginary sort
for the model $(\widehat \B, \widehat \mu, \widehat d)$ of $Pr$.  One value of doing so is that it allows 
seeing $\P(a|\A)$ as an imaginary in $(\widehat \B, \widehat \mu, \widehat d)^{\meq}$, for any $a \in \widehat \B$
and any closed subalgebra $\A$ of $\widehat \B$.

For $n \geq 1$, consider the subset $RV_n = RV_n(\B,\mu)$ of $L_1(\B,\mu;[0,1])$ consisting of those 
functions that can be written as $\sum_{i=1}^n \frac{i}{n} \chi_{E_i}$ 
where $E = (E_1,\dots,E_n)$ is a partition of $X$ from $\B$.  Below we denote $\sum_{i=1}^n \frac{i}{n} \chi_{E_i}$
by $f_E$.  On $RV_n(\B,\mu)$ we take the $L_1$-distance.
If $f_E \in RV_m$ and $f_F \in RV_n$, we have $f_E,f_F$ both in $L_1(\B,\mu;[0,1])$, so we can compare
them, compute the $L_1$-distance between them, etc., even if $m \neq n$.
\index{$f_E$}

Of course we may identify $RV_n$ with the set  of $E = (E_1,\dots,E_n)$ that are partitions
of $X$ in $\B$.  For two such partitions $E,F$, we define $\rho_n(E,F)$ to be the $L_1$-distance;
\begin{equation*}\tag{A}
\begin{aligned}
 \rho_n(E,F) & := \| f_E - f_F\|_1 = \Big\| \sum_{i=1}^n \frac{i}{n} \chi_{E_i} - \sum_{i=1}^n \frac{i}{n} \chi_{F_i} \Big\|_1 \\
& = \frac1n \sum_{i \neq j} |i-j| \mu(E_i \cap F_j) \text{.}
\end{aligned}
\end{equation*}
\index{$\rho_n(E,F)$}

By analyzing the expressions in $(A)$, we obtain (in the following Lemma) a Lipschitz equivalence between 
 $\rho_n(E,F)$ and ${d}_P(E,F) :=  \frac12\sum_{i=1}^n d(E_i,F_i) =  \frac12 \sum_{i=1}^n \mu(E_i \triangle F_i)$.  
 Note that when $n=2$ we have ${d}_P(E,F) = \frac12 \big(d(E_1,F_1) + d(E_1^c,F_1^c) \big) = d(E_1,F_1)$,
since $d(E_1^c,F_1^c) = d(E_1,F_1)$; this explains the factor $\frac12$.  Also, ${d}_P$ is
equivalent to the usual pseudometric $d(E,F) := \max_i d(E_i,F_i)$, since 
\[
\frac12 d(E,F) \leq {d}_P(E,F) \leq \frac{n}{2} d(E,F)
\]
holds for all $E,F \in RV_n$.
\index{${d}_P(E,F)$}

\begin{lema}
\label{rho_n compared to L_1-distance}
For all $E,F \in RV_n$ we have
\begin{equation*}
\frac{1}{n} {d}_P(E,F) \leq  \rho_n(E,F) = \| f_E - f_F \|_1 \leq {d}_P(E,F) \text{.}
\end{equation*}
\begin{proof}
First note that for each $i$ the family $(E_i \cap F_j \mid j \neq i)$ is a partition of $E_i \setminus F_i$, 
and the same with $E$ and $F$ interchanged.  Therefore
\begin{align*}
\sum_{i=1}^n \mu(F_i \setminus E_i) = \sum_{i \neq j} \mu(E_i \cap F_j) = \sum_{i=1}^n \mu(E_i \setminus F_i) \text{,}
\end{align*}
from which follows
\begin{align*}
  \sum_{i \neq j} \mu(E_i \cap F_j) = \frac12 \sum_{i=1}^n \mu(E_i \triangle F_i) = {d}_P(E,F)  \text{.}
\end{align*}
Then we get the desired inequality using $(A)$ together with
\begin{equation*}
 \frac1n \sum_{i \neq j} \mu(E_i \cap F_j) \leq \frac1n \sum_{i \neq j} |i-j| \mu(E_i \cap F_j) \leq  \sum_{i \neq j} \mu(E_i \cap F_j) \text{.} \qedhere
\end{equation*}
\end{proof}
\end{lema}

Write $\widehat {RV}_n = \widehat {RV}_n(\B,\mu)$ 
for the image of $RV_n$ under the quotient map from $\B$ onto $\widehat \B$.
If $\M$ is the model $(\widehat \B,\widehat\mu,\widehat d)$, we also denote $\widehat{RV}_n$
by $\widehat{RV}_n^\M$, and note that it satisfies
\[
\widehat{RV}_n^\M = \{ (e_1,\dots,e_n) \in M^{n} \mid (e_1,\dots,e_n) \mbox{ is a partition of $1$ in } \M \} \text{.}
\]
\index{$\widehat \rho_n^\M(e,f)$}
We put on $\widehat{RV}_n^\M$ the pseudometric $\widehat \rho_n^\M$ obtained canonically 
from $\rho_n$ on ${RV}_n$; that is,
for $e,f \in \widehat{RV}_n^\M$, we have
\begin{equation*}
\tag{B}   
\begin{aligned}
\widehat \rho_n^\M(e,f) =  \sum_i \sum_j \Big| \frac{i}{n} - \frac{j}{n} \Big| \mu(e_i \cap f_j) \\
= \frac1n \sum_i \sum_j \big| i-j \big| \mu(e_i \cap f_j) \text{.}
\end{aligned}
\end{equation*}

\begin{lema}
\label{imaginary sort from RVn}
Let $\M  \models Pr$.  Then $\widehat \rho^\M_n$ is a complete metric on $\widehat{RV}_n^\M$.  Also,
$\widehat{RV}^\M_n$ is a definable set, and $\widehat \rho^\M_n$ 
is a definable predicate on $\widehat{RV}^\M_n$, uniformly in all models of $Pr$.  
\end{lema}
\begin{proof}
Obviously $\widehat \rho^\M_n$ is a pseudometric.  The fact that it is a metric follows from 
Lemma \ref{rho_n compared to L_1-distance} and the definition of ${d}_P$ on $RV^\M_n$, which implies
that ${d}_P(E,F)=0$ iff $\mu(E_i \triangle F_i)$ for all $i = 1,\dots,n$.  

To show that $\widehat{RV}_n^\M$ is uniformly a definable subset of $M^n$, 
it is sufficient to show that it is the image of a definable function on a definable set.  To do this, consider the
function defined on $M^{n-1}$ by $(a_1,\dots,a_{n-1}) \mapsto (e_1,\dots,e_n)$ where 
$e_1=a_1$, $e_j = a_j \cap a_1^c \cap \dots \cap a_{j-1}^c$ for $2 \leq j \leq n-1$, and 
$e_n = a_1^c\cap \dots \cap a_{n-1}^c$.  Obviously this is a definable function, since the coordinates
$e_j$ are given by boolean terms.  Note that $e_i \cap e_j = 0$ whenever $i \neq j$.  Moreover, by
induction on $j < n$ we can show $a_1 \cup \dots \cup a_j = e_1 \cup \dots \cup e_j$.  Therefore
$(e_1,\dots,e_{n-1})$ is a partition of $a_1 \cup \dots \cup a_{n-1}$.  Since 
$e_n = (a_1 \cup \dots \cup a_{n-1})^c$, we have that $(e_1,\dots,e_n)$ is always a partition of $1$
in $\M$.  To see that this map is surjective onto $\widehat{RV}_n^\M$, note that whenever
$(e_1,\dots,e_n)$ is a partition of $1$, then $(e_1,\dots,e_{n-1}) \mapsto (e_1,\dots,e_n)$

Equation (B) shows that $\widehat \rho_n^\M$ is a definable predicate, uniformly on all models $\M$ of $\Pr$.

It remains to show that $\widehat{RV}^\M_n$ is complete.  Because $M = \widehat \B$ is complete with respect to the
metric $d(a,b) := \widehat \mu(a \triangle b)$ and $\widehat{RV}^\M_n \subseteq (\widehat \B)^n$ is closed, we see that
$\widehat{RV}^\M_n$ is complete with respect to the metric ${d}_P(e,f) = \frac12 \sum_{i=1}^n d(e_i,f_i)$.  Therefore
$\widehat{RV}^\M_n$ is complete with respect to $\widehat \rho^\M_n$, since $\widehat \rho^\M_n$ is uniformly equivalent
to ${d}_P$ by Lemma \ref{rho_n compared to L_1-distance}.
\end{proof}

\begin{rema}
Similar reasoning to that in the preceding proof shows that the map $E \mapsto f_E$ for $E \in RV^\M_n$ is
an isometric map from $RV^\M_n$ into $L_1(\B,\mu;[0,1])$ whose range is the collection of
all random variables whose values are in $\{ 1/n, \dots, n/n \}$.  (Here isometric means with respect to
$\rho^\M_n$ and the $L_1$-distance.)  Hence this map induces an isometry from $\widehat{RV}_n^\M$
onto the set of $=$ a.e.\ equivalence classes of those random variables.
\end{rema}

Fix $\M \models Pr$ and, as above, let $(X,\B,\mu)$ be a probability  
space whose probability algebra is (isomorphic to) $\M$.   
We denote by $\widehat{RV}^\M$ the inverse limit of the spaces 
$(\widehat{RV}_{2^n}^\M \mid n \geq 1)$ equipped with a suitable family of maps 
$\widehat \pi \colon \widehat{RV}_{2^{n+1}}^\M \to \widehat{RV}_{2^n}^\M$ that we now define. For
$(e_1,\dots,e_{2^{n+1}}) \in \widehat{RV}_{2^{n+1}}^\M$ we set
\[
\widehat \pi(e_1,\dots,e_{2^{n+1}}) := (e_1 \cup e_2,\dots,e_{2^{n+1}-1} \cup e_{2^{n+1}}) \text{.}
\]
This clearly makes $\widehat \pi \colon \widehat{RV}_{2^{n+1}}^\M \to \widehat{RV}_{2^n}^\M$ a definable map,
uniformly for all models of $Pr$.  Note that an element of $\widehat{RV}^\M$ is given by a sequence
$(e(n))_n = (e(n) \mid n\geq 1)$ where $e(n) = (e_1(n),\dots,e_{2^n}(n)) \in \widehat{RV}_{2^n}^\M$ and
$\widehat \pi(e(n+1)) = e(n)$ for all $n \geq 1$.   In what follows, we refer to such a 
sequence as \emph{coherent}.\index{coherent sequence}

On $\widehat{RV}^\M$ we want to define the inverse limit pseudometric,
denoted by $\widehat \rho^\M$, by
\[
\widehat \rho^\M\big((e(n))_n,(f(n))_n\big) := \lim_n\, \widehat \rho_n(e(n),f(n))
\]
for any elements $(e(n))_n,(f(n))_n$ of $\widehat{RV}^\M$.  
The fact that the limit in this definition of $\widehat \rho^\M$ exists 
(with a rate of convergence that can be taken to be uniform over all 
sequences in $\widehat{RV}^\M$ and all $\M \models Pr$)) 
and that the resulting quotient corresponds to  $L_1(\B,\mu;[0,1])$ 
modulo the $L_1$-distance follows from the results in Lemma \ref{L_1 quotient sort}
below.  For proving such results it is useful to pull the objects involved 
back to the probability space $(X,\B,\mu)$ of which $\M$ is the probability algebra.  

We write $\pi$ for the corresponding maps from $RV_{2^{n+1}}$ to $RV_{2^n}$; namely
\[
\pi(E_1,\dots,E_{2^{n+1}}) := (E_1 \cup E_2,\dots,E_{2^{n+1}-1} \cup E_{2^{n+1}}) \text{.}
\]
\index{$\pi(E_1,\dots,E_{2^{n+1}})$}
Via the quotient map from $\B$ to $M = \widehat \B$, the inverse limit $\widehat{RV}^\M$ described above 
corresponds to the  inverse limit of the spaces $(RV_{2^n}(\B,\mu) \mid n\geq 1)$ equipped with their
$L_1$-pseudometrics and the 
connecting maps $\pi \colon  RV_{2^{n+1}} \to  RV_{2^n}$.  An element of this inverse limit
consists of a sequence $(E(n) \mid n \geq 1)$ such that $E(n) \in RV_{2^n}$ and
$\pi(E(n+1)) = E(n)$, for all $n \geq 1$.  As above, we refer to such a 
sequence as \emph{coherent}.\index{coherent sequence}

\begin{fact}
It is evident that the quotient map from $\B$ to $\widehat \B$ induces a map from coherent sequences
$(E(n) \mid n \geq 1)$ of measurable partitions of $X$ to coherent sequences $(\widehat E(n) \mid n \geq 1)$
of partitions of $1$ in $\M$, where $\widehat E(n) := (\widehat {E_1(n)},\dots,\widehat {E_{2^n}(n)})$ for all $n \geq 1$.
In fact, this map is surjective.  That is, 
suppose $(e(n) \mid n\geq 1)$ is a coherent sequence in $\widehat{RV}^\M$, with
$e(n) = (e_1(n),\dots,e_{2^n}(n))$ for all $n \geq 1$.  It is easy to show, working by induction on $n$,
that there exists a coherent sequence $(E(n) \mid n \geq 1)$ with $E(n) \in RV_{2^n}$ for all $n$ such that
$e_j(n) = \widehat {E_j(n)}$ for all $1 \leq j \leq 2^n$ and all $n \geq 1$.
\end{fact}

\begin{lema}
\label{L_1 quotient sort}
\begin{itemize}
\item[(a)] For all $m \geq 1$ and $E \in RV_{2m}$, we have $f_E \leq f_{\pi(E)} \leq f_E + \frac1m$
pointwise; therefore $\| f_E - f_{\pi(E)} \|_1 \leq \frac1m$.
\item[(b)] For every coherent sequence $(E(n) \mid n \geq 1)$, with $E(n) \in RV_{2^n}$ for all $n \geq 1$,
the sequence $(f_{E(n)} \mid n \geq 1)$ is monotone decreasing pointwise and has
$\| f_{E(n+1)} - f_{E(n)} \|_1 \leq 2^{-n}$ for all $n \geq 1$.  Therefore $(f_{E(n)} \mid n \geq 1)$
converges to its pointwise infimum, in $L_1$-distance, in $L_1(\B,\mu;[0,1])$, and does so with a rate of
convergence that can be taken to be uniform over all sequences in ${RV}^\M$ and all probability
spaces $(X,\B,\mu)$.
\item[(c)]  For every $f \in RV = L_1(\B,\mu;[0,1])$ there is a coherent sequence 
$(E(n) \mid n \geq 1)$, with $E(n) \in RV_{2^n}$ for all $n \geq 1$ such that $(f_{E(n)} \mid n \geq 1)$
converges to $f$ in $L_1$-distance.
\end{itemize}
\end{lema}
\begin{proof}
(a) This follows immediately from the definitions.

(b) From the definition of $f_{E(n)}$ we have $0 \leq f_{E(n)}(x) \leq 1$ for all $x \in X$.
Further, (a) implies that $\chi_X - f_{E(n)}$ is pointwise monotone increasing on $X$
and that it converges in $L_1$.  Since  $\chi_X$ is integrable, the Monotone Convergence 
Theorem implies that $(\chi_X - f_{E(n)} \mid n \geq 1)$ converges in $L_1$ to its 
a.e. pointwise supremum, and therefore $(f_{E(n)} \mid n \geq 1)$ converges in $L_1$ 
to its  a.e.\,pointwise infimum.  Uniformity of the rate of convergence follows from the
uniformity of the estimates in (a).

(c) For each $n \geq 1$, consider the dyadic intervals $I_1,\dots,I_{2^n}$ defined by
$I_1 := [0,2^{-n}]$ and for $1 < j \leq 2^n$ by $I_j :=  ((j-1)2^{-n},j2^{-n}]$. Let
$E(n) \in RV_{2^n}$ be defined by $E(n) := (E_1(n),\dots,E_{2^n}(n))$ where
$E_j(n)  := f^{-1}(I_j)$ for $j = 1,\dots,2^n$.  Note that the sequence $(E(n) | n \geq 1)$
is coherent.  Let $\C \subseteq \B$ be the $\sigma$-subalgebra generated by 
 $\{ E_j(n) \mid n \geq 1 \mbox{ and } 1 \leq j \leq 2^n \}$; clearly $\C$ is the smallest
 $\sigma$-subalgebra of $\B$ such that $f$ is $\C$-measurable.  In particular,
 $\E(f | \C) = f$.  The proof of Lemma \ref{approximating P(a|C)} shows that $(f_{E(n)} \mid n \geq 1)$ 
converges to $f$ relative to the $L_1$-distance.  The rate of convergence is as indicated in (b).
\end{proof}

\begin{coro}
\label{inverse limit random variables}
The inverse system $(\widehat{RV}_{2^n}^\M,\widehat\rho_{2^n}^\M)_{(n \geq 1)}$ equipped with the maps 
$\widehat \pi \colon \widehat{RV}_{2^{n+1}}^\M \to \widehat{RV}_{2^n}^\M$ 
has an inverse limit pseudometric space $(\widehat{RV}^\M,\widehat\rho^\M)$ 
for every model $\M = (\widehat \B, \widehat \mu, \widehat d)$ of $Pr$.
Its metric quotient corresponds to  $L_1(\B,\mu;[0,1])$ modulo the $L_1$-distance.
Moreover, this quotient is a metric imaginary sort for $\M$, uniformly over all models $\M$ of $Pr$.
\end{coro}
\begin{proof}
The first two sentences follow from Lemma \ref{L_1 quotient sort}. For the third sentence
we apply \cite[Lemma 1.5]{BY6}.
\end{proof}

\begin{rema}
\label{not a metric space}
Note that although the spaces $\widehat{RV}_n^\M$ are metric spaces (as shown in the proof
of Lemma \ref{imaginary sort from RVn}), the inverse limit distance
$\widehat \rho^\M$ is not a metric.  That is, there exist distinct coherent sequences $(e(n))_n, (f(n))_n$
such that $ \lim_n\, \widehat \rho^\M_n(e(n),f(n)) = 0$.  Therefore, to obtain the imaginary sort
described in Corollary \ref{inverse limit random variables}, it is necessary to form the
metric quotient.  However, in contrast to the general case of metric imaginary sorts, the metric space 
quotient of $(\widehat{RV}^\M,\widehat \rho^\M)$ is complete no matter which model $\M$ of $Pr$
is being considered.  (In general one needs to take the metric completion of such a quotient for
some models.)  Indeed, by the Riesz-Fischer Theorem, the quotient of $L_1(\B,\mu)$ modulo 
the $L_1$-distance is complete, no matter which probability space $(X,\B,\mu)$ is considered, 
and the image of $L_1(\B,\mu;[0,1])$ in that quotient is a norm-closed subset.  
(See \cite[Theorem 6.6, pp.\ 124--125]{Roy}.)
\end{rema}

\bigskip
We next discuss some definable operations on the metric imaginary sort
just described; they correspond to the operations taken to be basic (or proved to be definable) 
in \cite{BY4} and thus show that our imaginary sort does indeed provide a model of the theory $RV$.
(See \cite[Lemma 2.13]{BY4} for a discussion of the corresponding operations on models of $RV$.)

Consider any continuous function $\theta \colon [0,1]^m \to [0,1]$.  This function induces an
operation on $L_1(\B,\mu; [0,1])$ by composition, which we also denote by $\theta$.
Namely, we define
\[
\theta (f_1,\dots,f_m)(x) := \theta(f_1(x),\dots,f_m(x))
\]
for all $x \in X$, where $f_1,\dots,f_m \colon X \to [0,1]$ are $\B$-measurable .

We show below that $\theta$ induces a definable operation from $\widehat{RV}^m$ to
$\widehat{RV}$.  Let us here restrict its domain to $RV_n(\B,\mu)$.  
We consider the case $m=2$ to reduce the complexity of notation.  For $E,F \in RV_n$ we have
\[
\theta(f_E,f_F) = \sum_{i=1}^n \sum_{j=1}^n \theta(\frac{i}{n},\frac{j}{n}) \chi_{E_i \cap F_j} \text{,}
\]
which induces a definable function of the events of $(E,F)$, uniformly over all probability spaces.

Moreover, the restrictions of $\theta$ to the spaces $\widehat{RV}_n(\B,\mu)$ converge to a definable function 
on their inverse limit.  Again restricting attention to the case $m=2$, suppose $(E(n))_n$ and $(F(n))_n$
are two coherent families representing elements of the inverse limit, and suppose $f,g \in L_1(\B,\mu; [0,1])$
are their limits: $f = \lim_n f_{E(n)}$ and $g = \lim_n f_{F(n)}$.  Then $\theta(f_{E(n)},f_{F(n)})$ converges in
$L_1$-distance to $\theta(f,g)$, and does so at a uniform rate which is determined by the modulus of uniform
continuity of $\theta$ and the exponential rates of convergence of $\lim_n f_{E(n)}$ and $\lim_n f_{F(n)}$,
as given by Lemma \ref{L_1 quotient sort}(b).  We leave details to the reader.

It is a general fact that every automorphism $\tau$ of a metric structure has a unique extension to an
automorphism of its $\meq$-expansion.  We illustrate this in the present context: given $\M \models Pr$
the probability algebra of $(X,\B,\mu)$ as above, consider an automorphism $\tau$ of $\M$.  Since $\widehat{RV}_n^\M$
is a definable subset of $M^{2^n}$, we see that $\tau$ induces a natural bijection of $\widehat{RV}_n^\M$ onto
itself, by the coordinatewise action, namely $e = (e_1,\dots,e_{2^n}) \mapsto \tau(e) = (\tau(e_1),\dots,\tau(e_{2^n}))$.
Furthermore, if $(e(n) \mid n \geq 1)$ is a coherent sequence in the inverse limit of the spaces $\widehat{RV}_n^\M$,
convergent to the element $[f]_\mu$ of the quotient of $\widehat{RV}^\M$ modulo the $L_1$-distance, 
then $(\tau(e(n))  \mid n \geq 1)$ is also coherent; the image of $[f]_\mu$
under the desired extension of $\tau$ is defined to be the (equivalence class of the) limit of $(\tau(e(n)))_n$.

Another way of looking at this extension process concerns the situation where $a \in \widehat \B$ and $C$ is a
closed subalgebra of $\widehat \B$, and $f \in L_1(\B,\mu;[0,1])$ is a $C$-measurable random variable
representing $\P(a | C)$.  When $\tau$ is an automorphism of $\M$, then $\tau(f)$ as defined above represents
$\P(\tau(a) | \tau(C))$, as can be shown by a routine argument based on the details presented in this section.
Indeed, this fact is easy to show when $C$ is finite (use Lemma \ref{P(A|C) for finite C}) and the 
general case follows using the proof approach for Lemma \ref{approximating P(a|C)} by taking limits.

Finally, we use the operations $\theta$ defined above to 
prove a definability relationship between each random variable $f \in RV(\B,\mu)$ and the smallest 
$\sigma$-subalgebra of $\B$ with respect to which $f$ is measurable, which we will denote by $\sigma(f)$.
Note that $\sigma(f)$ is generated as a $\sigma$-subalgebra by the measurable sets of the form
$f^{-1}(r,1]$ for $r \in [0,1]$.  This implies that 
$\dcl^{\meq}(\sigma(f)) = \dcl^{\meq}(\{ f^{-1}(r,1] \mid r \in [0,1] \})$. 

\begin{lema}
\label{meq-dcl of f}
Let $\M \models Pr$ be the probability algebra of the probability space $(X,\B,\mu)$.
For every $f \in RV(\B,\mu)$ we have $\dcl^{\meq}(f) = \dcl^{\meq}(\sigma(f))$.
\end{lema}
\begin{proof}
First we note that the proof of \ref{L_1 quotient sort}(c) shows that $f \in \dcl^{\meq}(\sigma(f))$.

Thus it remains to show $f^{-1}(r,1] \in \dcl^{\meq}(f)$ for every $r \in [0,1]$.  For $r=1$ this is
trivial.  Fix $0 \leq r<1$ and for each 
$n \geq \frac{1}{1-r}$ let $\theta_n \colon [0,1] \to [0,1]$ be the continuous function 
defined by: $\theta_n(t) = 0$ when $0 \leq t \leq r$, $\theta_n(t) = n(t-r)$ when $r < t < r+ \frac1n$,
and $\theta_n(t) = 1$ for $r + \frac1n \leq t \leq 1$.  
For each $t \in [0,1]$ the sequence $(\theta_n(t))_n$ is monotone increasing in $n$,
and its supremum is the function with values $0$ for $t \leq r$ and $1$ for $t > r$.  Therefore the 
sequence $(\theta_n(f))_n$ converges in $L_1$ to the characteristic function of $f^{-1}(r,1]$, 
which shows that $f^{-1}(r,1] \in \dcl^{\meq}(f)$, as desired.
\end{proof}

\section{Atomless probability spaces}
\label{atomless prob  spaces}

By Proposition \ref{distance to atoms}(e), the fact
that a model of $Pr$ is \emph{atomless} is expressed by the
condition $\theta(1) = 0$, which is equivalent in $Pr$ to the
condition
\begin{equation*}
 \tag{A} \sup_{x}\inf_y|\mu(x\cap y)-\mu(x\cap y^c)|=0 \text{.}
\end{equation*}
We denote by $APA$ the set of axioms $Pr$ together with $(A)$.\index{$APA$}

\begin{coro}\label{axiomatization of atomless prob algs}
Let $\M$ be an $\LLpr$-structure.  Then $\M$ is a model of $APA$
if and only if $\M$ is isomorphic to the probability algebra
of an atomless probability space.
\end{coro}

\begin{proof}
Immediate from the discussion above.
\end{proof}

The main purpose of this section is to give a basic model theoretic
analysis of $APA$.  Many of the results correspond to
things about atomless probability algebras that were proved
by Ben Yaacov in the framework of compact abstract theories
\cite{BY1}.  We bring these results into the setting of continuous
first order logic and give proofs expressed in familiar language of
measure theory and analysis.

In terms of the invariants introduced in \ref{invariants for atoms}, a
model $\M$ of $Pr$ is a model of $APA$ if and only if $\Phi^\M$
consists of the constant sequence with every entry equal to $0$.
Therefore, using results in Section \ref{model theory of prob spaces}
we get the following basic properties of $APA$:

\begin{coro}\label{APA properties}
The theory $APA$ admits quantifier elimination, is separably categorical, 
and is complete, and the unique separable model of $APA$ is
strongly $\omega$-homogeneous.  
Further, $APA$ is the model companion of $Pr$.
\end{coro}
\begin{proof}
Let $APA^*$ denote the theory of all structures $\M^*$, where $\M$
is a model of $APA$ (\emph{i.e.}, $\M$ is an atomless model of
$Pr$).  By Theorem \ref{QE for Pr*} we have that $APA^*$ admits
quantifier elimination.  However, in models of $APA^*$, all of the
extra predicates $P_n$ have the trivial value $0$, and thus can be
eliminated from any formula.  That is, every formula is equivalent
in $APA^*$ to a quantifier-free formula in the language of $Pr$.  It
follows that $APA$ admits quantifier elimination.  Separable categoricity of $APA$ 
and strong $\omega$-homogeneity of the separable model 
follow from Corollary \ref{omega categoricity and homogeneity of all models of Pr}.
Completeness of $APA$ follows from separable categoricity and also from
quantifier elimination (because every model of $APA$ contains the trivial probability
algebra $\{0,1\}$ as a substructure). 

Since $APA$ admits quantifier elimination, it is model complete; also,
it is an extension of $Pr$.  Therefore, to show that $APA$ is the
model companion of $Pr$ it remains only to show that every model of
$Pr$ has an extension that is a model of $APA$.  Let $\M \models Pr$
be the measured algebra of the probability space
$(X,\B,\mu)$ and let $(Y,\C,\nu)$ be any atomless probability space. 
Then the product measure space $(X \times Y, \B
\tensor \C, \mu \tensor \nu)$ is atomless and it can be seen as an
extension of $(X,\B,\mu)$ by the embedding that takes $B \in \B$ to
$B \times Y \in \B \tensor \C$. Therefore the probability algebra of
$(X \times Y, \B \tensor \C, \mu \tensor \nu)$ is a model
of $APA$ into which $\M$ can be embedded.
\end{proof}

\begin{rema}
\label{super-strong homogeneity of separable model of APA}
Let $\M$ be the unique separable model of $APA$, and let $a = (a_1,\dots,a_n), b=(b_1,\dots,b_n) \in M^n$.
From Corollary \ref{APA properties}, if $\tp(a)=\tp(b)$, then there exists an automorphism $\sigma$ of $\M$
such that $\sigma(a_i)=b_i$ for all $i=1,\dots,n$.  However, this result gives no information about the behavior
of $\sigma$ on the rest of $\M$.  In \cite{BHI} we proved a stronger form of homogeneity for $\M$.
In order to state that result, we need to bring in another natural metric on $M^n$, defined by
$$ {d}_P(a,b) := \frac12 \sum_s d(a^s,b^s) \text{.} $$
(The version of this distance that is defined on measurable partitions of $1$ in a probability space was used
in Section 5.  See Lemma \ref{rho_n compared to L_1-distance}.  It also appears below in 
Corollary \ref{corollary for distance {d}_P}.)
Here we are using the notation introduced in Notation \ref{tuples generate partitions} for the partitions of $1$
associated to the tuples $a$ and $b$, and $s$ ranges over $\{ -1,+1 \}^n$.  Note that when $a,b \in M$ are single elements,
${d}_P(a,b) = d(a,b)$, since $\mu(a^c \triangle b^c)=\mu(a \triangle b)$; this is the reason for the $\frac12$ factor in the
definition of ${d}_P$.  Further, when $a$ and $b$ are partitions of $1$, then 
${d}_P(a,b)$ agrees with the definition of ${d}_P(a,b)$ given just before Lemma \ref{rho_n compared to L_1-distance},
since $a^s,b^s$ will both be $0$ unless $s$ contains exactly one occurrence of $+1$; further, if the unique occurrence
of $+1$ is at place $i$, then $a^s=a_i$ and $b^s = b_i$.

The homogeneity result from \cite{BHI} is the following:
\begin{lema*}[5.6 in \cite{BHI}]
Let $a,b\in M^n$ be tuples with $\tp(a)=\tp(b)$. Then there is an automorphism $\sigma$ of $\M$ such 
that $\sigma(a_i)=b_i$ for all $i$ and, for every finite tuple $c\in M^k$, we have
$${d}_P(ac,b\sigma(c))={d}_P(a,b) \text{.}$$
In particular, for every $c \in M$ we have $d(\sigma(c),c) \leq {d}_P(a,b)$.

\end{lema*}

\end{rema}

The following lemma appears in \cite[Section 2.1]{BY1}, in
the framework of compact abstract theories.  To make our
paper more self-contained and because our setting is different, 
and in order to make clear the elementary tools from analysis
from which these facts can be derived, we give complete proofs.

\begin{lema}
\label{types over C - equality}
Let $\M = (\B,\mu,d) \models APA$.  Let $C\subseteq M = \B$
and $a,b \in M^n$.
Recall that $\bra C \ket$ is the $\s$-subalgebra of $\B$ generated by $C$.
The following conditions are equivalent:
\newline
(1) $\tp(a/C)=\tp(b/C)$;
\newline
(2) $\mu(a^s\cap c)=\mu(b^s \cap c)$ for all $s = (k_1,\dots,k_n) \in \{-1,+1\}^n$ and all $c \in \bra C \ket$;
\newline
(3) $\P(a^s | \bra C \ket)=\P(b^s  | \bra C \ket )$ for all $s = (k_1,\dots,k_n) \in \{-1,+1\}^n$.
\end{lema}
\begin{proof}  
(1) $\Leftrightarrow$ (2):  First we deal with the case $C = \emptyset$.  
Only the right to left direction needs to be proved.
By QE for $APA$ (Corollary \ref{APA properties}), showing 
$\tp(a_1,\dots,a_n)=\tp(b_1,\dots,b_n)$ is equivalent to proving
$\mu(t(a_1,\dots,a_n)) = \mu(t(b_1,\dots,b_n))$ holds for every
boolean term $t(x_1,\dots,x_n)$.  By the discussion in \ref{tuples generate partitions},
this holds whenever $\mu(a^s)=\mu(b^s)$ for all $s \in \{-1,+1\}$.

Now consider arbitrary $C$.
From the discussion before Theorem \ref{Radon-Nikodym}, we see 
$C \subseteq \bra C \ket \subseteq \dcl(C)$, 
so each type over $\bra C \ket$ is determined by its restriction to a type over $C$.  
Since $\bra C \ket$ is closed under boolean combinations, the equivalence of (1) and (2) 
follows immediately from the first part of this proof.

(2) $\Leftrightarrow$ (3):  By Corollary \ref{axiomatization of atomless prob algs}, there is an atomless
probability space $(X,\A,\nu)$ whose probability algebra is $(\B,\mu,d)$.  Using
Lemma \ref{closed subalgebras are probability algebras} there exists a $\s$-subalgebra
$\D$ of $\A$ for which $\widehat \D = \bra C \ket$.  According to
\ref{cond prob notation over a sigma-subalgebra of prob alg},
for any $a \in \B = \widehat \A$ we have
$\P(a | \bra C \ket) = \P(a | \D)$.  That is, (3)
is equivalent to the version of (3) in which we have replaced 
$\bra C \ket$ by $\D$.

The equivalence of (2) and this new version of (3) follows from Theorem \ref{Radon-Nikodym};
specifically, from the fact that for any $a = [A]_\mu \in \widehat \A$, the function 
$\P(a|\D)$ is $\D$-measurable and is determined, up to
equality $\mu$-almost everywhere, by the values of $\int_E \chi_A\, d\mu = \mu(A \cap E)$
as $E$ ranges over $\D$.
\end{proof}

\begin{rema}
\label{types of partitions}
The preceding result takes an especially simple form when $a = (a_1,\dots,a_n)$ and
$b = (b_1,\dots,b_n)$ are partitions of $1$ in $\M$.  For example, condition (2) reduces to
$\mu(a_i \cap c) = \mu(b_i \cap c)$ for all $i = 1,\dots,n$ and all $c \in \bra C \ket$.
This means in particular that when $a = (a_1,\dots,a_n)$ is a partition of $1$, the $n$-type
of $a$ over $C$ is determined by the $1$-types $\tp(a_i/C)$ for $i = 1,\dots,n$.
\end{rema}

Lemma \ref{types over C - equality} characterizes \emph{equality} of types.  Next we complete the
description of the type space $S_n(C)$.  First we need a definition:

\begin{defi}
\label{additive functionals}
Let $\M = (\B,\mu,d) \models APA$ and $C\subseteq M = \B$.  An 
\emph{additive functional on $\bra C \ket$}\index{additive functional}
is a finitely additive function  $\lambda \colon \bra C \ket \to [0,1]$.  
\end{defi}

\begin{note}
Let $\M = (\B,\mu,d) \models APA$.  Let $C\subseteq M = \B$ and $a \in M^n$.  In \ref{types over C - equality}(2),
each of the functions $\lambda^s \colon \bra C \ket \to [0,1]$ defined by $c \mapsto \mu(a^s \cap c)$ 
is an additive functional
on $\bra C \ket$ with $\lambda^s(1) = \mu(a^s)$.  Furthermore, for every $c \in \bra C \ket$ we
have $\sum \{ \lambda^s(c) \mid s \in \{-1,+1 \}^n \} = \mu(c)$, since $(a^s \cap c \mid s \in \{-1,+1 \}^n)$
is a partition of $c$.  Our next result shows that any such family of additive functionals 
arises from a type  in $S_n(\bra C \ket)$.
\end{note}

\begin{lema}
\label{types over C - description}
Let $\M = (\B,\mu,d) \models APA$ and $C\subseteq M = \B$.  Assume $\M$ is $\kappa$-saturated,
where $\kappa > \card(\bra C \ket)$.  Let $(\lambda^s \mid s \in \{-1,+1 \}^n)$ be a family of additive functionals
on $\bra C \ket$ such that $\sum \{ \lambda^s(c) \mid s \in \{-1,+1 \}^n \} = \mu(c)$ for all $c \in \bra C \ket$.  
Then there exists $a = (a_1,\dots,a_n) \in M^n$ such that for every $s \in \{-1,+1 \}^n$ and $c \in \bra C \ket$
we have $ \mu(a^s \cap c) = \lambda^s(c)$.  Moreover, $\tp(a/C)$ is determined by these conditions.
\end{lema}
\begin{proof}
Let $x = (x_1,\dots,x_n)$ be a tuple of distinct variables, and let $\Sigma(x)$ be the set of all conditions of the form
$| \mu(x^s \cap c) - \lambda^s(c) | = 0$ as $c$ varies over $\bra C \ket$ and $s$ varies over $\{-1,+1 \}^n$.
We must show that $\Sigma$ is satisfiable in $\M$, and by saturation it suffices to show that $\Sigma(x)$ is
finitely satisfiable.  So let $F$ be any finite subset of $\bra C \ket$ and let $f_1,\dots,f_k$ be the atoms of
$F^\#$. For each $i = 1,\dots,n$, let $(a_{i,s} \mid s \in \{-1,+1 \}^n)$ be a partition of $f_i$ in $\B$ such
that $\mu(a_{i,s}) = \lambda^s(f_i)$ for all $s$.  This is possible because 
$\mu(f_i) = \sum \{ \lambda^s(f_i) \mid s \in \{-1,+1 \}^n \}$.  Note that the family $(a_{i,s})$ is a partition of $1$
in $\B$.  Finally, for each $i$ set $a_i = \cup \{ a_{i,s} \mid s_i = +1 \}$, and set $a = (a_1,\dots,a_n)$.  
An easy calculation shows that for all $i,s$ we have $a^s \cap f_i = a_{i,s}$ and hence 
$\mu(a^s \cap f_i) = \lambda^s(f_i)$.  Additivity implies that $\mu(a^s \cap f) = \lambda^s(f)$ for every
$f \in F$.  This shows that $\Sigma(x)$ is finitely satisfiable in $\M$ and completes the proof (when combined
with Lemma \ref{types over C - equality} to provide uniqueness).
\end{proof}

\begin{rema}
\label{1-types as additive functionals}
The preceding Lemma takes an especially simple form for $1$-types over $C$.  Namely, suppose $\lambda$
is an additive functional on $\bra C \ket$ that satisfies $\lambda(c) \leq \mu(c)$ for all $c \in \bra C \ket$.  Define
$\lambda'$ on $\bra C \ket$ by $\lambda'(c) = \mu(c) - \lambda(c)$.  Then $\lambda'$ is also an additive functional
on $\bra C \ket$, and the pair $\lambda,\lambda'$ satisfies the assumptions in Lemma \ref{types over C - description}.
Therefore, $\lambda$ determines a $1$-type $p_\lambda \in S_1(C)$, and every element of $S_1(C)$ 
can be described in this way.  Specifically, $a$ realizes $p_\lambda$ if and only if $\mu(a \cap c) = \lambda(c)$ for
all $c \in \bra C \ket$ (since this implies $\mu(a^c \cap c) = \lambda'(c)$). 
This observation is especially useful when considering $n$-types of partititions of $1$, 
as discussed in Remark \ref{types of partitions}.  We also use this description of $1$-types in discussing
the model theoretic content of 
Maharam's Lemma in Section \ref{maharam} (Lemma \ref{Maharam's lemma}).

An equivalent approach to $1$-types over $C$ in terms of $\bra C \ket$-measurable, $[0,1]$-valued
functions (i.e., random variables), corresponding to clause (3) in Lemma \ref{types over C - equality}, is the
following: let $f$ be any $\bra C \ket$-measurable, $[0,1]$-valued function.  For $\M \models APA$ and $C \subseteq M$, 
the condition $\P(a|\bra C \ket) = f$ on $a \in M$ is type-definable over $C$; indeed, it precisely determines $\tp(a/C)$.
To see this, consider $\lambda_f$ defined for $c \in \bra C \ket$ by $\lambda_f(c) := \int_c f\,d\mu$.  Then $\lambda_f$
is an additive functional on $\bra C \ket$ that satisfies $\lambda(c) \leq \mu(c)$ for all $c \in \bra C \ket$.  Therefore,
as discussed in the preceding paragraph, $\lambda_f$ exactly determines a $1$-type in $S_1(C)$.  Moreover, the
condition $\mu(a \cap c) = \lambda_f(c)$ for all $c \in \bra C \ket$ is equivalent to $\P(a | \bra C \ket) = f$.
\end{rema}

\begin{lema}\label{dcl} Let $\M = (\B,\mu,d) \models APA$ and $C\subseteq M$.
Then $\dcl(C)=\acl(C)= \bra C \ket$.
\end{lema}
\begin{proof}
Recall that $\bra C \ket$ is the $\s$-subalgebra of $\B$
generated by $C$.
As noted at the beginning of the previous proof, 
$C \subseteq \bra C \ket \subseteq \dcl({C})$.  Therefore, to complete
the proof it suffices to prove $\acl({C}) \subseteq \bra C \ket$.

Now let $a\in M\setminus \bra C \ket$. To show $a\not \in \acl(C)$,
by \cite[Exercise 10.8]{BBHU} it suffices to prove that for some
$\NN\succeq \M$, there is a realization of $\tp(a/C)$ in $\NN$ that is
not in $\M$. Let $\M$ be the probability algebra of 
the probability space $(X,\A,\nu)$. Consider the standard probability space
$([0,1],\mathcal{L},m)$ of Lebesgue measure and form the product space 
$(X\times [0,1],\A\tensor\mathcal{L},\nu \tensor m)$; let
$\NN$ be the probability algebra of this product space. There
is a canonical embedding $J \colon \A \to \A\tensor {\mathcal{L}}$ 
defined by $J(A) = A \times [0,1]$ for $A \in \A$;
this map gives rise to an embedding $\widehat J$ of $\M$ into $\NN$.  Since
$APA$ admits quantifier elimination, $\widehat J$ is an elementary
embedding. As in the previous proof, there is a $\s$-subalgebra
$\D$ of $\A$ for which $\widehat \D = \bra C \ket$.  
Since $\P(a|\D)$ is $\D$-measurable, the set
$A' = \{(x,s) \in X \times [0,1] \mid s \leq \P(a|\D)(x) \}$ is
$\D\tensor\mathcal{L}$-measurable.  We let 
$a' = [A']_{\nu\tensor m} \in \widehat{\D \tensor \mathcal{L}} \subseteq \widehat{\A \tensor \mathcal{L}} = N$.
We  complete the proof by showing that $a'$ is not in the image of $\M$ under 
$\widehat J$ and that $a'$ is a realization of $\tp(a/C)$ in $\NN$.

For the first of these statements, we note that $0 < \P(a|\D) < 1$ holds on a set of positive
measure. Otherwise $\P(a|\D)=\chi_B$ for some $B \in \D$; this would
imply $a = [B]_\nu \in \bra C \ket$, which would contradict our assumptions.  It follows that $A'$
is not of the form $A \times [0,1]$ where $A \in \A$, and thus $a'$ is not in the image of $\M$ under
$\widehat J$.

Finally, let $\D' = \{ B \times [0,1] \mid B \in \D \} = J(\D)$, so $\D'$ is a $\s$-algebra and 
$\widehat{\D'} = \widehat{J}(\bra C \ket) = \bra \widehat{J}({C}) \ket$.
Fubini's Theorem shows that $\P(a'|\D')=\P(\widehat J (a)|\D')$, which
implies by Lemma \ref{types over C - equality} that $\tp_\NN(a'/\widehat{J}({C}))=\tp_\M(a/C)$.
(Here we mean, of course, that the parameters in $\widehat{J}({C})$ are identified 
with those in $C$ via the bijection $\widehat J$.)
\end{proof}

In several results in the rest of this section it is convenient to work in
a $\kappa$-universal domain\index{universal domain} 
for $APA$, where $\kappa$ is uncountable.
For the rest of the section we denote such a model of $APA$ as $\U$.
Recall that a subset $C$ of $U$ is called \emph{small}\index{small set}
 if $\mbox{card}({C}) < \kappa$.
In this situation, every type in $S_n({C})$ is realized in $\U$.
Furthermore, $\U$ is strongly $\kappa$-homogeneous;\index{strongly homogeneous}
 \textit{i.e.}, every
elementary map between small subsets of $U$ extends to an automorphism
of $\U$.  (In applications, $\U$ and the
size of $\kappa$ may need to be changed in order to insure that specific
parameter sets are small.)  

Recall that the metric $d$ on $\U$ yields an \emph{induced metric}%
\index{induced metric (on spaces of types)} 
on each space of types (see
\cite[Section 8]{BBHU}), as follows: when $C\subseteq U$ is small
and $p,q$ are $n$-types over $C$, the \emph{distance} between $p$
and $q$ is defined by $$d(p,q)=\inf\{\max_{1 \leq i \leq n}
d(a_i,b_i): (a_1,\dots,a_n)\models p, (b_1,\dots,b_n)\models q\}
\text{.}$$
\index{$d(p,q)$}

The next result provides an explicit formula for the induced metric on
types of partitions of $1$ in atomless probability algebras. 
(When the parameter set $C$ is empty, this formula occurs as (6.2) in the
proof of \cite[Lemma 6.3]{Shields}.)  

\begin{theo}
\label{distance}
Let $C\subseteq U$
be small and let $a=(a_1,\dots,a_n)$ and 
$b=(b_1,\dots,b_n)$ be partitions of $1$ in $\U$.  Then
\[
d(\tp(a/C),\tp(b/C))= \max_{1\leq i \leq n}\|\P(a_i|\bra C \ket)-\P(b_i|\bra C \ket)\|_1
\]
where $\| \ \ \|_1$ is the $L_1$-norm.

Moreover, there exists $b'=(b'_1,\dots,b'_n)$, a partition of $1$ in $U$, such that
$\tp(b'/C)=\tp(b/C)$ and for all $i=1,\dots,n$
\[
d(a_i,b'_i) = \| \P(a_i | \bra C \ket) - \P(b_i | \bra C \ket) \|_1 \text{.}
\]
\end{theo}
\begin{proof}
Replacing $C$ by $C^\#$ (which is still small)
we may assume throughout this proof that $C$ is a boolean
subalgebra of $\U$.  Since $C \subseteq C^\# \subseteq \dcl({C})$, this does not change
the types being considered nor the distance between them.  Furthermore,
it is obvious that $\bra C^\# \ket = \bra C \ket$, so the right side of the equality to
be proved is also not changed by this move.

We begin the proof by noting that
\[
\| \P(u|\bra C \ket) - \P(v|\bra C \ket) \|_1 \leq \mu(u \triangle v)
\]
for any $u,v \in U$.  Indeed, linearity of the conditional expectation yields
\[
\| \P(u|\bra C \ket) - \P(v|\bra C \ket) \|_1 = \| \P(u \!\setminus\! v|\bra C \ket) - \P(v\!\setminus\! u|\bra C \ket) \|_1 \leq \mu(u \triangle v)
\]
where the last step uses the triangle inequality for the $L_1$-norm and the fact that $\| \P(w|C) \|_1 = \mu(w)$ for any $w \in U$.
By Lemma \ref{types over C - equality}, $\| \P(u|\bra C \ket) - \P(v|\bra C \ket) \|_1$ only depends on $\tp(u/C)$ and
$\tp(v/C)$.  Fixing $i \in \{ 1,\dots,n \}$ and letting $u,v$ range over realizations of $\tp(a_i/C),\tp(b_i/C)$
respectively, and taking the infimums, we obtain
\[
\| \P(a_i|\bra C \ket) - \P(b_i|\bra C \ket) \|_1 \leq d(\tp(a_i/C),\tp(b_i/C)) \text{.}
\]
Taking the maximum over $i$ yields
\[
\max_{1\leq i \leq n}\|\P(a_i|\bra C \ket)-\P(b_i|\bra C \ket)\|_1 \leq d(\tp(a/C),\tp(b/C)) \text{.}
\]
Therefore it remains to show
\[
d(\tp(a/C),\tp(b/C)) \leq \max_{1\leq i \leq n}\|\P(a_i|\bra C \ket)-\P(b_i|\bra C \ket)\|_1 \text{,}
\]
given that $a,b \in U^n$ are partitions of $1$.  We do this in the remainder of the proof.

We first prove this inequality when $C = \emptyset$, noting that the right side of the last inequality
is equal to $\max_{1\leq i \leq n} |\mu(a_i)-\mu(b_i)|$ in this situation.

Let $I = \{ i \mid \mu(a_i) \geq \mu(b_i) \}$ and 
$ J = \{ i \mid \mu(a_i) < \mu(b_i) \}$.  Note that $I \neq \emptyset$;
also, we may assume $J \neq \emptyset$, since otherwise $\tp(a)=\tp(b)$ and so the
inequality to be proved is trivial.  Since $\U$ is atomless, we may choose $b'_i \leq a_i$
in $U$ satisfying $\mu(b'_i) = \mu(b_i)$, for each $i \in I$.
For each such $i$, let $u_i = a_i \setminus b'_i$, and set
$u = \cup\{ u_i \mid i \in I \}$.  Note that $\mu(u) = \sum \{ \mu(a_i) - \mu(b_i) \mid i \in I \}$.
Because $(a_1,\dots,a_n)$ and $(b_1,\dots,b_n)$ are partitions of $1$ in $\U$, it follows that
$\mu(u) = \sum \{ \mu(b_j) - \mu(a_j) \mid j \in J\}$.  Hence we may partition $u$ into
$\{u_j \mid j \in J \} \subseteq U$ such that $\mu(u_j) = \mu(b_j) - \mu(a_j)$
for all $j \in J$.  Finally, for $j \in J$ we set $b'_j = a_j \cup u_j$ and note that 
$(b'_1,\dots,b'_n)$ is a measurable partition of $1$ in $\U$ satisfying $\mu(b'_j) = \mu(b_j)$
for all $j \in J$; in other words, $(b'_1,\dots,b'_n)$ realizes the same type as
$(b_1,\dots,b_n)$.  Moreover, for all $j \in J$ we have
\[
\mu(a_j \triangle b'_j) = \mu(u_j) = | \mu(a_j) - \mu(b'_j) | =  | \mu(a_j) - \mu(b_j) |
\]
which justifies the desired inequality.

Now assume that $C$ is a finite boolean subalgebra of $\U$ and
let the atoms of $C$ be $c_1,\dots,c_p$.   
For each $i\leq n$ and $j\leq p$, let $a_{ij}=a_i\cap c_j$ and 
$b_{ij}=b_i\cap c_j$.  We argue as in the previous
paragraph within each $c_j$.  This yields $b'_{ij}$ for
$i\leq n, j\leq p$ with the following properties: (a) $\mu(b'_{ij})=\mu(b_{ij})$
for all $i,j$; (b) for each $j\leq p$, the tuple $(b'_{1j},\dots,b'_{nj})$ is
a partition of $c_j$; and (c) $\mu(a_{ij}\triangle b'_{ij})= |\mu(a_{ij})-\mu(b_{ij})|$
for all $i,j$.  For each $i\leq n$, let 
$b'_i=\cup_{j\leq m}b'_{ij}$. Then $\tp(b'_1,\dots,b'_n/C)=
\tp(b_1,\dots,b_n/C)$  and 
\[
\mu(a_{i}\triangle b'_{i})=
 \sum_{j\leq m}|\mu(a_{ij})-\mu(b_{ij})|=\sum_{j\leq m}|\mu(a_{i}\cap c_j)-
\mu(b_{i}\cap c_j)|=\|\P(a_i|\bra C \ket)-\P(b_i|\bra C \ket)\|_1 \text{.}
\]

Finally, consider a general algebra $C$. For each $k \geq 1$, use Lemma \ref{approximating P(a|C)}
applied to $C$ and $a_1,\dots,a_n,b_1,\dots,b_n$ to obtain a finite subalgebra $C_k \subseteq C$ such that for
all closed subalgebras $D \subseteq C$ that contain $C_k$ we have
$\| \P(u|C) - \P(u|D) \|_1 \leq 1/k$ for all $u=a_i$ and $u=b_i$ with $1 \leq i \leq n$.  
We may assume $C_k \subseteq C_{k+1}$ for all $k \geq 1$.

Further, we may use properties of type spaces to enlarge each $C_k$ to ensure
additionally for $k \geq 1$ that 
\[ 
|d(\tp(a/C),\tp(b/C)) - d(\tp(a/C_k),\tp(b/C_k))| \leq 1/k \text{.}
\]
Indeed, note that if $E \subseteq D \subseteq C$, then 
$d(\tp(a/E),\tp(b/E)) \leq d(\tp(a/D),\tp(b/D)) \leq d(\tp(a/C),\tp(b/C))$.
Moreover, $d(\tp(a/C),\tp(b/C))$ is the supremum of $d(\tp(a/D),\tp(b/D))$
as $D$ varies over finite subsets of $C$.  (Otherwise there would exist $r < d(\tp(a/C),\tp(b/C))$
such that $d(\tp(a/D),\tp(b/D)) \leq r$ for all finite $D \subseteq C$.  Thus the following set
of conditions would be finitely satisfiable in $\U$:
\[
\Sigma := \{ \vphi(x) = 0 \mid \vphi(x) \in \tp(a/C) \} \cup \{ \psi(y) = 0 \mid \psi(x) \in \tp(b/C) \}
\cup \{ \max_i \ d(x_i,y_i) \leq r \} \text{.}
\]
Since $C$ is a small set, we may choose $x=a'$ and $y=b'$ that realize $\Sigma$ in $\U$.  But then
we would have $\tp(a'/C) = \tp(a/C), \tp(b'/C) = \tp(b/C)$, and $d(a',b') \leq r < d(\tp(a/C),\tp(b/C))$,
which is impossible.)

Putting these two arguments together, we have an increasing family $(C_k \mid k \geq 1)$ of
finite subalgebras of $C$ such that for all $k \geq 1$
\[
\| \P(u|C) - \P(u|C_k) \|_1 \leq 1/k
\]
for all $u=a_i$ and $u=b_i$ with $1 \leq i \leq n$, and
\[
|d(\tp(a/C),\tp(b/C)) - d(\tp(a/C_k),\tp(b/C_k))| \leq 1/k \text{.}
\]
From what is proved earlier for $n$-types over finite algebras, for all $k \geq 1$ we have
\[
d(\tp(a/C_k),\tp(b/C_k) = \max_{1 \leq i \leq n}\|\P(a_i|\bra C_k \ket)-\P(b_i|\bra C_k \ket)\|_1 \text{.}
\]
Taking limits as $k \to \infty$ yields
\[
d(\tp(a/C),\tp(b/C))= \max_{1 \leq i \leq n}\|\P(a_i|\bra C \ket)-\P(b_i|\bra C \ket)\|_1 \text{,}
\]
completing the proof.
\end{proof}

\begin{coro}
\label{distance between 1-types over C}
Let $C\subseteq U$ be small and let $a,b$ be elements of $U$.  Then
\[
d(\tp(a/C),\tp(b/C)) = \|\P(a|\bra C \ket)-\P(b|\bra C \ket)\|_1 \text{.}
\]
\end{coro}
\begin{proof}
Apply the preceding Lemma to $(a,a^c)$ and $(b,b^c)$, and use the fact that each
side of the equation to be proved is unchanged if we replace $a,b$ by $a^c,b^c$
\end{proof}

The definition of the metric ${d}_P$ on $U^n$ that is given in Remark \ref{super-strong homogeneity of separable model of APA} says for $a,b \in U^n$ 
\[
{d}_P(a,b) := \frac12 \sum_s d(a^s,b^s) \text{,}
\]
where $s$ ranges over $\{-1,+1\}$.  Following the established pattern, we can define ${d}_P$ on $S_n(C)$ by
\[
{d}_P(p,q) := \inf \{ {d}_P(a,b) \mid a \models p \mbox{ and } b \models q \} \text{.}
\]
From Theorem \ref{distance} we get immediately

\begin{coro}
\label{corollary for distance {d}_P}
Let $C\subseteq U$
be small and let $a=(a_1,\dots,a_n)$ and 
$b=(b_1,\dots,b_n)$ be in $U^n$.  Then
\[
{d}_P(\tp(a/C),\tp(b/C))= \frac12 \inf \{ \sum_s \|\P(a^s|\bra C \ket)-\P(b^s|\bra C \ket)\|_1 \mid a \models p \mbox{ and } b \models q \} \text{.}
\]
where $\| \cdot \|_1$ is the $L_1$-norm.
\end{coro}
\begin{proof}
This follows from the ``Moreover'' statement in Theorem \ref{distance}.
\end{proof}

Moving beyond types of partitions of $1$, we now discuss the induced metric on the full type space $S_n(C)$ for $APA$.  
For $r \geq 1$, let $S_r^*(C)$ denote the space of $r$-types for
$APA$ that are realized by partitions $(a_1,\dots,a_r)$ of $1$ in
the $\kappa$-universal domain $\U$ for $APA$, where $C$ is a small
subset of $U$. Theorem \ref{distance} gives an explicit formula for
the induced metric on $S_r^*(C)$.  Since $S_r^*(C)$ is a proper, metrically closed
subset of the full space of $r$-types $S_r(C)$, this does not immediately
characterize the metric on all of $S_r(C)$.  However, by
looking at types for $APA$ in the right way, and taking $r=2^n$, we can use this lemma
to characterize the induced metric on $S_n(C)$ up to equivalence of 
metrics, which is enough for most purposes.

To accomplish this, consider the map $\Pi_n \colon S_n(C) \to S^*_{2^n}(C)$
on types that is induced by mapping the type of an arbitrary $n$-tuple $(a_1,\dots,a_n)$
to the type of its associated partition $(a^s \mid s \in \{ -1,+1 \})$ 
(as discussed in \ref{tuples generate partitions}).
Since $APA$ admits quantifier elimination,
$\Pi_n$ is a bijection from $S_n(C)$ onto
$S^*_{2^n}(C)$.  The discussion in Remark \ref{special formulas over Pr}
shows that $\Pi_n$ is also a homeomorphism for the (logic) topologies.
In what follows, we often drop the subscript $n$ when doing so will not cause
confusion.

\begin{lema}
\label{Lipschitz estimates}
Let $C\subseteq U$ be small and let $p,q \in S_n(C)$.  Then
$$  \big( 2^{-n+1}\big) \cdot {d}_n(p,q) \leq {d}_{2^n}(\Pi_n(p),\Pi_n(q)) \leq n \cdot {d}_n(p,q) $$
where ${d}_n,{d}_{2^n}$ denote the induced metrics on the type spaces $S_n(C),S_{2^n}(C)$ respectively
(usually denoted simply by $d$, but here given a subscript to indicate the type space
on which the metric is defined).
\end{lema}
\begin{proof}
This uses an easy calculation based on the description of the bijection between
$n$-tuples $(a_1,\dots,a_n)$ and partitions $(a^s \mid s \in \{-1,+1\})$ that
is given in \ref{tuples generate partitions}.
\end{proof}
Thus $\Pi_n$ is a bi-Lipschitz homeomorphism from $S_n(C)$ onto
$S_{2^n}^*(C)$ with respect to the two induced metrics. 
Since an explicit formula for the induced metric on
$S_{2^n}^*(C)$ is given by Theorem \ref{distance}, this gives us
considerable information about the induced metric topology on all of $S_n(C)$.

Note that this observation strengthens Lemma \ref{types over C - equality}.

We next prove that the theory $APA$ is stable; indeed, we  simply count
types, and show that $APA$ is $\omega$-stable:

\begin{prop}[Prop.\ 4.4, \cite{BY1}]
\label{omega stable}
The theory $APA$ is $\omega$-stable.
\end{prop}
\begin{proof}  
We may take\, $\U$ to be the probability algebra of an atomless
probability space $(X,\A,\nu)$.
Let $C\subseteq U$ be countable.  For each $a \in C$ chose a set $A_a \in \A$
satisfying $a = [A_a]_\nu$ and let $\C$ be the boolean subalgebra of $\A$ generated by 
$\{ A_a \mid a \in C \}$.  Then $\C$ is countable and $\widehat \C = C^\#$.

Let $\SS(C)$ be the set of $\C$-measurable simple functions with
coefficients in $\Q \cap [0,1]$, and let $$\F=\{\tp(a/C)\mid \P(a|\C )\in
\SS(C)\}.$$ Then $\F$ is a countable set of types. By Lemmas \ref{approximating P(a|C)} and 
\ref{distance between 1-types over C}, $\F$ is a metrically dense subset of the space of $1$-types
over $C$.
\end{proof}

\section{Maharam's theorem}
\label{maharam}

Maharam's Theorem is a structure theorem for probability algebras.  It says that a model 
$\M = (\B,\mu,d)$ of $Pr$ is determined up to isomorphism by the information $\Phi^\M$ 
given in Section 4 about the atomic part of $\B$ together with a countable set $\K^\M$ of 
infinite cardinal numbers and a function $\Psi \colon \K^\M \to (0,1]$ whose sum equals the 
$\mu$-measure of the atomless part of $\B$.  Note that $1$ is atomic in $\B$ if and only if 
$\K^\M = \emptyset$.  In general, we know that the atomic part of $\M$ is determined up to 
isomorphism by $\Phi^\M$, as is the measure of the atomless part of $\B$.  
(See Corollary \ref{categoricity for models of Pr}.)  Therefore we may focus our attention on the atomless part of $\M$.
When it is nonzero, it can be considered as a model of $APA$ by rescaling the measure and 
the metric.  That is, to prove Maharam's Theorem, we may focus on models of $APA$.

In this section we give a full discussion of Maharam's Theorem for models of $APA$, 
to make clear the ways in which its proof resonates with ideas from model theory.  

\begin{defi}
Let $\M = (\B,\mu,d) \models APA$ and $0 \neq b \in \B$.  Define $\B\rharp b$ to be the ideal of all $a \leq b$ in $\B$.
\end{defi}
\index{$\B\rharp b$}

\begin{note}
Since we require $b \neq 0$ in the preceding definition, we may regard $\B\rharp b$ as a 
boolean algebra; the interpretations of $\cap$ and $\cup$ as well as of $0$ are inherited from $\B$, 
while $1$ is interpreted as $b$ and the complement operation is taken to be $a \mapsto a^c \cap b$.  
Note that with this understanding of the structure of $\B\rharp b$, the map 
$a \mapsto a \cap b$ is a boolean morphism from $\B$ onto $\B \rharp b$.  We equip $\B\rharp b$ 
with the measure and distance obtained from $\M$ by restriction to $\B \rharp b$; for convenience we 
continue to denote these restrictions by $\mu$ and $d$.  

It is clear that $(\B \rharp b,\mu,d)$ is a \emph{measured algebra}, and that it becomes a model of $APA$ 
if we rescale $\mu$ and $d$ appropriately (namely, multiply by $1/\mu(b)$).  We systematically use this point of view below.
\end{note}

\begin{nota}
Unless otherwise specified, in the rest of this section we take $\M = (\B,\mu,d)$ to be a model of $APA$.
When we refer to the \emph{density} of a subset of $\B$, we mean the metric density.
\end{nota}

A key quantity for the arguments behind Maharam's Theorem is the density of $\B \rharp b$; for
brevity we  also refer to this as the \emph{density of $b$}.  
When $b$ is atomless, this density is an infinite cardinal number.

\begin{defi}
\label{homogeneous}
For $b \in \B$, we say \emph{$\B \rharp b$ is homogeneous} and (alternatively) \emph{$b$ is homogeneous}
if $b \neq 0$ and $\B\rharp a$ has the same density 
as $\B\rharp b$ for every $0 \neq a \leq b$.\index{homogeneous!element}
\end{defi}

We are now in position to define the \emph{Maharam invariants}\index{Maharam invariants}
 $(\K^\M,\Psi^\M)$ for a model $\M = (\B,\mu,d)$ of $APA$.

\begin{defi}
\label{maharam invariants}
Define $\K^\M$ to be the set of all infinite cardinal numbers $\kappa$  for which there exists $b \in \B$ such that $b$ is 
homogeneous and the density of $\B \rharp b$ is $\kappa$.  For each $\kappa \in \K^\M$ define
\[
\Psi^\M(\kappa) := \sup \{ \mu(b) \mid b \mbox{ is homogeneous and the density of } \B\rharp b = \kappa \} \text{.}
\]
\index{$\K^\M$, $\Psi^\M$}
We call $b \in \B$ \emph{maximal homogeneous} if $b$ is homogeneous and $\mu(b) = \Psi^\M(\kappa)$, where
$\kappa = \mbox{ density of } \B\rharp b$.  
We call $\B$ \emph{homogeneous} if $1$ is homogeneous in $\B$.\index{homogeneous!algebra}

We say $\M$ \emph{realizes} its Maharam invariants if there exists a family $(b_\kappa \mid \kappa \in \K^\M)$ of
pairwise disjoint maximal homogeneous elements of $\B$ such that $b_\kappa$ has density $\kappa$ for every
$\kappa \in \K^\M$ and $\sum \{ \mu(b_\kappa) \mid \kappa \in \K^\M \}$ exists and equals $1$.
\end{defi}

We  show below that every model $\M$ of $APA$ realizes its Maharam invariants.  In particular, this means that
$\K^\M$ is nonempty and countable.

\begin{note}
\label{density from Maharam invariants}
If $\M \models APA$ realizes its  Maharam invariants, then the density of $\M$ is the 
supremum of $\K^\M$ (taken in the cardinal numbers).
\end{note}

\begin{exam}
\label{homogeneous example}
Obviously the unique separable model $\M$ is homogeneous of density $\aleph_0$.

Let $(X,\B,\mu)$ be any countably generated, atomless probability space, and let $\kappa$ be any 
uncountable cardinal number.  Let $\A$ be the probability algebra of the product space $X^\kappa$ with the product
probability measure obtained by taking $\mu$ as the measure on each factor.  
Then $\A$ is homogeneous and has density $\kappa$.
\end{exam}
\begin{proof}
Let $\S$ be a countable dense subset of $\B$.  For each $\alpha < \kappa$ let $\pi_\alpha$ be the coordinate
projection from $X^\kappa$ onto $X$.  The $\sigma$-algebra of product-measurable subsets of $X^\kappa$
is generated by the sets of the form $\pi^{-1}_\alpha(Q)$, where $\alpha < \kappa$ and $Q \in \S$.  Therefore
$\A$ has density at most $\kappa$.  Also, if $Q \in \B$ has $\mu(Q) = r \in (0,1)$ and $\alpha,\beta$ are distinct, then
$d(\pi^{-1}_\alpha(Q),\pi^{-1}_\beta(Q)) = 2r(1-r)>0$, so $\A$ has density at least $\kappa$.

If $V$ is any product-measurable subset of $X^\kappa$, then $V$ only depends on countably many ordinals
$\alpha < \kappa$, in the sense that there is a countable set $S$ of such ordinals such that for any $u,v \colon \kappa \to X$,
if $u \in V$ and $u(\alpha) = v(\alpha)$ for all $\alpha \in S$, then also $v \in V$.  
When $V,S$ satisfy this condition, we say \emph{$V$ depends only on the coordinates in $S$}.
(Note that the collection
of product measurable $V \subseteq X^\kappa$ that only depend on countably many $\alpha < \kappa$ is a $\sigma$-algebra, and it contains all sets of the form $\pi^{-1}_\alpha(Q)$, where $\alpha < \kappa$ and $Q \in \S$.)

A variant of the argument in the first paragraph shows that
the restriction of $\A$ to the event determined by any product-measurable set $V$ also has density equal to $\kappa$.  (Just work on the coordinates in $\kappa \setminus S$, where $S$ is countable and $V$ depends only on the coordinates in $S$.) Therefore $\A$ is homogeneous of density $\kappa$.
\end{proof}

\begin{rema}
\label{realizing every maharam invariant}
It is now clear that for every nonempty countable set $\K$ of infinite cardinal numbers 
and every function $\Psi \colon \K \to (0,1]$
whose sum equals $1$, we can construct an atomless probability space 
$(X,\B,\mu)$ whose probability algebra $\M$ realizes its Maharam invariants and such that 
$\K^\M = \K$ and $\Psi^\M = \Psi$.  For each $\kappa \in \K$, let $(X_\kappa,\B_\kappa,\mu_\kappa)$ be a probability space whose probability algebra is homogeneous of density $\kappa$; take the sets $X_\kappa$ to be 
pairwise disjoint.  For each $\kappa \in \K$, let $\mu'_\kappa$ be $\Psi(\kappa)\mu$.  Then take $X$ to be the union of
$(X_\kappa \mid \kappa \in \K)$ and let $\B$ be the $\sigma$-algebra of subsets of $X$ generated by $\cup \{ \B_\kappa \mid \kappa \in \K \}$.
Note that each $Q \in \B$ is equal to $\cup \{ Q \cap X_\kappa \mid \kappa \in \K \}$, and set 
$\mu(Q) := \sum \{ \mu'_\kappa(Q \cap X_\kappa) \mid \kappa \in \K \}$.  Then it is clear that $(X,\B,\mu)$ is an atomless
probability space and that its probability algebra $\M$ satisfies $(\K^\M,\Psi^\M) = (\K,\M)$.
\end{rema}

Next we state a lemma giving properties of homogeneous elements.

\begin{lema}
\label{properties of homogeneous elements}
Let $\M = (\B,\mu,d) \models APA$.
\begin{itemize}
\item[(a)]  If $b_1,b_2$ are homogeneous elements of $\B$, and if $b_1,b_2$ have different
densities, then $b_1 \cap b_2 = 0$.
\item[(b)]  If $b_n$ is a homogeneous element of $\B$ for $n\geq 1$, and the density of $b_n$ is $\kappa$ for all $n$, then
$b = \cup \{ b_n \mid n \geq 1 \}$ is also homogeneous in $\B$ and $b$ has density $\kappa$.
\item[(c)]  If there exists a homogeneous element $b \in \B$ of density $\kappa$, then there exists a
 maximal homogeneous element $b'$ such that $b'$ also has density $\kappa$.
 \item[(d)]  If $b'$ is a maximal homogeneous element of density $\kappa$, then every homogeneous
 element $b$ of density $\kappa$ satisfies $b \leq b'$.
\end{itemize}
\end{lema}
\begin{proof}
Left as exercises for the reader.
\end{proof}

\begin{prop}
\label{models realize their invariants}
Every model $\M =(\B,\mu,d)$ of $APA$ realizes its system of Maharam invariants.
\end{prop}

\begin{proof}
We refer to the items in Lemma \ref{properties of homogeneous elements} by their letters.
Let $\K^\M$ be defined as in Definition \ref{maharam invariants}.  Note that $\K^\M$ is nonempty, since taking
$0 \neq b \in \B$ such that $b$ has the least possible density implies that $\B\rharp b$ is homogeneous.
For each $\kappa \in \K^\M$, let $b_\kappa$
be a maximal homogeneous element of $\B$ that has density $\kappa$, which exists by (c).  By (a) the elements
$(b_\kappa \mid \kappa \in \K^\M)$ are pairwise disjoint in $\B$ and by (d) we have
$\mu(b_\kappa) = \Psi^\M(\kappa)$ for every $\kappa$.  Note that this implies that $\K^\M$ is countable.

It remains to show that $\sum \{ \mu(b_\kappa) \mid \kappa \in \K^\M \} = 1$.  If not, let $b$ be the complement
in $\B$ of $\cup \{ b_\kappa \mid \kappa \in \K^\M \}$, so $b > 0$.  Let $b'$ be a nonzero element of $\B\rharp b$ of
least possible density, so $\B\rharp b'$ is homogeneous.  If $\kappa$ is the density of $b'$, then $\kappa \in \K^\M$ by
definition, and we have that $b' \cap b_\kappa = 0$.  This contradicts the maximality of $b_\kappa$.
\end{proof}

\begin{note}
\label{path to proving Maharam's Theorem}
It remains to show that a model $\M$ of $APA$ is determined up to isomorphism by its Maharam invariants.
Evidently it suffices to prove the special case that when $\M,\NN$ are homogeneous models and have the same density,
then $\M \cong \NN$.  Indeed, if $\M$ is any model of $APA$ and the family $(b_\kappa \mid \kappa \in \K^\M)$
witnesses that $\M$ realizes its Maharam invariants (as in the proof of 
Proposition \ref{models realize their invariants}),
then the isomorphism type of each $\B\rharp b_\kappa$ (as a measured algebra) would be determined by $\kappa$
and $\mu(b_\kappa) = \Psi^\M(\kappa)$.  The isomorphism type of $\M$ is easily reconstructed from this data,
since $\K$ is countable, the elements $(b_\kappa \mid \kappa \in \K^\M)$ are pairwise disjoint, and
$\sum \{ \mu(b_\kappa) \mid \kappa \in \K^\M \} = 1$.

A similar discussion applies to arbitrary models $\M = (\B,\mu,d)$ of $Pr$.  In this case the necessary decomposition
of $\B$ consists of a family $(b_i \mid i \in I)$ of elements of $\B$ and a family $(\kappa_i \mid i \in I)$ 
of cardinal numbers satisfying the
following conditions: (i) the elements $b_i$ are pairwise disjoint and nonzero; 
(ii) $\sum \{ \mu(b_i) \mid i \in I \} = 1$; (iii) if $\kappa_i$ is finite, it equals $1$ and $b_i$ is an atom in $\B$;
(iv) if $\kappa_i$ is infinite, then $b_i$ is a maximal homogeneous component of the atomless part of $\B$ of
density $\kappa_i$; and (v) if $\kappa_i,\kappa_j$ are infinite with $i \neq j$, they are distinct.  As we  show now,
the additional information needed to determine $\M$ up to isomorphism is the family 
$(\mu(b_i) \mid i \in I)$ of real numbers, which all come from $(0,1]$ and whose sum is $1$.
\end{note}

What remains to be proved is that every homogeneous model of $APA$ is determined  
up to isomorphism by its density.  It is in this proof where
model theoretic ideas come into play, as we explain next.  Indeed, the homogeneous models of $APA$ are the
same as the saturated models (i.e., the models that have density $\kappa$ and are $\kappa$-saturated, for some $\kappa$).  
To make this connection precise requires the introduction of the following notion.

\begin{defi}
\label{atomless over}
Let $(\B,\mu,d)$ be a probability algebra and let $\A$ be a $\sigma$-subalgebra of $\B$. 
A non-zero element $b \in \B$ is called an \emph{atom relative to $\A$}\index{atom!relative}
 if for all $b' \leq b$ 
in $\B$ there is $a \in \A$ such that $b' = a \cap  b$.
We say that $\B$ is \emph{atomless over $\A$}%
\index{atomless over}
 if no nonzero element $b \in \B$ is an atom relative to $\A$. 
\end{defi}

\begin{rema}
\label{characterization of atom relative to A}
Consider the setting of Definition \ref{atomless over} and let $b$ be a nonzero element of $\B$.  Then $b$ is an atom
relative to $\A$ if and only if  $\P(b' \mid \A)$ is equal to a restriction of $\P(b \mid \A)$, for every $b' \leq b$ in $\B$.
Here we are considering each $\P(\cdot \mid \A)$ as a $\mu$-ae equivalence class of
$\A$-measurable $[0,1]$-valued functions, and ``restriction'' means to multiply by the characteristic
function of an $\A$-measurable set.  (See Notation \ref{cond prob notation over a sigma-subalgebra of prob alg}.)
\end{rema}

Note that if $\B$ is a probability algebra and $0 \neq b \in \B$, then $b$ is an atom in $\B$ 
if and only if $b$ is an atom relative to the trivial subalgebra $\{ 0,1 \}$. 

\begin{lema}
\label{homogeneous implies atomless over}
Suppose $\M = (\B,\mu,d)$ is a homogeneous model of $APA$ and its density is $\kappa$, and $\NN = (\A,\mu,d)$ is
a substructure of $\M$ of density $<\kappa$.  Then $\B$ is atomless over $\A$.
\end{lema}
\begin{proof}
For each nonzero $b \in \B$, the density of $\B\rharp b$ is $\kappa$, whereas the density of $\{a \cap b \mid a \in \A \}$
is at most the density of $\A$.
\end{proof}

\begin{lema}
\label{add elements to atomless over}
Suppose $\M = (\B,\mu,d)$ is a model of $APA$, and  $\NN = (\A,\mu,d)$ is a substructure of $\M$.
If $\B$ is atomless over $\A$, then $\B$ is atomless over $\bra \A \cup F \ket$ for every finite set $F \subseteq \B$.
\end{lema}
\begin{proof}
Using induction, it suffices to consider the case $F = \{b\}$.  We prove the contrapositive.  Suppose there is a nonzero
$b'$ in $\B$ that is an atom relative to $\bra \A \cup \{b\} \ket$.  We show that $b' \cap b$ is either $0$ or an atom
relative to $\A$, and the same for $b' \cap b^c$.  Since $b' \neq 0$, at least one of them must be an atom relative to $\A$.

Consider $b'' \leq b' \cap b$ (the case of $b' \cap b^c$ is similar).  Note that $b'' \leq b$, so $b'' \cap b^c =0$.
Since $b'$ is an atom relative to $\bra \A \cup \{b\} \ket$, there exists $x \in \bra \A \cup \{b\} \ket$ with
$b'' = x \cap b'$.  There exist $a_1,a_2 \in \A$ such that $x = (a_1 \cap b) \cup (a_2 \cap b^c)$, and therefore
\[
b'' = x \cap b' = (a_1 \cap b \cap b') \cup (a_2 \cap b^c \cap b') = a_1 \cap (b' \cap b) \text{.}
\]
The last equality is because $b''$ and $a_2 \cap b^c \cap b'$ are disjoint, so $a_2 \cap b^c \cap b'= 0$.
It follows that $b' \cap b$ is either $0$ or an atom relative to $\A$.
\end{proof}

\begin{lema}[Maharam's lemma]\index{Maharam's lemma}
\label{Maharam's lemma}
Let $\M = (\B,\mu,d)$ be a model of $APA$ and let $\NN = (\A,\mu,d)$ be a substructure of $\M$.  If $\B$ is atomless
over $\A$, then $\M$ realizes every $n$-type over $\A$.
\end{lema}
\begin{proof}
Using Lemma \ref{add elements to atomless over} and the fact that it allows us to realize $n$-types 
over $\A$ ``coordinate by coordinate'', it suffices to prove the result for $1$-types.  
Remark \ref{1-types as additive functionals} implies that proving
$\M$ realizes every $1$-type over $\A$ is equivalent to proving the following statement:

Suppose $\lambda \colon \A \to [0,1]$ is an additive functional over $\A$ such that $\lambda(a) \leq \mu(a)$ holds
for every $a \in \A$.  Then there exists $b \in \B$ such that $\lambda(a) = \mu(a \cap b)$ for every $a \in \A$.

A proof of exactly this statement is given as Lemma 3.2 in Fremlin's chapter \cite{Fr-handbook} 
on measure algebras, and also as
Lemma 331B in volume 3 \cite{Fr-treatise} of his multi-volume treatise on measure theory.
\end{proof}

\begin{coro}
\label{uniqueness for homogeneous models of APA}
Every homogeneous model of $APA$ is determined up to isomorphism by its density.
\end{coro}
\begin{proof}
Suppose $\M = (\B,\mu,d) \models APA$ has density $\kappa$ and is homogeneous.  Since $\B$ is 
homogeneous, it is atomless over $\bra C \ket$ for every $C \subseteq \B$ with $\card(C) < \kappa$.  
By Lemma \ref{Maharam's lemma}, $\M$ realizes every $n$-type over $C$ for every such $C$.  
That is, $\M$ is a $\kappa$-saturated model of $APA$ and it has density $\kappa$.
Using the standard back-and-forth argument from model theory, any two such models are isomorphic.
\end{proof}

Finally, we have Maharam's Theorem,\index{Maharam's Theorem}
 which characterizes the structure of all probability algebras up to isomorphism.

\begin{theo}
\label{Maharam's Theorem for models of Pr}
Every model $\M$ of $Pr$ is determined up to isomorphism by its invariants $\Phi^\M$ for the atomic part and its
Maharam invariants $(\K^\M,\Psi^\M)$ for the atomless part.
\end{theo}
\begin{proof}
The definition of the Maharam  invariants for general probability algebras
 is in the first paragraph of this section; the definition of $\Phi^\M$ is 
in Section \ref{model theory of prob spaces}.  The proof of the Theorem is given above, with the 
key result being Corollary \ref{uniqueness for homogeneous models of APA}, which handles the 
maximal homogeneous components of the atomless part of $\M$.  Note \ref{path to proving Maharam's Theorem} 
indicates how the structure of $\M$ is determined by what these invariants say about its component parts.
\end{proof}

Note that for each infinite cardinal $\kappa$, we identified the $\kappa$-saturated model of $APA$ of 
density character $\kappa$ as the Maharam homogenous model of density $\kappa$. More information 
on $\kappa$-saturated and $\kappa$-homogeneous models of $APA$ can be found in \cite{Song2}.

The following characterization of the ``atomless over'' property is often useful:

\begin{prop}
Let $\M = (\B,\mu,d)$ be a model of $APA$ and let $\A$ be a $\sigma$-subalgebra of $\B$.  The following are
equivalent.
\begin{enumerate}
\item[(a)] $\B$ is atomless over $\A$.
\item[(b)] For an infinite set of positive integers $n$, there is in $\B$ a partition of $1$, say $u=(u_1,\dots,u_n)$, such
that $\mu(a \cap u_i) = \frac1n  \mu(a)$ for all $a \in \A$ and $i=1,\dots,n$.  (In other words, each $u_i$ satisfies
$u_i \pindep\A$ and has measure $\frac1n$.)
\item[(c)]  There is an atomless $\sigma$-subalgebra $\C$ of $\B$ such that $\A \pindep\C$.
\end{enumerate}
\end{prop}
\begin{proof}
(a) $\Rightarrow$ (c): We build inductively a sequence $\{\C_n\}_{n\geq 1}$ of finite subalgebras of $\B$ such that 
for all $n \geq 2$ we have
$\C_n \pindep(\A\cup (\bigcup_{i<n}\C_i))$ and 
$\C_n$ is generated by a partition of $1$, say $(u_1,\dots,u_n)$, such that 
$\mu(u_i) = \frac1n$ for all $i=1,\dots,n$.  We take $\C_1=\{0,1\}$.  
Assume we have built $\{\C_i\}_{i<n}$.  By Lemma \ref{add elements to atomless over} the algebra $\B$ is atomless over 
$(\A\cup (\bigcup_{i<n}\C_i))^{\#}$. The existence of 
$\C_n$ follows from Lemma \ref{Maharam's lemma}, since we can describe the properties of $(u_1,\dots,u_n)$
by formulas over $(\A\cup (\bigcup_{i<n}\C_i))^{\#}$. Now let $\C$ be the $\sigma$-algebra 
generated by $\bigcup_{i\geq 1}\C_i$.

(c) $\Rightarrow$ (b): This is immediate, since for any $n$ there exists a partition of $1$, say $(u_1,\dots,u_n)$, 
in $\C$ with $\mu(u_i)= \frac1n$ for all $i$, and $u_i \pindep\A$ automatically for all $i$.

(b) $\Rightarrow$ (a):  Let $0 \neq b \in \B$.  Let $(u_1,\dots,u_n)$ be a partition of $1$ in $\B$ as in (b)
such that $\frac1n \leq \frac12 \mu(b)$.  Note that
\[
\P(u_i \cap b \mid \A) \leq \P(u_i | \A) \leq \tfrac1n \leq \tfrac12 \mu(b) \text{.}
\]
Therefore, for some $i$ we have that on a set of positive $\mu$-measure 
\[
0 < \P(u_i \cap b \mid \A) < \P(b \mid \A) \text{,}
\]  
which means that $u_i \cap b$ is not of the form $a \cap b$ with $a \in \A$.  Thus $b$ is not an atom relative to $\A$.
(See Remark \ref{characterization of atom relative to A}.)
\end{proof}

\section{Stability of APA}
\label{stability of APA}

In this section we continue our study of the theory $APA$, concentrating on
stability-theoretic properties.  Throughout this section we work in a $\kappa$-universal domain
for $APA$, which is denoted by $\U$, with underlying set $U$.
A subset of $U$ is \emph{small} if its cardinality is $< \kappa$.  Unless otherwise specified,
we take parameter sets always to be small subsets of $U$.  We adjust $\kappa$ as
needed for specific models of $APA$ to be substructures of $\U$.

This section uses background on stability, forking, definitions of types, and canonical bases 
that can be found in \cite{BBH} and  \cite{BU}.

Since $APA$ is stable, by Proposition \ref{omega stable}, we have the relation of 
\emph{model theoretic independence},\index{independence!model theoretic}
 denoted $C \indep_E D$, defined for small sets $C,D,E \subseteq U$ by: 
\index{$\indep$}
\newline
\centerline{\emph{$ C \indep_E D$ if and only if $\tp(a/DE)$ does not fork over $E$ for all finite tuples
$a$ from $\bra C \ket$.}} 
\newline
Our next result is that model theoretic independence is exactly
the same as probabilistic independence, 
from which we also get a quantitative criterion for non-forking in $APA$.  The corresponding result
in the CAT setting was proved in \cite[Theorem 2.10]{BY1}.  Our proof uses the
same general approach, with details based on properties of conditional expectation
and $\pindep$ that are discussed in Section \ref{probability spaces}.

The argument follows a familiar pattern: prove that in models of $APA$ the relation $\pindep$
satisfies invariance, symmetry, finite character, transitivity, extension, and
local character, and also that types of tuples over arbitrary
sets are stationary.  From this one gets that $APA$ is stable and that
$\indep$ is the same as $\pindep$ (see  \cite[Theorem 14.14]{BBHU}).  
Throughout the proof we use the results from 
Lemma \ref{characterization of independence}.

\begin{theo}\label{pr:non-div}
Let $C,D,E \subseteq  U$ be small.  Then
$$ C \indep_E D \mbox{ if and only if } C \pindep_E D \text{.} $$

Consequently, for every $c = (c_1,\dots,c_n) \in U^n$, we have that $\tp(c/DE)$ does not fork over $E$ if and only if
\[
\P(c^s | \bra DE \ket)
=\P(c^s | \bra E \ket)
\]
for all $s = (k_1,\dots,k_m) \in \{-1,+1\}^n$. 
\end{theo}
\begin{proof}  
Let $a$ be a finite tuple from $U$, and let $A,C,D,E$ be small subsets of $U$.  We prove each
of the conditions invariance, symmetry, finite character, transitivity, extension, and local character, 
and also that types of tuples over arbitrary sets are stationary.  (As we verify each condition, we make
clear what it means.)

Invariance, Symmetry, and Finite Character:  it is obvious from Definition \ref{conditional independence} that the relation $C \pindep_E D$ is invariant under automorphisms of $\U$ and equivalent to $D \pindep_E C$.  Finite character requires that
$C \pindep_E D$ holds if and only if $c \pindep_E D$ holds for every finite tuple $c$ from $C$. This follows from the
definition using the disjoint additivity of the conditional expectation operators (over $\bra D \ket$ and over $\bra DE \ket$).  
Indeed, if $c_1,\dots,c_k$ are the atoms in $c^\#$, then $c \pindep_E D$ iff for every $j=1,\dots,k$ we have $c_j \pindep_E D$.  

Transitivity:  This condition says that $a \pindep_E CDE$ if and only if $a \pindep_{CE} CDE$ and $a \pindep_E CE$.  
By Lemma \ref{characterization of independence}((i) $\Leftrightarrow$ (ii)), this statement is equivalent to the 
statement $\P(a| \bra CDE \ket) =  \P(a| \bra E \ket)$ if and only if
$\P(a| \bra CDE \ket) = \P(a| \bra CE \ket)$ and $\P(a| \bra CE \ket) = \P(a| \bra E \ket)$, which is true by Fact \ref{droponnorm}.
Applying this for $a$ ranging over $A^\#$ proves transitivity for $A \pindep_E CD$, using finite character.

Extension: We need to show that for all small subsets $A,C,D$ of $U$, there is a copy $E$ of $A$ over
$C$ such that $E \pindep_C D$.  By a ``copy'' we mean that there is a bijection $f$ from $A$ onto $E$ such
that for every $a \in A^n$, every $\LLpr$-formula $\vphi(x;y)$, and every finite tuple $c$ from $C$, 
the values of $\vphi(f(a);c)$ and $\vphi(a;c)$ in  $\U$ are equal.  (In short: $\tp(A/C)=\tp(E/C)$,
with $f$ giving the correspondence between enumerations of $A$ and $E$.  Otherwise said, $f$ is an 
elementary map over $C$, from $A$ onto $E$.)  Thus
the statement that $E$ is a copy of $A$ over $C$ is expressed by a family of $\LLpr$-conditions in
$\card(A)$ many variables and in parameters from $C$. The same is true of the condition
$E \pindep_C D$, except that the parameters come from $C \cup D$.  (Namely, for each $a = (a_1,\dots,a_m) \in A^m$
and the corresponding $e = (f(a_1),\dots,f(a_m))$, for each $u = (u_1,\dots,u_n) \in D$ and for each
$s \in \{-1,+1\}^m$ and $t \in \{-1,+1\}^n$, 
we require $\P(e^s \cap u^t | \bra C \ket) = \P(a^s  | \bra C \ket) \cdot \P(u^t | \bra C \ket)$.
By the second paragraph of Remark \ref{1-types as additive functionals}, this is a type-definable condition
over $CD$, since $\P(a^s  | \bra C \ket) \cdot \P(u^t | \bra C \ket)$ is $\bra C \ket$-measurable.  
Note that taking $u^t=1$, this independence condition already implies 
$\P(e^s | \bra C \ket) = \P(a^s  | \bra C \ket))$, which is equivalent to $\tp(e/C)=\tp(a/C)$ by
Lemma \ref{types over C - equality}.)

Since $\U$ is $\kappa$-saturated and the conditions discussed above involve $< \kappa$
many formulas, it suffices to show that this set of conditions is finitely satisfiable in $\U$.  In particular,
we may assume $A$ is finite.  

The rest of the argument is based on Lemma \ref{extension}.  Let $(X,\B,\mu)$ be a probability space whose
probability algebra $(\widehat {\B},\widehat \mu, \widehat d)$ is a small elementary substructure of $\U$ that contains $ACD$.  Lemma \ref{extension}
yields a probability space $(X',\B',\mu)$ and a map $B \mapsto B'$ that is a measure-preserving boolean 
embedding of $\B$ into $\B'$.  The construction used in proving Lemma \ref{extension} ensures that
the probability algebra of $(X',\B',\mu')$ is small.  Using the fact that $\U$ is $\kappa$-saturated and strongly
$\kappa$-homogeneous, as well as the fact that $APA$ has $QE$, we may realize 
$(\widehat{\B'},\widehat {\mu'}, \widehat {d'})$ as a substructure of $\U$, and
ensure that the induced embedding of $(\widehat {\B},\widehat \mu, \widehat d)$ 
into $(\widehat{\B'},\widehat {\mu'}, \widehat {d'})$ is an inclusion.  Therefore we obtain 
in $\widehat{\B'}$ a copy $E$ of $A$ over $C$ such that $E \pindep_C D$, as desired.

Local Character:  We need to show that there exists a countable set $C' \subseteq C$ such that 
$a \pindep_{C'} C$.   Let $C'$ be any countable set such that $\P(a|\bra C \ket)$ is $\bra C' \ket$-measurable.
Then $\P(a|\bra C' \ket) = \P(a|\bra C \ket)$, which implies $a \pindep_{C'} C$ by 
Lemma \ref{characterization of independence}((i) $\Leftrightarrow$ (iii)).

Stationarity of Types:  We need to show that if $a \pindep_E D$, then $\tp(a/DE)$ is uniquely determined by 
$\tp(a/E)$.  So assume $a,b \in U^n$ satisfy $a \pindep_E D$, $b \pindep_E D$, and $\tp(a/E) = \tp(b/E)$.
So for each $s \in \{ -1,+1 \}^n$ we have $a^s \pindep_E D$ and $b^s \pindep_E D$ by Definition \ref{conditional independence}, 
and $\tp(a^s/E) = \tp(b^s/E)$ by Lemma \ref{types over C - equality}.  
It follows for all $s$ that
\begin{align*}
d(\tp(a^s/DE),\tp(b^s/DE)) = \| \P(a^s & | \bra DE \ket) - \P(b^s | \bra DE \ket) \|_1 \\
= \| \P(a^s | \bra E \ket) - \P(b^s | \bra E \ket) \|_1 & = d(\tp(a^s/E),\tp(b^s/E)) = 0
\end{align*}
by Corollary \ref{distance between 1-types over C} and Lemma \ref{characterization of independence}.  
Therefore $\tp(a/DE)=\tp(b/DE)$ by Lemma \ref{types over C - equality}.
\end{proof}

\begin{rema}
\label{superstable}
Since $APA$ is $\omega$-stable (see Proposition \ref{omega stable}), it follows that $APA$ is also
superstable, by \cite[Remark 14.8]{BBHU}.  In fact, $APA$ has a property
that is analogous, in the continuous logic setting, to the classical
property of being superstable of finite $SU$-rank.  To see this, 
take $\epsilon>0$ and
small sets $D\subseteq C\subseteq U$. Say that $\tp(a/C)$ \emph{$\
\epsilon$-forks over $D$} if $d(\tp(a/C),\tp(a'/C))\geq \epsilon$,
where $\tp(a'/C)$ is the (unique) non-forking extension of
$\tp(a/D)$. Let $SU_{\epsilon}(\tp(a/D))$ be the foundation rank of
$\tp(a/D)$ for this relation of $\epsilon$-forking. Then for any
$\epsilon>0$, $a\in U$ and small $D\subseteq U$, it can be shown
using Fact \ref{droponnorm}, and Theorems \ref{distance} and \ref{pr:non-div}
that $SU_{\epsilon}(\tp(a/D))$ is at most $(1/\epsilon)^2$.
\end{rema}

The next result shows that $APA$ has built-in canonical bases.\index{canonical bases}  
Before getting into the details,
we provide some intuition about the connection between canonical bases and conditional probabilities. 
Consider the case where $x$ is a single variable and $p(x)=\tp(a/C)$, where $C$ is a small closed subalgebra of $\U$.
Let $(a_i)_{i \in \N}$ be a Morley sequence in $p$ and consider $(\chi_{a_i})_i$ as elements of $L^2(U,\mu)$. 
Let $H$ be the Hilbert subspace corresponding to $L^2(C,\mu)$, that is, the collection of  elements 
of $L^2(U,\mu)$ that are $C$-measurable, and let $P_H$ be the orthogonal projection operator from
$L^2(U,\mu)$ onto $L^2(C,\mu)$. Then, for each $i\in \N$, we may write 
$\chi_{a_i}=P_{H}(\chi_{a_i})+v_i$ where $\{v_i\}_i$ are pairwise orthogonal and they all have the same norm. 
Then the sequence of averages $\big(\Sum_{i=1}^n\frac{P_{H}(a_i)+v_i}{n}\big)_n$ converges 
(in $L^2(U,\mu)$) to $P_{H}(a_i)=\P(a|C)$, so $\P(a|C)\in \dcl(\{a_i\}_i)$. A similar computation 
can be carried out using any Morley sequence in a type parallel to $p$, so
$\P(a|C)$ belongs to the definable closure of the parallelism class of $p$ and thus 
$\P(a|C)\in \dcl^{\meq}(Cb(p))$. 

On the other hand, we would like to know the information that $\P(a|C)$ provides
at the level of definability of types for $p$. 
By Lemma \ref{types over C - equality}, to understand $Cb(p)$, it is enough to find the $p$-definitions 
for the formulas $\psi_1(x,y)=\mu(x\cap y)$ and $\psi_2(x,y)=\mu(x^c\cap y)=\mu(y) \dotminus \mu(x\cap y)$. 
Note that for any $c\in C$, we have
$$\psi_1^p(x,c)=\mu(a\cap c)=\int_c \P(a|C)\, d\mu \text{ and }$$
$$\psi_2^p(x,c)=\mu(c)-\int_c \P(a|C)\, d\mu \text{.}$$
Thus from $\P(a|C)$ we recover the $p$-definitions of the formulas 
$\psi_1(x,y)$ and $\psi_2(x,y)$, and so we recover $Cb(p)$.

(Our approach uses the fact that that for $APA$ we
have proved that all types over small algebraically closed
sets are stationary; indeed, stationarity for types over all sets follows from 
Theorem \ref{pr:non-div} and Lemma \ref{types over C - equality}.)

\begin{theo}[Prop.\ 4.5, \cite{BY1}]
\label{probCb} 
Let $C\subseteq U$ be small and let $a = (a_1,\dots,a_n) \in U^n$.
Further, let $D$ be the smallest $\s$-subalgebra of $\U$ such that
$\P(a^s | \bra C \ket )$ is $D$-measurable for all $s = (k_1,\dots,k_n) \in \{-1,+1\}^n$, 
so $D \subseteq \bra C \ket$. 
Then $D$ is a canonical base for $\tp(a/C)$.
\end{theo}
\begin{proof}
See Notation \ref{cond prob notation over a sigma-subalgebra of prob alg}
for some background that we use here, especially for
what we mean precisely by $D$-measurability of $\P(a^s | \bra C \ket )$
for $\sigma$-subalgebras $D \subseteq \bra C \ket$ of $\U$.

Let $a=(a_1,\dots,a_n) \in U^n$ and $p = \tp(a/C)$, and let $\tau$ be any automorphism of $\U$.
We must prove that $\tau(p) := \tp(\tau(a)/\tau(C))$ is parallel to $p$\index{parallelism of types} 
(that is, that they have a common non-forking extension) 
if and only if $\tau$ fixes $D$ pointwise.  

First assume that $\tau(p)$ is parallel to $p$.  Hence there is a type $q$
over $\bra C\cup \tau({C}) \ket$ such that $q$ extends $p$ and $\tau(p)$, and also that
$q$ does not fork over $\bra C \ket$, and q does not fork over $\bra \tau({C}) \ket = \tau(\bra C \ket)$.
Let $b = (b_1,\dots,b_n)\models q$.

By Lemma \ref{types over C - equality} and Theorem \ref{pr:non-div},  and the stated properties of $q$, we have
\[
\P(\tau(a^s) | \tau(\bra C \ket)) =  \P(b^s| \tau(\bra C \ket)) =\P(b^s| \bra C\cup \tau({C}) \ket) = \P(b^s| \bra C \ket) = \P(a^s| \bra C \ket)
\]
for all $s \in \{-1,+1\}^n$.  As discussed before Lemma \ref{meq-dcl of f},
$\tau(\P(a^s | \bra C \ket)) = \P(\tau(a^s) | \tau(\bra C \ket))$, 
so $\tau(\P(a^s | \bra C \ket)) = \P(a^s | \bra C \ket)$, 
by the preceding calculation.  Therefore, applying Lemma \ref{meq-dcl of f} to the representatives of 
$\P(a^s | \bra C \ket)$ for each $s \in \{-1,+1\}^n$, we conclude 
$\tau(u) = u$ for every $u \in D$, as needed to be shown.

Conversely, assume that $\tau$ fixes $D$ pointwise.  
We know $\P(a^s | \bra C \ket) = \P(a^s | D)$ for all $s \in \{-1,+1\}^n$, 
so by Theorem \ref{pr:non-div} we get $a \indep_D C$.  By Invariance for $\indep$, we also have 
$\tau(a) \indep_D \tau(C)$.  Using Extension for $\indep$ we get $p',q' \in S_n(C \cup \tau(C))$ such that $p'$ is
a non-forking extension of $p = \tp(a/C)$ and $q'$  is a non-forking extension of $\tau(p) = \tp(\tau(a)/\tau(C))$.
By Transitivity for $\indep$, it follows that both $p'$ and $q'$ are non-forking over $D$, so by Stationarity
for $\indep$ and the fact that $\tp(\tau(a)/D) = \tp(a/D)$, we conclude that $p'=q'$.  
It follows that $p$ and $\tau(p)$ are parallel.
\end{proof}

In \cite{BY1}, the perspective on canonical bases is the same as the one we
use here.  The proof of Prop.\ 4.5 in \cite{BY1} shows that if $E$ is
any closed algebra $\subseteq U$ over which the type does not fork, and it is minimal 
with this property, then $E$ coincides with $D$. Another approach 
to canonical bases can be found in \cite{BY3}, where Ben Yaacov shows that 
a better way of dealing with these objects is by introducing a sort for $[0,1]$-valued 
random variables associated to the corresponding probability space. It turns out 
that one can identify the canonical base of $\tp(a/C)$ with $\P(a|\C )$ (a $[0,1]$-valued random variable) 
in a \emph{uniform} way in order to construct \emph{uniform canonical bases} 
(see \cite[Definition 1.1 and Corollary 2.3]{BY3}) a process which requires
imaginaries for $APA$ (see \cite[Corollary 2.5]{BY3}).  

The fact that types have canonical bases in the home sort gives some
information about elimination of imaginaries for $APA$.  
Namely, $APA$ has \emph{weak} elimination of metric imaginaries,\index{weak elimination of metric imaginaries} 
which means that for every element $a$ of an imaginary sort, there exists a subset
$A$ of the home sort such that $\acl^{\meq}(a) = \dcl^{\meq}(A)$.
(See \cite[Defn.\ 1.5]{BY5}.)

\begin{coro}
\label{APA has wEMI}
The theory $APA$ has weak elimination of metric imaginaries.
\end{coro}
\begin{proof}
This follows from the fact that $APA$ is stable and that it has canonical 
bases in the home sort (Theorem \ref{probCb}), together with \cite[Fact 1.6]{BY5}.
\end{proof}

\bigskip
We turn now to another property that is related to canonical bases, namely being 
\emph{strongly finitely based (SFB)}. 
Let $C\subseteq U$ be a small set, 
$a=(a_1,\dots,a_n)\in U^n$, $x=(x_1,\dots,x_n)$, and $p(x)=\tp(a/C)$.
The type $p(x)$ is stationary.  For $\varphi(x;y)$ an $\LLpr$-formula, with 
$y=(y_1,\dots,y_k)$, let ${d}_{x}^p\varphi(y)$ be a $\varphi$-definition for 
\index{${d}_{x}^p\varphi(y)$}
$p(x)$, which exists because $APA$ is stable.  Its main property is that
for every $c\in C^k$, the $p(x)$-value of the formula $\vphi(x;c)$ 
equals the value of ${d}_{x}^p \vphi(c)$ in $\U^{\meq}$.
The $\varphi$-definition for $p(x)$ can be constructed from a Morley sequence 
$(a^i)_{i\in \N}$ in any type parallel to $p(x)$ over $C$, by 
defining it as the average value of $\varphi(x;c)$ along the Morley sequence:
$${d}_{x}^p\varphi(c)=\lim_{k\to \infty} \frac{\sum_{i=1}^k \varphi(a^i;c)}{k}$$
This definition depends only on the parallelism class of $p$ and not on the specific Morley sequence under consideration.  
So ${d}_{x}^p \varphi (y)$ is an $(\LLpr)^{\meq}$-formula in which a parameter from $\dcl^{\meq}(C)$ occurs.

Now consider $p,q \in S_n(C)$, with $\varphi$-definitions ${d}_{x}^p\varphi(y)$, ${d}_{x}^{q}\varphi(y)$ respectively. 
Another way to measure how much $p,q$ differ is by considering the pseudometrics%
\index{${d}_\varphi(p, q)$}
$${d}_\varphi(p, q) :=\sup_{c\in C^k} |{d}_{x}^p\varphi(c)-{d}_{x}^{q}\varphi(c)|$$ as $\varphi(x;y)$ 
ranges over all  $\LLpr$-formulas in variables $(x;y)$, with $x$ fixed and $y=(y_1,\dots,y_k)$ any finite sequence of parameter variables.  These pseudometrics define a uniform structure on $S_n(C)$, which we denote by $\V_{Cb}$.
The topology induced by $\V_{Cb}$ on $S_n(C)$ is denoted by $\tau_{Cb}$.%
\index{$\V_{Cb}$, $\tau_{Cb}$}
\footnote{The reason for including these details is to justify that our topology $\tau_{Cb}$
is the same as the topology introduced in \cite{BBH}.  See Remark \ref{alternate expression for d phi} below for a more
elementary formula for ${d}_\vphi$.}

With the topology just described, we have three natural topologies 
on $S_n(C)$, namely, the logic topology $\tau_{\mathcal{L}}$, the metric topology $\tau_{d}$
and now the \emph{canonical base topology} $\tau_{Cb}$.\index{canonical base topology}

\begin{rema}
\label{alternate expression for d phi}
A more transparent way of evaluating ${d}_\vphi(p,q)$ comes from simply using what the $p$- and $q$-definitions
express.  Suppose $a \models p$ and $b \models q$,  Then
$${d}_\varphi(p, q) :=\sup_{c\in C^k} |\vphi^\U(a;c) - \vphi^\U(b;c)| \text{.}$$
\end{rema}

In the setting of $APA$, since types over the set $C$ are stationary, we have 
$\tau_{d}\subseteq \tau_{Cb}\subseteq \tau_{\mathcal{L}}$ (see \cite[Lemma 1.5]{BBH}). 
When $C$ is finite, adding names for the elements of $C$ preserves $\aleph_0$-categoricity over $APA$, and thus 
the three topologies are identical. On the other hand, when we take a sufficiently large set of parameters, 
for example if $\bra C \ket$ is atomless and thus the universe of a model of $APA$, 
then $\tau_{d} \subsetneq  \tau_{\mathcal{L}}$
(by the continuous Ryll-Nardzewski Theorem).  It is natural to ask for more precise information
for how the topology $\tau_{Cb}$ 
relates in general to the other two topologies.

\begin{defi}[\cite{BBH}]
\label{def: SFB}
A stable theory $T$ is \emph{strongly finitely based (SFB)} if for every  $\M\models T$ and every
$n$, the topologies $\tau_{d}$ and $\tau_{Cb}$ agree on $S_n(M)$.\index{strongly finitely based (SFB)}
\end{defi}

Applications of the $SFB$ property can be found in \cite{BBH}, where the concept was 
introduced as a continuous analogue of a strong version of the notion of being $1$-based.

In that paper, the theory of lovely pairs was used to prove that $APA$ has $SFB$. 
Below we give a direct proof using Theorem \ref{distance} and Lemma \ref{Lipschitz estimates}. 

\begin{defi}
 \label{defn of {Cb}-distance}
 Let $C\subseteq U$ be a small subalgebra. Let $n\geq 1$ and 
$x = (x_1,\dots,x_n)$ a tuple of distinct variables, and $y$ a variable not occurring in $x$.  
Let $\varphi_s(x;y) := \mu(x^s\cap y)$ for each $s \in \{-1,1\}^n$. 
Then we define ${d}_{Cb}$ on $S_n(C)$ by ${d}_{Cb}(p,q) := \max_s {d}_{\varphi_s} (p,q)$.
\end{defi}
\index{${d}_{Cb}$}

Obviously ${d}_{Cb}$ is a pseudometric; Proposition \ref{APA is SFB} below shows that
${d}_{Cb}(p,q) =0$ implies $d(p,q)=0$, so ${d}_{Cb}$ is in fact a metric.

\begin{lema}
\label{containments}
 Let $C\subseteq U$ be a small subalgebra.  Then $\V_{Cb}$ contains
 the uniform structure induced by ${d}_{Cb}$ and is contained in the uniform structure
 induced by $d$. 
 \end{lema}
\begin{proof}
As $\{\varphi_s\}_{s \in \{-1,+1\}^n}$ are $\LLpr$-formulas, the set
$\{ (p,q) \in S_n(C)^2 \mid {d}_{Cb}(p,q) \leq \epsilon \}$
is in $\V_{Cb}$ for every $\epsilon >0$. 

On the other hand, any $\LLpr$-formula 
$\vphi(x;y_1,\dots,y_k)$ with parameters from $C$ is uniformly continuous with respect
to $d$ (see \cite[Theorem 3.5]{BBHU}).  Therefore, all formulas $\vphi(x;c_1,\dots,c_k)$ have the same
modulus of uniform continuity with respect to $d$.  By \cite[Proposition 2.8]{BBHU} it follows that ${d}_\vphi$ is
uniformly continuous with respect to $d$, with the same modulus.
\end{proof}

\begin{prop}
\label{APA is SFB}
The theory $APA$ is SFB; that is, the topologies $\tau_{Cb}$ and $\tau_{d}$ coincide over any set $C$ of parameters.
Indeed, the metrics ${d}_{Cb}$ and $d$ both induce the uniform structure $\V_{Cb}$ on $S_n(C)$.
\end{prop}
\begin{proof}
Let $C\subseteq U$ be small.  Together with Lemma \ref{containments}, it suffices to show that
$d$ is uniformly continuous with respect to ${d}_{Cb}$ on $S_n(C)$.  Fix $p,q \in S_n(C)$ and let
$a,b \in U^n$ satisfy $a \models p$ and $b \models q$.

Recall from Theorem \ref{distance}
and Lemma \ref{Lipschitz estimates} that
$$ \big( 2^{-n+1}\big) \cdot d(p,q) \leq d(\Pi_n(p),\Pi_n(q)) = \max_s \|\P(a^s|\bra C \ket)-\P(b^s |\bra C \ket)\|_1  \text{.}$$

Fix $s \in \{ -1,+1 \}^n$ and $\epsilon \in (0,1]$.  Let $k$ satisfy $k-1 \leq 1/\epsilon < k$, so $1/k < \epsilon$ and $k \leq \frac{1}{\epsilon}+1$.

Let $u_1,\dots,u_k$ be a partition of $1$ in $\bra C \ket$ obtained using Lemma \ref{approximating P(a|C)}
applied to $a^s$ and $k$.  This ensures that for any closed subalgebra $E$ of $\bra C \ket$ that contains $\{ u_1,\dots,u_k \}$
we have $\|\P(a^s|\bra C \ket)- \P(a^s |E) \|_1\leq 1/k$.

Similarly, let $v_1,\dots,v_k$ be a partition of $1$ in $\bra C \ket$ such that 
for any closed subalgebra $E$ of $\bra C \ket$ that contains $\{ v_1,\dots,v_k \}$
we have $\|\P(b^s|\bra C \ket)- \P(b^s |E) \|_1\leq 1/k$.

Then $E := \{u_1,\dots,u_k,v_1,\dots,v_k\}^\#$ is a finite subalgebra of $\bra C \ket$ with at most $k^2$ atoms (namely all the intersections $u_i \cap v_j$) such that both 
$\|\P(a^s|\bra C \ket)- \P(a^s |E) \|_1\leq 1/k$ and $\|\P(b^s|\bra C \ket) - \P(a^s |E)  \|_1\leq 1/k$.
Let the atoms of $E$ be $e_1,\dots,e_N$, so $N \leq k^2$.

Using the triangle inequality for $ \| \cdot \|_1$ we get
$$ \|\P(a^s|\bra C \ket)-\P(b^s |\bra C \ket)\|_1 \leq \|\P(a^s|E)-\P(b^s |E)\|_1 +2\epsilon \text{.}$$
The proof is completed by the following estimate:
\begin{align*}
\|\P(a^s|E)-\P(b^s |E)\|_1  = 
\big\| \sum_{j=1}^N & \frac{\mu(a^s \cap e_j)}{\mu(e_j)} \chi_{e_j} - \sum_{j=1}^N \frac{\mu(b^s \cap e_j)}{\mu(e_j)} \chi_{e_j} \big\|_1 \\
 \overset{\star}{=} \sum_{j=1}^N \|  \frac{\mu(a^s \cap e_j) - \mu(b^s \cap e_j)}{\mu(e_j)} &\chi_{e_j}\|_1
 = \sum_{j=1}^N | \mu(a^s \cap e_j) - \mu(b^s \cap e_j) | \\
= N \cdot {d}_{\vphi_s}(p,q) & \leq (\tfrac{1}{\epsilon} + 1)^2 \cdot {d}_{\vphi_s}(p,q) \text{.}
\end{align*}
(The equality ($\star$) holds because the different $\chi_{e_j}$ are disjointly supported.)
Therefore 
$$(2^{-n+1}) \cdot d(p,q) \leq  (\frac{1}{\epsilon} + 1)^2 \cdot {d}_{Cb}(p,q) + 2\epsilon \text{,}$$ 
so ${d}_{Cb}(p,q) < \epsilon^3$ implies $d(p,q) <  (2^{n+2}) \cdot \epsilon$, 
showing that $d$ is uniformly continuous relative to ${d}_{Cb}$.

When combined with Lemma \ref{containments}, this argument completes the 
proof that $d$ and ${d}_{Cb}$ induce the same uniform structure 
on $S_n(C)$ (namely $\V_{Cb}$).
It follows that the topologies $\tau_{Cb}$ and $\tau_{d}$ coincide, over any set $C$ of parameters.
\end{proof}

\bigskip
In the rest of this section we make a few connections with Shelah's classification program
for models of classical first order theories, and we offer some speculative suggestions about
how some aspects of the program might be carried into continuous model theory.

Shelah describes in \cite{Shelah} what one would hope for from a
\emph{structure theorem} versus \emph{non-structure theorem} that
\emph{allows one to} versus
\emph{prevents one from} completely classifying the models of a complete theory $T$.
This distinction is expressed in terms of invariants that determine the models of $T$ up to
isomorphism.  The invariants that come into the picture at the lowest level of complexity
(ordered by \emph{depth}, which is in general any ordinal number), are defined as follows.
(When we assign an invariant to a model $\M$, it should only depend on the isomorphism
type of $\M$.)  An \emph{invariant of depth $0$} for models of $T$ is an
assignment to each $M$ of a cardinal $\leq \lambda = \card(M)$.
An \emph{invariant of depth $1$} for models of $T$ is an assignment to
each $\M$ of a set $\K$ of cardinals $\leq \lambda = \card(M)$ together with a
family of $\leq 2^{\aleph_0}$ many
functions from the set $\K$ to the set of cardinals $\leq \lambda$. 
(Frequently one ignores models of cardinality $< \kappa$, for some infinite $\kappa$.)  
For example, if $T$ is uncountably categorical
and its language is countable, the models of $T$ have invariants of depth $0$, namely to the
uncountable model $\M$ is assigned $\card(M)$.
Shelah's thesis is that $T$ has a structure theory iff there is an
ordinal $\alpha$ and invariants (or sets
of invariants) of depth $\alpha$ that determine every model of $T$ up to
isomorphism.

The Maharam invariants for models of $APA$ fit into this framework, with
a twist that is not surprising, given that the
setting has changed from classical model theory to its continuous
counterpart.  Namely, invariants coming from the
interval $[0,1]$ come into the picture.  Consider $\M \models APA$ of
density $\lambda$.  The invariant $\K^\M$ is
a nonempty countable set of cardinal numbers that satisfies $\sup \K^\M
= \lambda$; the additional invariant
$\Phi^\M$ is a function from $\K^\M$ to $(0,1]$ such that $\sum \{
\Phi(\kappa) \mid \kappa \in \K^\M \}=1$.  This
feels analogous to Shelah's invariants of depth $1$, with the function
$\Phi^\M$ as an additional feature.

Recall that for every model $\M$ of $APA$ there exists a family
$(b_\kappa \mid \kappa \in \K^\M)$ that witnesses
the invariants $(\K^\M,\Phi^\M)$ in the sense that each $b_\kappa$ is
maximal homogeneous of density $\kappa$
and $\mu(b_\kappa) = \Phi^\M(\kappa)$ for each $\kappa$.  (This implies
that the elements $b_\kappa$ are
pairwise disjoint and their union is $1$.)  (See Section \ref{maharam}.)

Note that exactly as in Shelah's framework, the Maharam invariants can
be used to calculate $I(\lambda,APA)$, which here is
defined to be the number of models (up to isomorphism) of $APA$ having density character
$\lambda$.  We know $I(\aleph_0,APA)=1$.
If $1 \leq n < \omega$ and we are considering models of density
$\aleph_n$, there are $2^n$ many choices for $\K^\M$ (it must
contain $\aleph_n$) and $2^{\aleph_0}$ many choices of $\Phi^\M$ on each
choice of $\K^\M$,
except for the case $\K_\M = \{ \aleph_n \}$, where $\Phi(\aleph_n) = 1$
is required.  (This last choice is the invariant of the unique
homogeneous model of density $\aleph_n$.) Hence $I(\aleph_n,APA) =
2^{\aleph_0}$.   For an ordinal $\gamma \geq \omega$, a similar
calculation shows that $I(\aleph_\gamma,APA) = (\card(\gamma))^{\aleph_0}$.

This analogy makes it seem likely that $APA$ can be placed somewhere in
a Shelah-style classification framework for
$\omega$-stable continuous theories with a countable language.  To begin exploring
this possibility, we introduce possible definitions of notions like
\emph{unidimensional} and \emph{non-multidimensional} into the
continuous logic setting, and explore the extent to which they apply to
$APA$.  Mostly they are taken directly from the classical
first order discrete case, as presented in \cite{Bue}.  

\begin{nota}
\label{notation for non-forking extension}
Let $T$ be a stable theory and let $\V\models T$ be a $\kappa$-universal domain.
Let $B \subseteq A\subset V$ be small and let $p\in S_n(B)$. 
When $p$ is stationary, we let $p \lharp\! A$ denote the unique type in $S_n(A)$
that is a non-forking extension of $p$.
\end{nota}
\index{$p \lharp{A}$}

\begin{defi}
\label{def: orthogonal unidimensional, non-multidimensional}
Let $T$ be a stable theory and let $\V\models T$ be a $\kappa$-universal domain.
Let $A\subset V$ be small and let $p,q\in S(A) := \bigcup_n S_n(A)$. We say $p,q$ are 
\emph{almost orthogonal}\index{almost orthogonal}
 and write $p \perp^a q$ if for all $b\models p$ and all
$c\models q$ we have $b\indep_A c$. Given stationary types $p \in S(B),q \in S(C)$, 
with $B,C \subseteq V$ small, 
we say that $p$ and $q$ are \emph{orthogonal}\index{orthogonal}
and write $p\perp q$ if for all small sets $A$ with  $B \cup C \subseteq A \subseteq V$, 
we have $(p\lharp\! A) \perp^a (q\lharp\! A)$.

We say a theory $T$ is \emph{unidimensional}\index{unidimensional}
 if $p\not \perp q$ whenever $p,q$ are non-algebraic stationary types
over small sets of parameters $\subseteq V$.

Whenever $p$ is a stationary type, we write $p\perp \emptyset$ if $p\perp q$ for all $q\in S(\acl^{\meq}(\emptyset))$. 

For superstable $T$, we say $T$ is \emph{bounded} or 
\emph{non-multidimensional}\index{non-multidimensional}
 if every 
non-algebraic stationary type $p$ satisfies $p \not \perp \emptyset$.
\end{defi}

Now we return to $APA$. For types over $\emptyset$ the picture is very simple:

\begin{prop}
If $p \in S_m(\emptyset)$ and $q \in S_n(\emptyset)$ are non-algebraic, 
then $p$ and $q$ are not almost orthogonal.
\end{prop}
\begin{proof}
It suffices to consider types of partitions of $1$, since any $n$-tuple is interdefinable with its
associated partition of $1$.

Consider two partitions of $1$, say $a = (a_1,\dots,a_m)$ and $b = (b_1,\dots,b_n)$ in $\U$, and suppose
that neither type $\tp(a),\tp(b)$ is algebraic (here this means that at least one $a_i$ and at least
one $b_j$ are distinct from $0$ and $1$).  

Let $\M$ be the probability algebra of the standard Lebesgue space $([0,1],\mathcal{B},\mu)$,
in which all $n$-types over $\emptyset$ can be realized.  
Realize $\tp(a)$ in $\M$ by a sequence of pairwise disjoint intervals $(I_1,\dots,I_m)$, where each interval
is of the form $[r,s)$ for $r<s$ in $[0,1]$ and $\sup I_k = \min I_{k+1}$ for every $k=1,\dots,m-1$.  
Then the union of the intervals $I_i$ is $[0,1)$.  Let $(J_1,\dots,J_n)$ be a similar sequence of intervals
that realizes $\tp(b)$.  

Choose $i,j$ least so that $0<\mu(a_i)<1$ and $0<\mu(b_j)<1$.
By the choice of $i,j$ we have $I_i = [0,\mu(a_i))$ and $J_j=[0,\mu(b_j))$.

Then $\mu(I_i \cap J_j) = \min(\mu(a_i),\mu(b_j)) > \mu(a_i)\mu(b_j) = \mu(I_i)\mu(J_j)$. 
Therefore we have non-independent realizations of the types $\tp(a),\tp(b)$. This shows that no pair
of non-algebraic types over $\emptyset$ is almost orthogonal.
\end{proof}

Once we allow parameters, elements may be supported over disjoint sets and we obtain more freedom:

\begin{prop}\label{non-orthogonality}
For any $e\in U$ with $e\notin \{0,1\}$ there are non-algebraic $p,q\in S_1(\{e\})$ which are orthogonal.
\end{prop}
\begin{proof}
Let $e\in U$ have $0<\mu(e)<1$ and let
$E=\{0,1,e,e^{c}\}$. Let $a,b \in U$ satisfy $a \leq e$, $\mu(a) = \mu(e)/2$, $b \leq e^c$, and $\mu(b) = \mu(e^{c})/2$. 
Consider $p=\tp(a/E)$ and $q=\tp(b/E)$. 

{Claim $1$.} $p\perp^a q$.

We work in the measure algebra $(\widehat{\mathcal{B}},\mu,d)$ 
associated to the standard Lebesgue space $([0,1],
 \mathcal{B},d)$, which is $\aleph_0$-saturated, and may assume that $(\widehat{\mathcal{B}},\mu,d)$ is an 
 elementary substructure of $\U$. Let $r=\mu(e)$. Since the type of an element is 
 determined by its measure,  we take $e$ to be (the equivalence class of) $[0,r)$ and $e^c$ to be (the equivalence class of) $[r,1]$. 
 Let $a \leq e$  have $\mu(a) = r/2$, so $a\models p$; similarly let $b\leq e^c$ have $\mu(b)=(1-r)/2$, so $b\models q$.

Note that $\P(a|E)=\frac{1}{2}\chi_e$ and $\P(a^{c}|E)=\frac{1}{2}\chi_e+\chi_{e^{c}}$.
Since $b \leq e^c$, we have $b\cap a=\emptyset$, $b^c\cap a=e\cap a$ and $\P(a|(Eb)^{\#})=\frac{1}{2}\chi_e$. 
Similarly,  $\P(a^{c}|(Eb)^{\#})=\frac{1}{2}\chi_e+\chi_{e^{c}}$.  
By Theorem \ref{pr:non-div}, we have $a \indep_E b$, and thus Claim $1$ is proved.

{Claim $2$.} $p\perp q$.

Now let a small closed subalgebra $F \subseteq U$ have $E\subseteq F$. 
Choose $a_F,b_F \in U$ with $a_F\models p\lharp\! F$, $b_F\models q\lharp\! F$; 
then as before $a_F \leq e$ and by Theorem \ref{pr:non-div}, we have 
$\P(a_F|F)=\frac{1}{2}\chi_e$, $\P(a_F^{c}|F)=\frac{1}{2}\chi_e+\chi_{e^{c}}$. 
Likewise $b_F \leq e^{c}$ and we get
$\P(a_F|(Fb_F)^\#)=\frac{1}{2}\chi_e=\P(a_F|F)$, $\P(a_F^c|(Fb_F)^\#)=\frac{1}{2}\chi_e+ \chi_{e^{c}}=\P(a_F^c|F)$.
Using Theorem \ref{pr:non-div} we conclude $a_F \indep_F b_F$, as desired.
\end{proof}

From the previous proposition we get:

\begin{coro}
The theory $APA$ is not unidimensional.
\end{coro}

Maharam's Theorem (\ref{Maharam's Theorem for models of Pr}) provides a countable 
set of cardinals $\K^\M$ that helps classify a given
model of $APA$ up to isomorphism. In the classical first order setting, the existence of two or more distinct 
classifying cardinals is related to the existence of orthogonal types. In the example that follows, 
we illustrate this phenomenon in the setting of continuous model theory, for $APA$.

\begin{exam}
Let $(\mathcal{B},\mu,d)\models APA$ and assume $b_1,b_2\in \B$ are homogeneous
elements of different density.  Let $C$ be the algebra generated by $\{b_1, b_2\}$.
Let $a_1,a_2\in \mathcal{B}$ satisfy $0 < a_i < b_i$ for $i = 1,2$. Then $\tp(a_1/C)\perp \tp(a_2/C)$.\\
More generally, assume that $b_1,b_2\in \B$ are disjoint
elements and $C$ is an algebra containing these elements. Also assume we are given $a_1,a_2\in \B$
nonalgebraic over $C$ with $a_i < b_i$ for $i=1,2$. Then $\tp(a_1/C)\perp \tp(a_2/C)$.
\end{exam}

\begin{proof}
We start with the first statement. Since $b_1,b_2$ are homogeneous elements of 
different density we have $b_1\cap b_2=\emptyset$
and thus $C=\{0,1,b_1,b_2,(b_1\cup b_2)^c \}$. We first show $a_1\indep_C a_2$.

We see  $\P(a_1|C)=\frac{\mu(a_1)}{\mu(b_1)}\chi_{b_1}$ and
$\P(a_1^{c}|C)=\frac{\mu(b_1)-\mu(a_1)}{\mu(b_1)}\chi_{b_1}+\chi_{b_1^c}$.
On the other hand we have $a_2\leq b_2$ and $b_2$ is disjoint from $b_1$, so
$\P(a_1|(Ca_2)^\#)=\frac{\mu(a_1)}{\mu(b_1)}\chi_{b_1}=\P(a_1|Ca_2)$ and
$\P(a_1^{c}|(Cb_2)^\#)=\frac{\mu(b_1)-\mu(a_1)}{\mu(b_1)}\chi_{b_1}+\chi_{b_1^c}=\P(a_1^{c}|C)$.  
Therefore we have $a_1\indep_C a_2$.

{Claim.} $\tp(a_1/C)\perp \tp(a_2/C)$.

Let $F \subseteq U$ be a small closed subalgebra with $C\subseteq F$ and let
$p\lharp\! F$ and $q\lharp\! F$ be the non-forking extensions
to $F$ of $p=\tp(a_1/C)$ and $q=\tp(a_2/C)$ respectively. Choose $a_{1F},a_{2F}\in U$ 
with $a_{1F}\models p\lharp\! F$, $a_{2F}\models q\lharp\! F$. By Theorem \ref{pr:non-div}, we have 
$\P(a_{1F}|F)=\frac{\mu(a_1)}{\mu(b_1)}\chi_{b_1}$ and 
$\P(a_{1F}^{c}|C)=\frac{\mu(b_1)-\mu(a_1)}{\mu(b_1)}\chi_{b_1}+\chi_{b_1^c}$. 
Since $a_{2F}\leq b_2$ and $b_1$ and $b_2$ are disjoint, we get
$\P(a_{1F}|(Fa_{2F})^\#)=\frac{\mu(a_1)}{\mu(b_1)}\chi_{b_1}=\P(a_{1F}|F)$ and
 $\P(a_{1F}^c|(Fa_{2F})^\#)=\frac{\mu(b_1)-\mu(a_1)}{\mu(b_1)}\chi_{b_1}+\chi_{b_1^c}=\P(a_{1F}^c|F)$.

Using Theorem \ref{pr:non-div} we conclude $a_{1F} \indep_F a_{2F}$, as desired. 

The more general statement has a similar proof and we leave the details to the reader.
\end{proof}

We need the following easy result.

\begin{obse}\label{boundscondprob}
Let $0<\delta<\frac{1}{2}$ and let $\delta\leq r\leq 1-\delta$. Then $\min\{\frac{1}{2},r\}-\frac{1}{2}r \geq \frac{\delta}{2}$.
\end{obse}

\begin{proof}
Assume first that $\min\{\frac{1}{2},r\}=\frac{1}{2}$. Then $\min\{\frac{1}{2},r\}-\frac{1}{2}r =\frac{1}{2}(1-r)\geq  \frac{\delta}{2}$. On the other hand, if $\min\{\frac{1}{2},r\}=r$, then $\min\{\frac{1}{2},r\}-\frac{1}{2}r =\frac{1}{2}r\geq  \frac{\delta}{2}$.
\end{proof}

\begin{prop}\label{nonmultidimensional}
The theory $APA$ is nonmultidimensional.
\end{prop}

\begin{proof}
Let $a\in U$ with $\mu(a)=1/2$ and let $q=\tp(a/\emptyset)$. We will show any non-algebraic type over any set is non-orthogonal to $q$. Let $D\subset U$ be a small closed subalgebra of $\U$ and let $b=(b_1,\dots,b_k)$ be a partition of $1$ such that $p=\tp(b/D)$ is not algebraic. We may assume without loss of generality that $b_1\not \in D$ and we may work with $p_1=\tp(b_1/D)$ instead of $p$. Below we will show that $p_1\not \perp^a q \lharp \! D$.

Since $b_1\not \in D$, there is $u\in D$ of positive measure and  $\delta>0$ such that 
\begin{equation*}\tag{*}
 \delta\leq \P(b_1|D)(x) \leq 1-\delta \mbox{ for } \mu \mbox{-almost every } x \in u \text{.}
\end{equation*}

We may assume $\U$ is the probability algebra associated to a probability space $(X,\B,\mu)$. Let $([0,1],\mathcal{C},\lambda)$ be a standard atomless Lebesgue space and work in the probability algebra of the space 
$(X\times [0,1],\B \otimes \mathcal{C},\mu \otimes \lambda)$, identifying each $v\in D$ with $v'=v\times [0,1]$ as done in the proof of Lemma \ref{extension}.

Let $b_1'=\{(x, y)\in X\times [0, 1] : 0 \leq y \leq \P(b_1|D)(x)\}$. Just as in the proof of Lemma \ref{extension}, we have that $b_1'\models p_1$. Also let $a'=X\times [0,1/2]=\{(x, y)\in X\times [0, 1] : 0 \leq y \leq \frac{1}{2}\}$, which is a realization of $q\lharp \! D$ since it has measure $1/2$ and is independent from all elements of $U$. 

We will prove $b_{1}'\nindep_D a'$ using Theorem \ref{pr:non-div} and 
Definition \ref{conditional independence}.

Note that $\P(b_1'|D)\P(a'|D)=\frac{1}{2}\P(b_1'|D)$.
On the other hand, 
\[
b_1'\cap a'=\{(x, y)\in X\times [0, 1] : 0 \leq y \leq \min( \P(b_1|D)(x),\frac{1}{2}) \} \text{.}
\]
For almost every $x\in u$, we get
\begin{center}
$\P(b_1'\cap a'|D)(x)-\frac{1}{2}\P(b_1'|D)(x)\geq \delta/2$
\end{center}
using (*) and Observation \ref{boundscondprob}.

Since $u$ has positive measure, we get 
\[
\int \big|\P(b_1'\cap a'|D)-\P(a'|D)\P(b_1'|D) \big| \,d (\mu \otimes \lambda)\geq \mu(u)\delta/2>0
\] 
and thus $b_{1}'\nindep_D a'$ as desired.
\end{proof}

\section{Ranks obtained from entropy}
\label{entropy}

In this section we discuss the definition and main properties of entropy, 
following \cite[Chapter 4]{Wal} and \cite[Chapter 4]{Br}, to bring out 
the connection with model theoretic aspects of $APA$.  The results given
in Fact \ref{properties0} and Corollary \ref{quantitative entropy-forking}
show how entropy provides a rank that is closely connected to model
theoretic forking.

Let $(X,\B,\mu)$ be an atomless probability space.

\begin{defi}
Let $\A$ be a finite subalgebra of $\B$ with atoms
$\{A_1,\dots,A_k\}$. Let $\C$ be a $\sigma$-subalgebra of $\B$. Then
the \emph{entropy of $\A$ given $\C$}\index{entropy}
 is $$H(\A/\C)=-\int \Sum_{1 \leq i \leq k}\P(A_i|\C)\ln(\P(A_i|\C))d\mu$$
\end{defi}

We write $H(\A)$ for $H(\A/\{\emp,X\})$. If $\A$ and $\C$ are
$\s$-algebras, we denote by $\A \vee \C$ the $\s$-algebra
generated by $\A$ and $\C$.

\begin{defi}
A continuous real-valued function $F$ with domain $[a,b]$ is \emph{convex}
if $$F(tx_1 + (1-t)x_2)\leq t F(x_1) + (1-t)F(x_2)$$
for all choices of $x_1,x_2$ in $[a,b]$ and all $t \in [0,1]$.
\end{defi}

Note that if $F \colon [a,b]\to \R$ is continuous and is twice differentiable on
 $(a,b)$, and if $F''(x)>0$ for all $x$ in $(a,b)$, then $F$ is convex.

\begin{fact}[\cite{Br} Proposition 4.4]\label{convex}
Let $\EE$ be a $\sigma$-subalgebra of $\B$. Let $F \colon [a,b]\to
\R$ be a continuous convex function, where $0\leq a\leq b<\infty$.
Then $$F(\E_{\EE}(f))\leq \E_{\EE}(F(f))$$ for each $f\in L_1(X,\B,\mu)$ with
$f(X) \subseteq [a,b]$.
\end{fact}

\begin{fact}
\label{properties0}
Let $\A$, $\C$ be finite subalgebras of $\B$ and let $\D$, $\EE$ be
$\sigma$-subalgebras of $\B$ such that $\EE \subseteq \D$. 
Let $\{a_1,a_2,\dots,a_n\}$ be the atoms in $\A$.  
Then:
\begin{enumerate}
\item $H(\A\vee\C/\EE)=H(\A/\EE)+H(\C/\A\vee \EE)$.
\item $\A\subseteq \C \implies H(\A/\EE)\leq H(\C/\EE)$.
\item $H(\A/\EE)\geq H(\A/\D)$.
\item If $\tau$ is an automorphism of $(\B,\mu.d)$, then $H(\tau(\A)/\tau(\EE))=H(\A/\EE)$.
\item $H(\A/\EE) - H(\A/\D) \geq \frac12 \sum_{j=1}^n \big( \|\E_{\D}(\chi_{a_j})\|_2^2 - \|\E_{\EE}(\chi_{a_j})\|_2^2 \big) \geq 0$.
\newline  Moreover,
$H(\A/\EE) = H(\A/\D)$ iff  $\A$ is independent from $\D$ over $\EE$.
\end{enumerate}
\end{fact}
\begin{proof}
The first four properties are proved in \cite[Section 4.3]{Wal}.  Note that $(5)$ implies $(3)$, 
and $(4)$ is obvious from the definition.  Throughout the argument, we let $a$ be any element of $U$; 
we will apply the results for $a$ ranging over the set of atoms $\{ a_1,\dots,a_n\}$.

Fact \ref{droponnorm} shows that
\begin{equation*}
\tag{A} \|\E_{\D}(\chi_{a})\|_2^2 - \|\E_{\EE}(\chi_{a})\|_2^2 = \| \E_{\D}(\chi_{a}) - \E_{\EE}(\chi_{a})\|_2^2 \geq 0 \text{.}
\end{equation*}
Applying this for $a \in \{a_1,\dots,a_n\}$ proves the second inequality in (5).  

Next we prove the first inequality in $(5)$.  

Consider $F(x)=2x\ln(x)-x^2$ restricted to $[0,1]$.
We have $F''(x)=2/x-2>0$ for $x\in (0,1)$. Applying Fact \ref{convex} for
this $F$ and $f=\E_{\D}(\chi_a)$ we get
\begin{equation*}
2\E_\EE(\chi_a) \ln(\E_\EE(\chi_a)) -\E_\EE(\chi_a)^2\leq
2\E_\EE\big(\E_{\D}(\chi_a) \ln(\E_{\D}(\chi_a))\big)-\E_{\EE}(\E_{\D}(\chi_a)^2) \text{.}
\end{equation*}

Integrating and moving the terms in the preceding inequality yields
\begin{equation*}
\tag{B}   
\begin{aligned}
\Big(- \int \P(a|\EE)\ln(\P(a|\EE)) d\mu \Big) & - \Big(- \int \P(a|\D)\ln(\P(a|\D)) d\mu \Big) \\
\geq \frac12 \big( \|\E_{\D}(\chi_{a})\|_2^2 & - \|\E_{\EE}(\chi_{a})\|_2^2 \big) \text{.}
\end{aligned}
\end{equation*}

Summing the terms in (B) over  $a \in \{a_1,\dots,a_n\}$ yields the first inequality in (5).

Finally, we prove the ``Moreover'' statement.    
We know $\A$ is independent from $\D$ over $\EE$
iff $\E_{\D}(\chi_{a_j}) = \E_{\EE}(\chi_{a_j})$ for all $j=1\dots,n$ (by Theorem \ref{pr:non-div}),
and the latter implies $H(\A/\EE) = H(\A/\D)$, by definition of entropy.  On the other hand,
by the inequalities in $(5)$, we see $H(\A/\EE) = H(\A/\D)$ implies
$\|\E_{\D}(\chi_{a_j})\|_2^2 = \|\E_{\EE}(\chi_{a_j})\|_2^2$ for all $j=1,\dots,n$, which in turn
implies $\E_{\D}(\chi_{a_j}) = \E_{\EE}(\chi_{a_j})$ for all $j=1\dots,n$ by statement (A)
applied to $a = a_1,\dots,a_n$.
\end{proof}

Fact \ref{properties0}(5) provides a connection between
forking and  change of entropy. It has a quantitative aspect that
we record here.  Recall that for
$a = (a_1,\dots,a_n)$ from $\widehat{\B}$ and $D\subseteq C\subseteq\widehat \B$, we say $\tp(a/C)$
\emph{$\epsilon$-forks over} $D$ 
if $d(\tp(a/C),\tp(a/D)\lharp\! {C})>\epsilon$
where $\tp(a/D)\lharp\!{C}$ is the unique non-forking extension
of $\tp(a/D)$ to $C$.  (See Remark \ref{superstable}.)

\begin{coro}
\label{quantitative entropy-forking}
Let $\A$ be a finite subalgebra of $\B$ and let $\EE$, $\D$ be
$\sigma$-subalgebras of $\B$ such that $\EE \subseteq\D$. Let $\{a_1,a_2,\dots,a_n\}$ be
the events corresponding to the atoms in $\A$, $D$ the set of events associated to $\D$ and
$E$ the set of events associated to $\EE$. If $\tp((a_1,\dots,a_n)/D)$ $\epsilon$-forks over $E$, then
$H(\A/\EE) > H(\A/\D) + \epsilon^2/2$.  
\end{coro}
\begin{proof}
Assume that $\tp(a_1,\dots,a_n/D)$ $\epsilon$-forks over $E$. Then by Theorems
\ref{distance} and \ref{pr:non-div}, for some $j=1,\dots,n$ we have
$\|\E_{\D}(a_j)-\E_{\EE}(a_j)\|_1>\epsilon$.  
By Fact \ref{droponnorm} this implies
$\|\E_{\D}(a_j)-\E_{\EE}(a_j)\|_2>\epsilon$ and thus
$\|\E_{\D}(a_j)-\E_{\EE}(a_j)\|_2^2>\epsilon^2$.
Then we get $H(\A/\EE) > H(\A/\D) + \epsilon^2/2$ by the inequality in Fact \ref{properties0}(5).
\end{proof}

\section{Some problems}
\label{problems}

In this final section we briefly indicate a few problems that seem interesting and worth investigation.

(P1)  Give an explicit formula for the induced distance  between types in $S_n(C)$ for $APA$.  
$$d(p,q)=\inf\{\max_{1 \leq i \leq n}
d(a_i,b_i): (a_1,\dots,a_n)\models p, (b_1,\dots,b_n)\models q\} \text{.}$$

(P2) Provide a thorough analysis of the imaginary sorts for $APA$.  

(P3)  Complete the model theoretic background behind a generalization to continuous model theory
of Shelah's classification theory for superstable theories.  (Some first steps for this as applied to $APA$
were discussed at the end of Section \ref{stability of APA}.). In particular, study appropriate versions of properties
such as $DOP$ (dimensional order property) and $OTOP$ (omitting types order property) in the 
continuous setting and prove a dichotomy theorem relating a small bound for $I(\lambda,T)$ to when $T$ is superstable 
and has neither $DOP$ nor $OTOP$, along the lines of \cite[Theorem 2.3]{Shelah}.

There is also the possibility of proving the equivalence, for continuous theories, 
between uncountable categoricity and being both $\omega$-stable and unidimensional, 
as is true for classical first order theories.

(P4)  Consider two existentially closed actions of the free group $F_k$ on the unique separable model $\M$ of $APA$,
where $2 \leq k \in \N \cup \{ \omega \}$.  Are they approximately isomorphic?  

In the joint paper \cite{BHI} of the authors with Ibarluc\'ia, the class of  existentially closed actions by a family 
of automorphisms of $\M$ is axiomatized, and some concrete examples of such actions are given 
that are approximately isomorphic but not isomorphic.
The answer is known to be positive when $k=1$. (See \cite[Remark 18.9]{BBHU}).


\begin{thebibliography}{BBHU}

\bibitem{BY1} Ita\"{\i} {Ben Yaacov}, Schroedinger's cat, \textit{Israel Journal of
Mathematics} 153, 2006, 157--191.

\bibitem{BY3} \bysame, On uniform canonical bases in $L_p$ lattices
and other metric structures, \textit{Journal of Logic and Analysis} 4, 2012, paper 12, 30pp.

\bibitem{BY4} \bysame, On theories of random variables,
\textit{Israel Journal of Mathematics} 194, 2013, 957--1012.

\bibitem{BY5} \bysame, On a Roelcke-precompact Polish group that cannot act
transitively on a complete metric space,
\textit{Israel Journal of Mathematics} 224, 2018, 105--132.

\bibitem{BY6} \bysame, Star sorts, Lelek fans, and the reconstruction of non-$\aleph_0$-categorical
theories in continuous logic, to appear in \textit{Model Theory}, 23 pages, arxiv:2203.02184v3.

\bibitem{BBH} Ita\"{\i} {Ben Yaacov}, Alexander Berenstein, and C.~Ward Henson,
Almost indiscernible sequences and convergence of canonical bases,
\textit{Journal of Symbolic Logic} 79, 2014, 460--484.

\bibitem{BBHU} Ita\"{\i} {Ben Yaacov}, Alexander Berenstein, C.~Ward Henson, and
Alexander Usvyatsov, Model theory for metric structures, in
\textit{Model Theory with Applications to Algebra and Analysis, Vol.\
II}, eds.\ Z.\ Chatzidakis, D.\ Macpherson, A.\ Pillay, and A.\
Wilkie, Lecture Notes series of the London Mathematical Society,
No.\ 350, Cambridge University Press, 2008, 315--427.

\bibitem{BU} Ita\"{\i} {Ben Yaacov} and Alexander Usvyatsov, Continuous first order
logic and local stability, \textit{Transactions of the American
Mathematical Society} 362, 2010, 5213--5259.

\bibitem{BHI} Alexander Berenstein, C.~Ward Henson, and Tom\'as Ibarluc\'ia,
Existentially closed measure-preserving actions of free groups, submitted,
arXiv:2203.10178.

\bibitem{Br}  James R.\ Brown,
\textit{Ergodic Theory and Topological Dynamics},
Academic Press, New York, 1976.

\bibitem{Bue} Steven Buechler,  \textit{Essential Stability Theory},
Springer-Verlag, Berlin, Heidelberg, 1996. 

\bibitem{Folland} Gerald B.\ Folland, \textit{Real Analysis}, 
John Wiley and Sons, 1984.

\bibitem{MT of C*-algebras}
Ilijas Farah, Bradd Hart, Martino Lupini, Leonel Robert, Aaron Tikuisis, 
Alessandro Vignati, and Wilhelm Winter,
\textit{Model Theory of C${}^*$-algebras},
Memoirs of the AMS No.\,1324 (2021).

\bibitem{Fr-handbook} D.\ H.\ Fremlin, Measure algebras, 
in \textit{Handbook of Boolean Algebras}, vol.~3, 
North-Holland, 1989, 877--980.

\bibitem{Fr-treatise} \bysame, \textit{Measure Algebras}, vol.~3 of
\textit{Measure Theory}, Torres Fremlin, 2003--04; for information see
\texttt{http://www.essex.ac.uk/maths/staff/fremlin/mt.htm}.

\bibitem{HaMT} Paul A.\ Halmos,
\textit{Measure Theory}, Van Nostrand, 1950.

\bibitem{Ho} Wilfrid Hodges, \textit{Model Theory}, Cambridge University
Press, 1993.

\bibitem{K} Olav Kallenberg, \textit{Foundations of Modern Probability}, 3rd edition,
Springer (2021).

\bibitem{Lindstrom} Tom Lindstr{\o}m, An invitation to nonstandard analysis, in
 \textit{Nonstandard Analysis and its Applications}, ed.\ Nigel Cutland, 
 London Mathematical Society, Student Texts, No.\ 10, Cambridge University Press, 1988, 1--105. 
		
\bibitem{Loeb} Peter A.\ Loeb, Conversion from nonstandard to standard measure spaces
and applications in probability theory, \textit{Transactions of the American Mathematical Society}
211, 1975, 113--122.

\bibitem{M} Dorothy Maharam, On homogeneous measure algebras,
Proceedings of the National Academy of Sciences USA 108, 1942, 108--111.

\bibitem{Ross} David Ross, Loeb measure and probability, in
{\em Nonstandard Analysis: Theory and Practice} 
(eds.\ L.~O.~Arkeryd, C.~W.~Henson, and N.~J.~Cutland),
NATO Advanced Study Institutes Series C, Vol.~493, Kluwer
Academic Publishers, 1997, 91--120.

\bibitem{Roy} H.\ L.\ Royden, \textit{Real Analysis} third edition,
Prentice Hall, 1988.

\bibitem{Shelah} Saharon Shelah, Classification of first order structures which
have a structure theorem,
Transactions of the American Mathematical Society 12, 1985, 227--232

\bibitem{Shields} Paul Shields, \textit{The Theory of Bernoulli Shifts},
Chicago Lectures in Mathematics, The University of Chicago Press, 1973.

\bibitem{Song2} Shichang Song, Saturated structures from probability
theory, Journal of the Korean Mathematical Society, 53,2016, 315--329.

\bibitem{Wal} Peter Walters, \textit{An Introduction to
Ergodic Theory}, Springer Verlag, 1982.

\end{thebibliography}
\end{document}